\documentclass[preprint]{imsart}

\usepackage{amsthm,amsmath}
\RequirePackage[colorlinks,citecolor=blue,urlcolor=blue]{hyperref}

\startlocaldefs

\usepackage{amssymb, amscd}
\usepackage{slashed}
\usepackage{graphicx}
\usepackage[title]{appendix}
\usepackage{hyperref}
\usepackage{wrapfig}

\usepackage{dashrule}

\newcommand{\rqm}{{\declareslashed{}{\text{--}}{0.3}{0.4}{d}\slashed{d}}}

\newtheorem{thm}{Theorem}

\newtheorem{rem}{Remark}

\newtheorem{cor}{Corollary}
\newtheorem{ass}{Assumption}

\newtheorem{exam}{Example}

\newtheorem{lem2}{Lemma}[section]
\newtheorem{rem2}{Remark}[section]

\newcommand{\Cov}{\operatorname{Cov}}
\newcommand{\Std}{\operatorname{Std}}
\newcommand{\E}{\mathbb{E}}
\renewcommand{\P}{\mathbb{P}}
\renewcommand{\Re}{\operatorname{Re}}
\newcommand{\U}{\mathbb{U}}

\newcommand{\sign}{\operatorname{sign}}

\newcommand{\supp}{\operatorname{supp}}
\newcommand{\bsc}{\overline{\operatorname{sc}}}
\newcommand{\TV}{\operatorname{TV}}

\newcommand{\graph}{\operatorname{graph}}

\newcommand{\Beta}{\operatorname{Beta}}
\newcommand{\op}{\operatorname{op}}
\newcommand{\Op}{\operatorname{Op}}
\newcommand{\Os}{\operatorname{Os}}

\newcommand{\mT}{\mathcal{T}}

\newcommand{\mfp}{\mathfrak{a}^{(t,h)}}
\newcommand{\om}{\overline{m}}

\endlocaldefs

\begin{document}

\begin{frontmatter}

% "Title of the paper"
\title{Multiscale Methods for Shape Constraints in Deconvolution: Confidence Statements for Qualitative Features}
\runtitle{Confidence Statements for Qualitative Features}

% J. Schmidt-Hieber\footnote{Department of Mathematics, Vrije Universiteit Amsterdam, De Boelelaan 1081a, 1081 HV Amsterdam, Netherlands.} \footnote{For correspondence \texttt{schmidth@math.uni-goettingen.de}}, \;
% A. Munk\footnote{Institut f\"ur Mathematische Stochastik, Universit\"at G\"ottingen, Goldschmidtstr. 7, 37077 G\"ottingen and Max-Planck Institute for Biophysical Chemistry, Am Fassberg 11, D-37077 G\"ottingen, Germany.}, \; and \ L. D\"umbgen\footnote{Institut f\"ur mathematische Stochastik und Versicherungslehre, Universit\"at Bern, Alpeneggstrasse 22, CH-3012 Bern, Switzerland.}

% indicate corresponding author with \corref{}
% \author{\fnms{John} \snm{Smith}\corref{}\ead[label=e1]{smith@foo.com}\thanksref{t1}}
% \thankstext{t1}{Thanks to somebody} 
% \address{line 1\\ line 2\\ printead{e1}}
% \affiliation{Some University}

\begin{aug}
\author{\fnms{Johannes} \snm{Schmidt-Hieber}\corref{}\ead[label=e1]{j.schmidt-hieber@vu.nl}}
\author{\fnms{Axel} \snm{Munk}\ead[label=e2]{munk@math.uni-goettingen.de}}
\and
\author{\fnms{Lutz} \snm{D\"umbgen}
\ead[label=e3]{lutz.duembgen@stat.unibe.ch}}

% \thankstext{t1}{The first author acknowledges support from ...}
% \thankstext{t2}{First supporter of the project}
% \thankstext{t3}{Second supporter of the project}

\runauthor{J. Schmidt-Hieber et al.}

\affiliation{Vrije Universiteit Amsterdam\thanksmark{m1}, Universit\"at G\"ottingen\thanksmark{m2} \\ and Universit\"at Bern\thanksmark{m3}}

\address{Department of Mathematics\\ Vrije Universiteit Amsterdam\\ De Boelelaan 1081a,\\ 1081 HV Amsterdam\\ Netherlands\\
\printead{e1}}

\address{Institut f\"ur Mathematische Stochastik \\ Universit\"at G\"ottingen \\ Goldschmidtstr. 7 \\ D-37077 G\"ottingen \\ and \\ Max-Planck Institute \\ for Biophysical Chemistry\\ Am Fassberg 11\\ D-37077 G\"ottingen\\ Germany \\
\printead{e2}}

\address{Institut f\"ur mathematische Stochastik \\ und Versicherungslehre\\ Universit\"at Bern\\ Alpeneggstrasse 22\\ CH-3012 Bern\\ Switzerland\\
\printead{e3}}
\end{aug}

\begin{abstract}
We derive multiscale statistics for deconvolution in order to detect qualitative features of the unknown density. An important example covered within this framework is to test for local monotonicity on all scales simultaneously. We investigate the moderately ill-posed setting, where the Fourier transform of the error density in the deconvolution model is of polynomial decay. For multiscale testing, we consider a  calibration, motivated by the modulus of continuity of Brownian motion. We investigate the performance of our results from both the theoretical and simulation based point of view. A major consequence of our work is that the detection of qualitative features of a density in  a deconvolution problem is a doable task although the minimax rates for pointwise estimation are very slow.
\end{abstract}



\end{frontmatter}

\section{Introduction} 
We observe $Y=\left(Y_{1},\ldots,Y_{n}\right)$ according to the deconvolution model
\begin{align}
 	Y_i=X_i+\epsilon_i, \quad i=1,\ldots,n,
	\label{eq.deconmod}
\end{align}
where $X_i,  \epsilon_i, \ i=1,\ldots,n$ are assumed to be real valued and independent, $X_i\stackrel{i.i.d.}{\sim} X, \epsilon_i\stackrel{i.i.d.}{\sim} \epsilon$ and $Y_1, X, \epsilon$ have densities $g, f$ and $f_\epsilon$, respectively. Our goal is to develop multiscale test statistics for certain structural properties of $f$, where the density $f_\epsilon$ of the blurring distribution is assumed to be known.

Although estimation in deconvolution models has attracted a lot of attention during the last decades (cf. Fan \cite{Fan91}, Diggle and Hall \cite{dig}, Pensky and Vidakovic \cite{pen}, Johnstone et {\it al.} \cite{joh2004}, Butucea and Tsybakov \cite{BT07a} as well as Meister \cite{mei2} for some selective references), inference about $f$ and its qualitative features is rather less well studied. In fact, adaptive confidence bands would be desirable but turn out to be very ambitious. First, they suffer from the bad convergence rates induced by the ill-posedness of the problem  (cf. Bissantz et {\it al.} \cite{bis}), making confidence bands less attractive for applications. Second, one would need to circumvent the classical problems of honest adaptation over H\"older scales. To overcome these difficulties the aim of the paper is to derive simultaneous confidence statements for qualitative features of $f.$

Structural properties or shape constraints will be conveniently expressed as (pseudo)-differential inequalities of the density $f$, assuming for the moment that $f$ is sufficiently smooth. Important examples are $f' \gtrless 0$ to check local monotonicity properties as well as $f'' \gtrless 0$ for local convexity or concavity. To give another example, suppose that we are interested in local monotonicity properties of the density $\tilde{f}$ of $\exp(aX)$ for given $a > 0$. Since $\tilde{f}(s) = (as)^{-1} f(a^{-1}\log(s))$, one can easily verify that local monotonicity properties of $\tilde{f}$ may be expressed in terms of the inequalities $f' - a f \lessgtr 0$.

This paper deals with the moderately ill-posed case, meaning that the Fourier transform of the blurring density $f_\epsilon$ decays at polynomial rate. In fact, we work under the well-known assumption of Fan \cite{Fan91} (cf. Assumption \ref{ass.noise}), which essentially assures that the inversion operator, mapping $g\mapsto f$, is pseudo-differential. This combines nicely with the assumption on the class of shape constraints. Our framework includes many important error distributions such as Exponential, $\chi^2$, Laplace and Gamma distributed random variables. The special case $\epsilon=0$ (i.e.\ no deconvolution or direct problem) can be treated as well, of course. 

\subsection{Example: Detecting trends in deconvolution}
\label{subs.example_Laplace_deconv}

To illustrate the key ideas, suppose that we are interested in detection of regions of increase and decrease of the true density in Laplace deconvolution, that is, the error density is given by $f_\epsilon = (2\theta)^{-1}\exp(-|\cdot|/\theta)$. Let $\phi$ be a sufficiently smooth, non-negative kernel function (i.e. $\int \phi(u)du=1$), supported on $[0,1]$. Then, since $f=g-\theta^2 g''$ in this case, it follows by partial integration that
\begin{align}
  T_{t,h}:= \frac 1{h\sqrt{n}}\sum_{k=1}^n \left(\frac{\theta^2}{h^2}\phi^{(3)}\left(\frac{Y_k-t}h\right)-\phi'\left(\frac{Y_k-t}h\right)\right).
  \label{eq.def_Tth_Lapl_mono}
\end{align}
has expectation $\E T_{t,h} = \sqrt{n}\int_t^{t+h} \phi(\tfrac{s-t}h\big)f'(s) ds$. The construction of the multiscale test relies on the following analytic observation. Suppose that for a given pair $(t,h)$ there is a number $d_{t,h}$ such that
\begin{align}
    |T_{t,h}-\E T_{t,h}|\leq d_{t,h}.
    \label{eq.fnh_hat_control}
\end{align}
If in addition $T_{t,h} > d_{t,h}$, then necessarily 
  \begin{align}
    \E T_{t,h}= \sqrt{n} \int_t^{t+h} \phi\big(\tfrac{s-t}h\big) f'(s) ds>0
    \label{eq.phi_mon_ineq}
  \end{align}
and by the non-negativity of $\phi$, $f(s_1) < f(s_2)$ for some points $s_1 < s_2$ in $[t,t+h]$. On the contrary, $T_{t,h}<- d_{t,h}$ implies that there is a decrease on $[t,t+h]$. For a sequence $N_n =o(n/\log^3 n)$ tending to infinity faster than $\log^3 n$ and $u_n=1/\log\log n$, define
\begin{align*}
  B_n:=\Big\{\Big(\frac k{N_n}, \frac l{N_n} \Big) \ \big | \ k=0,1,\ldots, \ l=1,2,\ldots,[N_nu_n], \  k+l\leq N_n\Big\}. 
\end{align*}
Given $\alpha \in (0,1)$, we will be able to compute bounds $d_{t,h}$ such that for all $(t,h)\in B_n$, inequality \eqref{eq.fnh_hat_control} holds simultaneously with asymptotic probability $1-\alpha$. Taking into account that \eqref{eq.fnh_hat_control} implies \eqref{eq.phi_mon_ineq}, this allows to identify regions of increase and decrease for prescribed probability.

% With this receipt, it follows from our theory (cf. Theorem \ref{thm.speciallimits}) that for all $(t,h)\in B_n$ for which 
% \begin{align*}
%   \widehat f_{nh}(t)>n^{-1/2}h^{-5/2} \sqrt{\widehat g_n(t)} \theta^2 \|\phi^{(3)}\|_2 \sqrt{2\log \tfrac \nu h}\big(1+q_{\alpha,n} \tfrac {\log \log \tfrac \nu h}{\log \tfrac \nu h}\big) 
% \end{align*}
% we can conclude that there is a monotone increase somewhere on $[t,t+h]$. Similarly for monotone decreases and these statements hold uniformly with probability $1-\alpha$. In the formula above $\nu>e\approx 2.71$, $\widehat g_n(t)$ is a nonparametric estimator of $g$ satisfying some minimal convergence properties and $q_{\alpha,n}$ denotes the $1-\alpha$-quantile of the statistic
% \begin{align*}
%   \sup_{(t,h)\in B_n} w_h \Big( \tfrac{|\int \phi^{(3)}(\frac{s-t}h)dW_s|}{\sqrt{h}\|\phi^{(3)}\|_2}-\sqrt{2\log \tfrac \nu h}\Big),
%   \quad \text{with} \ \ w_h:= \tfrac{ \sqrt{\frac 12 \log \frac \nu h}}{\log\log \frac \nu h}.
% \end{align*}

\begin{figure}[h]
\begin{center}
  \includegraphics[scale=0.65]{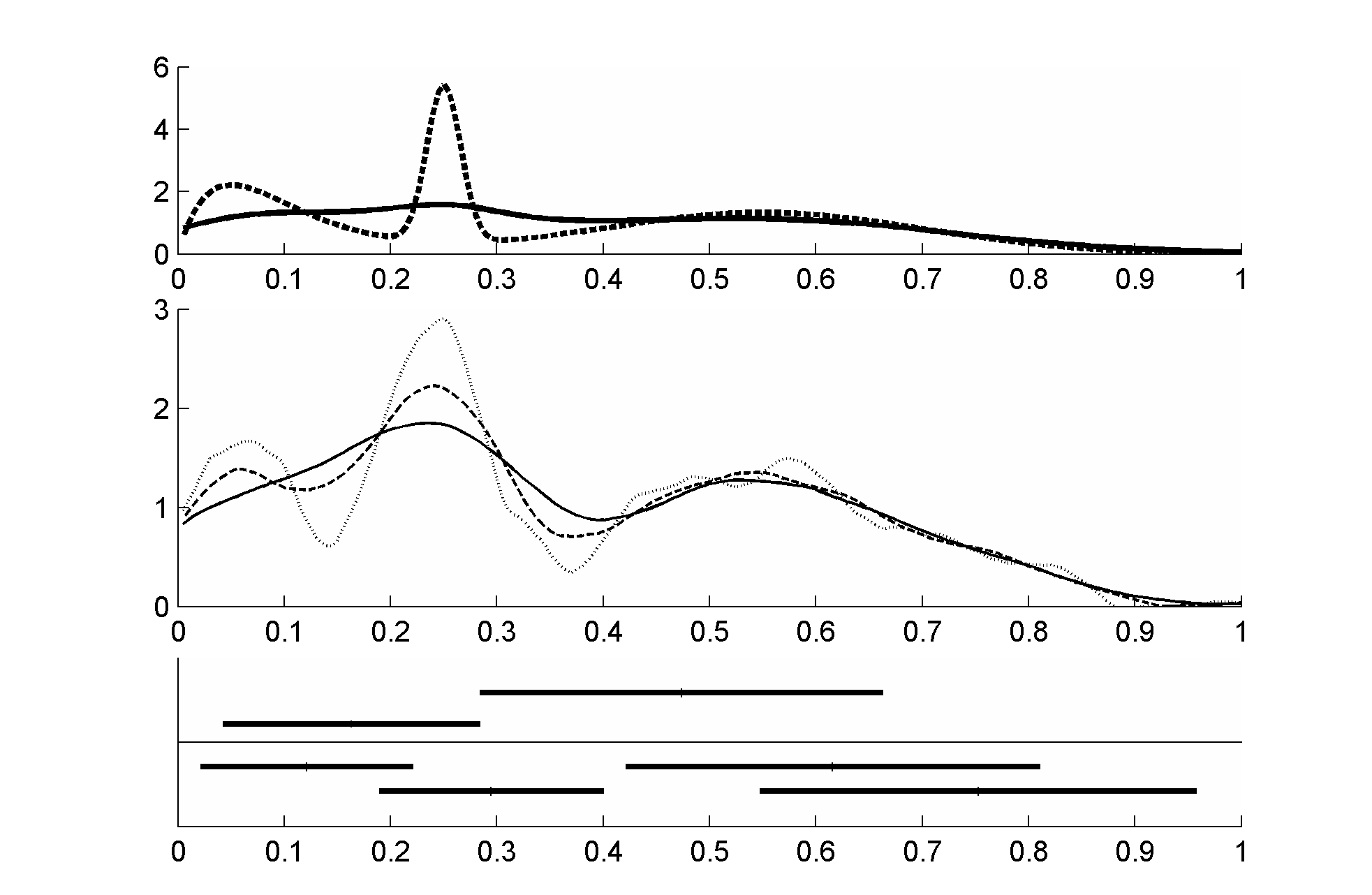}
\caption{Simulation for sample size $n=2000$ and $90\%$-quantile. {\it Upper display:} True density $f$ (dashed) and convoluted density $g$ (solid). {\it Middle display:} Kernel density estimates for $f$ based on the bandwidths $h=0.22$ (' $\hdashrule[0.5ex]{0.4cm}{0.5pt}{1pt}$'), $h=0.31$ (' $\hdashrule[0.5ex]{0.6cm}{0.5pt}{3pt}$'), and $h=0.40$ (' $\hdashrule[0.5ex]{0.4cm}{0.5pt}{}$'). {\it Lower display:} Confidence statements. Thick horizontal lines are intervals with monotone increase/decrease (above/below the thin line).}
\label{fig.intro}
\end{center}
\end{figure} 

Figure \ref{fig.intro} shows a simulation result for $n=2000, N_n=\lfloor n^{3/5}\rfloor, \theta=0.075$ and confidence level $90\%$. The upper panel of Figure \ref{fig.intro} displays the true density of $f$ as well as the convoluted density $g$. Notice that we only have observations with density $g$. In fact, by visual inspection of $g$ it becomes apparent how difficult it is to find segments on which $f$ is monotone increasing/decreasing.

The lower panel of Figure \ref{fig.intro} displays intervals for which we can conclude that there is a monotone increase/decrease. Let us give precise instructions on how to read this plot: Pick any of the thick horizontal lines. Then, with overall probability $90\%$, somewhere in this interval there is a monotone increase or decrease of $f$, depending on whether it is drawn above or below the thin line, respectively. In particular, the fact that intervals with monotone increase and decrease overlap does not yield a contradiction, since the statement is that the monotonicity holds only on a non-empty subset of the corresponding interval. (The way the intervals are piled up in the plot, besides the fact that they are above or below the thin line, is arbitrary and does not contain information.) Recall that we have uniformity in the sense that with confidence $90\%$ all these statements are true simultaneously (cf. also D\"umbgen and Walther \cite{due1}).

To illustrate our approach consider the middle panel in Figure \ref{fig.intro}. Here, we have displayed three reconstructions using $t\mapsto T_{t,h}/(h\sqrt{n})$ as kernel density estimator with the same unimodal kernel as for the test statistic and three different bandwidths $h\in\{0.22,0.31,0.40\}$. Not surprisingly (cf.\ Delaigle and Gijbels \cite{DG04}), the reconstructions yield very different answers what the shape of $f$ could be. For instance, focus on the left hand side of the graph. For $h=0.22$ and $h=0.31$, the density estimators have a mode at around $0.06$, which is completely smoothed out under the larger bandwidth $h=0.4$. As a practitioner, not knowing the truth, we might want to screen for modes by browsing through the plots for varying bandwidths and ask ourselves whether there is another mode or not. With the confidence statement in the lower display, we see that the true density $f$ has to have a monotone decrease on $[0.02, 0.22]$ with confidence  $90\%$ (this is exactly the meaning of the leftmost horizontal line). This rules out the reconstruction without a mode at $0.06$, since it is monotone increasing on the whole interval $[0, 0.25]$ and thus does not reflect the right shape behavior. The kernel density estimator corresponding to the smallest bandwidth $h=0.22$ (although it is the best estimator in a pointwise sense) suggests that there could be another mode at around $0.58$. However, since the confidence intervals do not support such a hypothesis, this could be merely an artefact. Combining the confidence statements in Figure \ref{fig.intro}, we conclude that with $90\%$ confidence the true density has a local minimum and a local maximum on $[0,1]$. Repetition of the simulation shows that often two, three or four segments of increases and decreases are detected, and at most one mode on $[0,1]$ is found (in $69\%$ of the cases). Therefore, sample size $n=2000$ is not large enough to detect systematically the correct number of minima and maxima ($2$ and $3$). Numerical simulations for larger sample size and more details are given in Section \ref{sec.numsim}. 

The derived confidence statements should be viewed as an additional tool for analyzing data, in particular for substantiating vague conclusions or visual impressions from point estimators.

\subsection{Pseudo-differential operators and multiscale analysis}

As mentioned at the beginning of the introduction, we interpret shape constraints as pseudo-differential inequalities. Let $\mathcal{F}(f)=\int_{\mathbb{R}}\exp\left(-ix \cdot\right) f(x) dx$ always denote the Fourier transform of $f\in L^1\left(\mathbb{R}\right)$ or $f\in L^2\left(\mathbb{R}\right)$ (depending on the context). Consider a general class of differential operators $\op(p)$ with symbol $p$ which can be written for nice $f$ as 
\begin{align}
  (\op(p)f)(x)=\frac 1{2\pi}\int e^{ix\xi}p(x,\xi) \mathcal{F}(f)(\xi) d\xi.
  \label{eq.defopp}
\end{align}
This class will be an enlargement of (elliptic) pseudo-differential operators by fractional differentiation. Given data from model \eqref{eq.deconmod} the goal is then to identify intervals at a controlled error level on which $\Re( \op(p)f) \not\le 0$ or $\Re(\op(p)f) \not\ge 0$. Here $\Re$ denotes the projection on the real part. In Subsection \ref{subs.example_Laplace_deconv} we studied implicitly already the case of $\op(p)$ being the differentiation operator $Df=f'$ (monotonicity). If applied to $\op(p)=D^2$ (i.e. $p(x,\xi)=-\xi^2$), our method yields bounds for the number and confidence regions for the location of inflection points of $f$. We also discuss an example related to Wicksell's problem with shape constraint described by fractional differentiation.

The statistic introduced in this paper investigates shape constraints of the unknown density $f$ on all scales simultaneously. Generalizing \eqref{eq.phi_mon_ineq}, we need to derive simultaneous confidence intervals for $\langle \phi\circ S_{t,h} , \Re (\op(p)f)\rangle$ with the scale-and-location shift $S_{t,h}=(\cdot-t)/h$ and the inner product $\langle h_1, h_2 \rangle := \int_{\mathbb{R}} h_1(x) \overline{h_2}(x) \, dx$ in $L^2$. If $\op(p)^\star$ is the adjoint of $\op(p)$ (in a certain space) with respect to $\langle \cdot,\cdot \rangle$, then
\begin{align}
  \sqrt{n} \ \big\langle \phi \circ S_{t,h}, \Re \op(p)f\big\rangle
  &=\sqrt{n} \ \Re \int \big(\op(p)^\star(\phi\circ S_{t,h})\big)(x)f(x) dx \notag \\ 
  &= \frac {\sqrt n}{2\pi} \ \Re\int \mathcal{F}\big(\op(p)^\star(\phi\circ S_{t,h})\big)(s)
  \overline{\mathcal{F}(f)(s)} ds
  \label{eq.Etthexp}
\end{align}
and the r.h.s. can be estimated unbiasedly by the test statistic $$T_{t,h}:= n^{-1/2}\sum_{k=1}^n \Re \ v_{t,h}(Y_k)$$ with
\begin{align*}
  v_{t,h}(u):=\frac 1{2\pi}\int \mathcal{F}\big(\op(p)^\star(\phi\circ S_{t,h})\big)(s)\frac{e^{isu}}
  {\mathcal{F}(f_\epsilon)(-s)}ds.
\end{align*}
This gives rise to a multiscale statistic
\begin{align*}
 T_n=
 \sup_{(t,h)} w_h\left(\frac{\big|T_{t,h}-\E \, T_{t,h}\big|}{\widehat {\Std(T_{t,h})}}-\widetilde w_h\right),
\end{align*}
where $w_h$ and $\widetilde{w}_h$ are chosen in order to calibrate the different scales with equal weight, while $\widehat{\Std(T_{t,h})}$ is an estimator of the standard deviation of $T_{t,h}$.

The key result in this paper is the approximation of $T_n$ by a distribution-free statistic from which critical values can be inferred. Given the critical values, we can in a second step compute bounds $d_{t,h}$ such that a statement of type \eqref{eq.fnh_hat_control} holds. Following the same ideas as in Subsection \ref{subs.example_Laplace_deconv}, this is enough to identify intervals on which $\Re(\op(p)f) \not\le 0$ or $\Re(\op(p)f) \not\ge 0$. In fact the multiscale method implies confidence statements which are stronger than the ones described up to now. These objects can be related to superpositions of confidence bands. For more precise statements see Section \ref{sec.conf_statements}.

\subsection{Comparison with related work and applications}

Hypothesis testing for deconvolution and related inverse problems is a relatively new area. Current methods cover testing of parametric assumptions (cf.\ \cite{BIS09, lau, bishol}) and, more recently, testing for certain smoothness classes such as Sobolev balls in a Gaussian sequence model (Laurent {\it et al.}\ \cite{lau2,lau} and Ingster {\it et al.}\ \cite{ing}). All these papers focus on regression deconvolution models. Exceptions for density deconvolution are Holzmann {\it et al.}\ \cite{hol}, Balabdaoui {\it et al.} \cite{bal}, and Meister \cite{mei} who developed tests for various global hypotheses, such as global monotonicity. The latter test has been derived for one fixed interval and allows to check whether a density is monotone on that interval at a preassigned level of significance. 

% Shape constraints have some connections to confidence bands. For construction and theory in density deconvolution see Bissantz {\it et al.} \cite{bis} and Lounici and Nickl \cite{lou}. 

Our work can also be viewed as an extension of Chaudhuri and Marron \cite{cha} as well as D\"umbgen and Walther \cite{due1} who treated the case $\op(p) = D^m$ (with $m=1$ in \cite{due1}) in the direct case, i.e.\ when $\epsilon=0$. However, the approach in \cite{cha} does not allow for sequences of bandwidths tending to zero and yields limit distributions depending on unknown quantities again. The methods in \cite{due1} require a deterministic coupling result. The latter allows to consider the multiscale approximation for $f=\mathbb{I}_{[0,1]}$ only, but it cannot be transferred to the deconvolution setting.

One of the main advantages of multiscale methods, making it attractive for applications, is that essentially no smoothing parameter is required. The main choice will be the quantile of the multiscale statistic, which has a clear probabilistic interpretation. Furthermore, our multiscale statistic allows to construct estimators for the number of modes and inflection points which have a number of nice properties: First, modes and inflection points are detected with the minimax rate of convergence (up to a log-factor). Second, the probability that the true number is overestimated can be made small, since it is completely controlled by the quantile of the multiscale statistic. To state it differently, it is highly unlikely that artefacts are detected, which is a desirable property in many applications. It is worth noting that neither assumptions are made on the number of modes nor additional model selection penalties are necessary.

For practical applications, we may use these models if for instance the error variable $\epsilon$ is an independent waiting time. For example let $X_i$ be the (unknown) time of infection of the $i$-th patient, $\epsilon_i$ the corresponding incubation time, and $Y_i$ is the time when diagnosis is made. Then, it is convenient to assume $\epsilon\sim\Gamma\left(r,\theta\right)$ (see for instance \cite{ang}, Section 3.5). By the techniques developed in this paper one will be able to identify for example time intervals where the number of infections increased and decreased for a specified confidence level. Another application is single photon emission computed tomography (SPECT), where the detected scattered photons are blurred by Laplace distributed random variables (cf.\ Floyd {\it et al.}\ \cite{flo}, Kacperski {\it et al.}\ \cite{kac}).

The paper is organized as follows. In Section \ref{sec.genmsc} we show how distribution-free approximations of multiscale statistics can be derived for general empirical processes under relatively weak conditions. For the precise statement see Theorem \ref{thm.gmsc}. These results are transferred to shape constraints and deconvolution models in Section \ref{sec.mono}. In Section \ref{sec.conf_statements} we discuss the statistical consequences and show how confidence statements can be derived. Theoretical questions related to the performance of the multiscale method and numerical aspects are discussed in Sections \ref{sec.performance} and \ref{sec.numsim}. Proofs and further technicalities are shifted to the appendix.% and a supplementary part \cite{sch}. 

\medskip

{\it Notation:} We write $\mT$ for the set $[0,1]\times (0,1]$. The expression $\lfloor x\rfloor$ means the largest integer not exceeding $x$. The support of a function $\phi$ is $\supp \phi$, $\|\cdot\|_p$ denotes the norm in $L^p:=L^p(\mathbb{R})$, and $\TV(\cdot)$ stands for the total variation of functions on $\mathbb{R}$. As customary in the theory of Sobolev spaces, put $\langle s\rangle:=(1+|s|^2)^{1/2}$. One should not confuse this with $\langle \cdot, \cdot \rangle$, the $L^2$-inner product. If it is clear from the context, we write $x^k\phi$ and $\langle x\rangle^k\phi$ for the functions $x\mapsto x^k\phi(x)$ and $x\mapsto \langle x\rangle^k\phi(x)$, respectively. The ($L^2$-)Sobolev space $H^r$ is defined as the class of functions with norm
\begin{align*}
  \|\phi\|_{H^r}:=\Big(\int \langle s\rangle^{2r}|\mathcal{F}(\phi)(s)|^2 ds\Big)^{1/2}< \infty.
\end{align*}
For any $q$ and $\ell \in \mathbb{N}$ ($\mathbb{N}$ is always the set of non-negative integers) define $H^q_{\ell}$ as the Sobolev type space
\begin{align*}
  H^q_{\ell}:=\big\{\psi \ | \ x^k\psi \in H^q, \ \text{for}\ k=0,1,\ldots,\ell \ \big\}
\end{align*}
with norm $\|\psi\|_{H^q_\ell}:=\sum_{k=0}^{\ell}\|x^k\psi\|_{H^q}$.

\section{A general multiscale test statistic}
\label{sec.genmsc}

In this section, we shall give a fairly general convergence result which is of interest on its own. The presented result does not use the deconvolution structure of model \eqref{eq.deconmod}. It only requires that we have observations $Y_i=G^{-1}(U_i), \ i=1,\ldots,n$ with $U_i$ i.i.d. uniform on $[0,1]$ and $G$ an unknown distribution function with Lebesgue density $g$ in the class
\begin{align}
\mathcal{G}:=  \mathcal{G}_{c,C,q}:= \big\{ G\  \big | &\ G \ 
\text{is a distribution function with density} \ g, \notag \\  
  &\ \ c\leq g \big |_{[0,1]}, \ \|g\|_\infty\leq c^{-1}, \ \text{and} \ g\in \mathcal{J}(C,q) \ \big\}
  \label{eq.GcCdef}
\end{align}
for fixed $c,C\geq 0$, $0\leq q<1/2$, and the Lipschitz type constraint
\begin{align*}
  \mathcal{J}&:=\mathcal{J}(C,q) \\&:=\big \{h \ \big| \ |\sqrt{h(x)}-\sqrt{h(y)}|\leq C (1+|x|+|y|)^q|x-y|, \ \text{for all} \ x,y\in \mathbb{R} \big \}.
\end{align*}

For a set of real-valued functions $(\psi_{t,h})_{t,h}$ define the test statistic (empirical process) $T_{t,h}=n^{-1/2} \sum_{k=1}^n \psi_{t,h}(Y_k)$. If $h$ is small and $\psi_{t,h}$ localized around $t$, then $\Std(T_{t,h})\approx (\int \psi_{t,h}^2(s)g(s)ds)^{1/2}\approx \|\psi_{t,h}\|_2 \sqrt{g(t)}$. It will turn out later on that one should allow for a slightly regularized standardization and therefore we consider 
\begin{align*}
  \frac{|T_{t,h}-\E[T_{t,h}]|}{V_{t,h} \ \sqrt{\widehat g_n(t)}}
\end{align*}
with $V_{t,h}\geq \|\psi_{t,h}\|_2$ and $\widehat g_n$ an estimator of $g$, satisfying
\begin{align}
  \sup_{G\in \mathcal{G}}\|\widehat g_n-g\|_\infty= O_P(1/\log n).
  \label{eq.widehatgdef}
\end{align} 
Unless stated otherwise, asymptotic statements refer to $n \to \infty$. We combine the single test statistics for an arbitrary subset
\begin{align}
 B_n \subset \big\{(t,h) \big | \ t\in [0,1], \ h\in[l_n,u_n] \big\}
  \label{eq.Bndef}
\end{align}
and consider for $\nu >e$ and
\begin{align}
  w_h= \frac{ \sqrt{\tfrac 12 \log \tfrac \nu h}}{\log\log \tfrac \nu h},
  \label{eq.whdef_first}
\end{align}
distribution-free approximations of the multiscale statistic
\begin{align}
  T_n:=\sup_{(t,h)\in B_n} w_h\left(\frac{\big|T_{t,h}-\E [T_{t,h}]\big|}{V_{t,h} \ \sqrt{\widehat g_n(t)}}-\sqrt{2\log \tfrac \nu h}\right).
  \label{eq.defgenmsc}
\end{align}

\begin{ass}[Assumption on test functions]
\label{as.testfcts}
Given a set $B_n$ of the form \eqref{eq.Bndef}, functions $(\psi_{t,h})_{(t,h)\in \mT}$, and numbers $(V_{t,h})_{(t,h)\in \mT}$, suppose that the following assumptions hold.
\begin{itemize}
 \item[(i)] For all $(t,h)\in\mT$, $\|\psi_{t,h}\|_2\leq V_{t,h}$.
 \item[(ii)] We have uniform bounds on the norms
   \begin{align*}
   \sup_{(t,h)\in \mT} \frac{\sqrt h \TV(\psi_{t,h}) + \sqrt h \|\psi_{t,h}\|_\infty+h^{-1/2}\|\psi_{t,h}\|_1}{V_{t,h}}\lesssim 1.
    \end{align*}
 \item[(iii)] There exists $\alpha>1/2$ such that
  \begin{align*}
  \kappa_n:=\sup_{(t,h)\in B_n, \ G\in \mathcal{G}} w_h \frac{\TV\Big(\psi_{t,h}(\cdot) \big[\sqrt{g(\cdot)}-\sqrt{g(t)}\big]\langle \cdot \rangle^{\alpha}\Big)}{ V_{t,h}} \rightarrow 0.
  \end{align*}
  \item[(iv)] 
  There exists a constant $K$ such that for all $(t,h), (t',h')\in \mT$,
  \begin{align*}
  \frac{\sqrt{h}\wedge \sqrt{h'}}{V_{t,h} \vee V_{t',h'}}
  \Big[
  \|\psi_{t,h}-\psi_{t',h'}\|_2+|V_{t,h}-V_{t',h'}|\Big]
  \leq K\sqrt{|t-t'|+|h-h'|}.
  \end{align*}
\end{itemize}
\end{ass}

\begin{thm}
\label{thm.gmsc}
Given a multiscale statistic of the form \eqref{eq.defgenmsc}. Work in model \eqref{eq.deconmod} under Assumption \ref{as.testfcts} and suppose that $l_n n \log^{-3}n \rightarrow \infty$ and $u_n=o(1)$. If the process $(t,h)\mapsto\sqrt{h}V_{t,h}^{-1}\int \psi_{t,h}(s)dW_s$ has continuous sample paths on $\mT$, then there exists a (two-sided) standard Brownian motion $W$, such that for $\nu>e$,
\begin{align}  
   \sup_{G\in \mathcal{G}_{c,C,q}} \Big |T_n-\sup_{(t,h)\in B_n} w_h \left(
  \frac{\big|\int \psi_{t,h}(s) dW_s \big|}{V_{t,h}}-\sqrt{2\log \tfrac \nu h }\right) \Big | = O_P(r_n),
  \label{eq.approx_by_distr_free_thm1}
\end{align}
with
\begin{align*}
  r_n=\sup_{G\in \mathcal{G}}\big\|\widehat g_n- g\big\|_\infty \frac{\log n}{\log\log n}+l_n^{-1/2}n^{-1/2}\frac{\log^{3/2} n}{\log\log n}
  +  \frac{\sqrt{u_n \log(1/u_n)}}{\log\log(1/u_n)}+\kappa_n .
\end{align*}
Moreover, 
\begin{align}  
  \sup_{(t,h)\in\mT}
  w_h \left(
  \frac{\big|\int \psi_{t,h}(s) dW_s \big|}{V_{t,h}}-\sqrt{2\log \tfrac \nu h }\right) \  <\infty, \quad \text{a.s.}
  \label{eq.limitbd1}
\end{align}
Hence, the approximating statistic in \eqref{eq.approx_by_distr_free_thm1} is almost surely bounded from above.
\end{thm}

The proof of the coupling in this theorem (cf.\ Appendix \ref{eq.secproofofmsc}) is based on generalizing techniques developed by Gin\'e {\it et al.}\ \cite{gin2}, while finiteness of the approximating test statistic utilizes results of D\"umbgen and Spokoiny \cite{due2}. Note that Theorem \ref{thm.gmsc} can be understood as a multiscale analog of the $L^\infty$-loss convergence for kernel estimators (cf.\ \cite{gin3, gin2, bis, gin}). 

To give an example, let us assume that $\psi_{t,h}=\psi(\tfrac{\cdot-t}h)$ is a kernel function. By Lemmas \ref{lem.condTVreplace} and \ref{lem.phiL2}, Assumption \ref{as.testfcts} holds for $V_{t,h}=\|\psi_{t,h}\|_2=\sqrt{h}\|\psi\|_2$ whenever $\psi\neq 0$ on a Lebesgue measurable set, $\TV(\psi)<\infty$ and $\supp \psi \subset [0,1]$. Furthermore, by partial integration, we can easily verify that the process $(t,h)\mapsto\|\psi\|_2^{-1}\int \psi_{t,h}(s)dW_s$ has continuous sample paths (cf. \cite{due2}, p. 144).

For an application of Theorem \ref{thm.gmsc} to wavelet thresholding, cf. Example \ref{exam.wav_thresh} in the appendix. Let us close this section with a result on the lower bound of the approximating statistic.

Theorem \ref{thm.gmsc} shows that the approximating statistic is almost surely bounded from above. On the contrary, we have the trivial lower bound
\begin{align*}
  T_n\geq -\inf_{(t,h)\in B_n}\frac{\log \tfrac \nu h}{\log\log \tfrac \nu h},
\end{align*}
which converges to $-\infty$ in general and describes the behavior of $T_n$, provided the cardinality of $B_n$ is small (for instance if $B_n$ contains only one element). However, if $B_n$ is sufficiently rich, $T_n$ can be shown to be bounded from below, uniformly in $n$. Let us make this more precise. Assume, that for every $n$ there exists a $K_n$ such that $K_n \rightarrow \infty$ and 
\begin{align}
  B_{K_n}^\circ:=\big\{\big(\tfrac i{K_n}, \tfrac 1{K_n}\big) \ \big |\  i=0,\ldots, K_n-1\big\}
  \subset B_n.
  \label{eq.defBKncirc}
\end{align}
Then, the approximating statistic is asymptotically bounded from below by $-1/4$. This follows from Lemma \ref{lem.limitlimit} in the appendix. It is a challenging problem to calculate the distribution for general index set $B_n$ explicitly. Although the tail behavior has been studied for the one-scale case (cf.\ \cite{gin2, bis}) this has not been addressed so far for the approximating statistic in Theorem \ref{thm.gmsc}. For implementation, later on, our method relies therefore on Monte Carlo simulations.

% (+ kernel estimator ???)

\section{Testing for shape constraints in deconvolution}
\label{sec.mono}

We start by defining the class of differential operators in \eqref{eq.defopp}. However, before making this precise, let us define pseudo-differential operators in dimension one as well as fractional integration and differentiation. Given a real $m$, consider $S^m$ the class of functions $a:\mathbb{R}\times \mathbb{R}\rightarrow \mathbb{C}$ such that for all $\alpha,\beta \in \mathbb{N}$, 
\begin{align}
  |\partial_x^\beta \partial_\xi^\alpha a(x,\xi)|\leq C_{\alpha,\beta}(1+|\xi|)^{m-\alpha} \quad \text{for all} \  x,\xi\in \mathbb{R}.
  \label{eq.Smsymboldef}
\end{align}
Then the pseudo-differential operator $\Op(a)$ corresponding to the symbol $a$ can be defined on the Schwartz space of rapidly decreasing functions $\mathcal{S}$ by
\begin{align*}
  &\Op(a): \mathcal{S}\rightarrow \mathcal{S} \\
  &\Op(a)\phi(x):= \frac 1{2\pi} \int e^{ix \xi} a(x,\xi)\mathcal{F}(\phi)(\xi)d\xi.
\end{align*}
It is well-known that for any $s\in \mathbb{R}$, $\Op(a)$ can be extended to a continuous operator $\Op(a): H^{m+s}\rightarrow H^s$. In order to simplify the readability, we only write $\Op$ for pseudo-differential operators and $\op$ in general for operators of the form \eqref{eq.defopp}. Throughout the paper, we write $\iota_s^\alpha = \exp(\alpha \pi i \sign(s) /2)$ and understand as usual $(\pm is)^\alpha = |s|^\alpha\iota_s^{\pm \alpha}$. The Gamma function evaluated at $\alpha$ will be denoted by $\Gamma(\alpha)$. Let us further introduce the Riemann-Liouville fractional integration operators on the real axis for $\alpha>0$ by
\begin{align}
  \big(I_+^\alpha h\big)(x)&:= \frac 1{\Gamma(\alpha)}\int_{-\infty}^x \frac{h(t)}{(x-t)^{1-\alpha}}dt
  \quad \text{and} \quad \notag \\
  \big(I_-^\alpha h\big)(x)&:= \frac 1{\Gamma(\alpha)}\int_x^{\infty} \frac{h(t)}{(t-x)^{1-\alpha}}dt.
  \label{eq.Ipmalphadef}
\end{align}
For $\beta\geq 0$, we define the corresponding fractional differentiation operators $(D_+^\beta h)(x):= D^n(I_+^{n-\beta}h)(x)$ and $(D_-^\beta h)(x)=(-D)^n(I_-^{n-\beta}f)(x)$, where $n=\lfloor\beta\rfloor+1$. For any $s\in \mathbb{R}$, we can extend $D_+^\beta$ and $D_-^\beta$ to continuous operators from $H^{\beta+s}\rightarrow H^s$ using the identity (cf. \cite{kil}, p.90),
\begin{align}
  \mathcal{F}\big(D_{\pm}^\beta h\big)(\xi)
  =(\pm i\xi)^\beta \mathcal{F}\big(h\big)(\xi)
  =\iota_\xi^{\pm \beta} |\xi|^\beta \mathcal{F}\big(h\big)(\xi).
\end{align}

In this paper, we consider operators $\op(p)$ which ``factorize'' into a pseudo-differential operator and a fractional differentiation in Riemann-Liouville sense. More precisely, the symbol $p$ is in the class
\begin{align*}
  \underline S^m:=\big\{  \ (x,\xi)\mapsto p(x,\xi)=a(x,\xi)|\xi|^\gamma \iota_\xi^\mu \ \big| \ &a\in S^{\om}, \ m=\om+\gamma, \\ &\gamma\in\{0\} \cup [1,\infty), \  \mu \in \mathbb{R}\ \big \}.
\end{align*}
Let us mention that we cannot allow for all $\gamma\geq 0$ since in our proofs it is essential that $\partial_\xi^2p(x,\xi)$ is integrable. The results can also be formulated for finite sums of symbols, i.e. $\sum_{j=1}^J p_j$ and $p_j \in \underline S^m$. However, for simplicity we restrict us to $J=1$.

Throughout the remaining part of the paper, we will always assume that $\op(p)f$ is continuous. A closed rectangle in $\mathbb{R}^2$ parallel to the coordinate axes with vertices $(a_1,b_1), (a_1,b_2), (a_2, b_1), (a_2,b_2), a_1<a_2, \ b_1<b_2$ will be denoted by $[a_1,a_2]\times [b_1,b_2]$.

The main objective of this paper is to obtain uniform confidence statement of the following kinds:

\begin{enumerate}
 \item[($i$)] The number and location of the roots and maxima of $\op(p)f$.
 \item[($ii$)] Simultaneous identification of intervals of the form $[t_i,t_i+h_i], \ t_i\in [0,1], h_i>0, \ i$ in some index set $I$, with the following property: For a pre-specified confidence level we can conclude that for all $i\in I$ the functions $(\op(p)f)|_{[t_i,t_i+h_i]}$ attain, at least on a subset of $[t_i,t_i+h_i]$, positive values.

 \item[($ii'$)] Same as $(ii)$, but we want to conclude that $(\op(p)f)|_{[t_i,t_i+h_i]}$ has to attain negative values.
 \item[($iii$)] For any pair $(t,h)\in B_n$ with $B_n$ as in \eqref{eq.Bndef}, we want to find $b_-(t,h,\alpha)$ and $b_+(t,h,\alpha)$, such that we can conclude that with overall confidence $1-\alpha$, the graph of $\op(p)f$ (denoted as $\graph(\op(p)f)$ in the sequel) has a non-empty intersection with every rectangle $[t,t+h]\times[b_-(t,h,\alpha),b_+(t,h,\alpha)]$.
\end{enumerate}

In the following we will refer to these goals as Problems $(i)$, $(ii)$, $(ii')$ and $(iii)$, respectively. Note that $(ii)$ follows from $(iii)$ by taking all intervals $[t,t+h]$ with $b_-(t,h,\alpha)>0$. Analogously, $[t,t+h]$ satisfies ($ii'$) whenever $b_+(t,h,\alpha)<0$. The geometrical ordering of the intervals obtained by $(ii)$ and $(ii')$ yields in a straightforward way a lower bound for the number of roots of $\op(p)f$, solving Problem ($i$) (cf.\ also D\"umbgen and Walther \cite{due1}). A confidence interval for the location of a root can be constructed as follows: If there exists $[t,t+h]$ such that $b_-(t,h,\alpha)>0$ and $[\widetilde t, \widetilde t+\widetilde h]$ with $b_+(\widetilde{t},\widetilde{h},\alpha)<0$, then, with confidence $1-\alpha$, $\op(p)f$ has a zero in the interval $\bigl[ \min(t,\widetilde{t}), \max(t+h,\widetilde{t}+\widetilde{h}) \bigr]$. The maximal number of disjoint intervals on which we find zeros is then an estimator for the number of roots.

\begin{exam}
\label{exam.A=D}
In the example in Section \ref{subs.example_Laplace_deconv} we had $\op(p)=D$. In this case we want to find a collection of intervals $[t,t+h]$ such that with overall probability $1-\alpha$ for each such interval there exists a nondegenerate subinterval on which $f$ is strictly monotonically increasing/decreasing.
\end{exam}

Instead of studying qualitative features of $X$ directly, we might as well be interested in properties of the density of a transformed random variable $q(X)$. If $X$ is non-negative and $a>0$, $q$ could be for instance a (slightly regularized) log-transform $q=\log(\cdot+a)$.

\begin{exam}
\label{ex.chofva}
Suppose that we want to analyze the convexity/concavity properties of $U=q(X)$, where $q$ is a smooth function, which is strictly monotone increasing on the support of the distribution of $X$. Let $f_U$ denote the density of $U$. Then, by change of variables
\begin{align*}
  f_U(y)=\frac 1{q'\big(q^{-1}(y)\big)}f\big(q^{-1}(y)\big),
\end{align*}
and there is a pseudo-differential operator $\Op(p)$ with symbol
\begin{align*}
  p(x,\xi)=-\frac{1}{(q'(x))^2}\xi^2-
  \frac{q''(x)q'(x)+2q''(x)}{(q'(x))^4} i\xi
  +\frac{3(q''(x))^2-q'''(x)q'(x)}{(q'(x))^5},
\end{align*}
such that $f''_U(y)=(\op(p)f)(q^{-1}(y))$. Therefore, $$\graph(\op(p)f) \cap [t,t+h]\times [b_-(t,h,\alpha),b_+(t,h,\alpha)] \neq \varnothing$$ implies
$$
	\graph(f''_U)\cap [q(t),q(t+h)]\times [b_-(t,h,\alpha),b_+(t,h,\alpha)] \neq \varnothing.
$$
In particular, if $b_-(t,h,\alpha)>0$ then, with confidence $1-\alpha$, we may conclude that $f_U$ is strictly convex on a nondegenerate subinterval of $[q(t),q(t+h)]$.

\end{exam}

\begin{exam}[Noisy Wicksell problem]
\label{exam.noisy_w_problem}
In the classical Wicksell problem, cross-sections of a plane with randomly distributed balls in three-dimensional space are observed. From these observations the distribution $H$ or density $h=H'$ of the squared radii of the balls has to be estimated (cf. Groeneboom and Jongbloed \cite{gro}). Statistically speaking, we have observations $X_1,\ldots,X_n$ with density $f$ satisfying the following relationship (cf. Golubev and Levit \cite{gol2}) 
\begin{align*}
  1-H(x) \propto \int_x^\infty \frac{f(t)}{(t-x)^{1/2}}dt
  = \Gamma(\tfrac 12) (I_{-}^{1/2} f)(x), \quad \text{for all} \ x\in [0,\infty),
\end{align*}
where $\propto$ means up to a positive constant and $I_{-}^{1/2}$ as in \eqref{eq.Ipmalphadef}. Suppose now, that we are interested in monotonicity properties of the density $h=H'$ on $[0,1]$. For $x>0$, $-h'\lessgtr 0$ iff the fractional derivative of order $3/2$ satisfies $(D^{3/2}_{-}f)(x)=D^2 (I_{-}^{1/2} f)(x) \lessgtr 0$. It is reasonable to assume in applications that the observations are corrupted by measurement errors, which means we only observe $Y_i=X_i+\epsilon_i$ as in model \eqref{eq.deconmod}. Hence, we are in the framework described above and the shape constraint is given by $\op(p)f\lessgtr 0$ for $p(x,\xi)=\iota_\xi^{-3/2}|\xi|^{3/2}$.
\end{exam}

In order to formulate our results in a proper way, let us introduce the following definitions. We say that a pseudo-differential operator $\Op(a)$ with $a\in S^m$ and $S^m$ as in \eqref{eq.Smsymboldef}, is elliptic, if there exists $\xi_0$ such that $|a(x,\xi)|>K|\xi|^m$ for a positive constant $K$ and all $\xi$ satisfying $|\xi|>|\xi_0|$. In the framework of Example \ref{ex.chofva} for instance, ellipticity holds if $\|q'\|_\infty<\infty$. It is well-known that ellipticity is equivalent to a generalized invertibility of the operator. Furthermore, for an arbitrary symbol $p\in S^{\overline m}$ let us denote by $\Op(p^\star)$ the adjoint of $\Op(p)$ with respect to the inner product $\langle \cdot,\cdot \rangle$. This is again a pseudo-differential operator and $p^\star \in S^{\overline m}$. Formally, we can compute $p^\star$ by $p^\star(x,\xi)=e^{\partial_x \partial_\xi}\overline p(x,\xi)$, where $\overline p$ denotes the complex conjugate of $p$. Here the equality holds in the sense of asymptotic summation (for a precise statement see Theorem 18.1.7 in H\"ormander \cite{hoer}). Now, suppose that we have a symbol in $\underline S^m$ of the form $a|\xi|^\gamma\iota_\xi^\mu=a(x,\xi)|\xi|^\gamma\iota_\xi^\mu$ with $a\in S^{\om}$ and $\om+\gamma=m$. Since for any $u,v\in H^{m}$,
\begin{align}
  \langle \op(a|\xi|^\gamma\iota_\xi^\mu)u,v \rangle 
  &= \langle \Op(a)\op(|\xi|^\gamma\iota_\xi^\mu)u,v \rangle
  = \langle \op(|\xi|^\gamma\iota_\xi^\mu)u, \Op(a^\star)v \rangle \notag \\
  &= \langle u, \op(|\xi|^\gamma\iota_\xi^{-\mu})\Op(a^\star)v \rangle
  \label{eq.comp_adj}
\end{align}
we conclude that $\mathcal{F}(\op(a|\xi|^\gamma\iota_\xi^\mu)^\star\phi)= |\xi|^\gamma\iota_\xi^{-\mu} \mathcal{F}(\Op(a^\star) \phi)$ for all $\phi \in H^m$. 

In order to formulate the assumptions and the main result, let us fix one symbol $p\in \underline S^m$ and one factorization $p(x,\xi)=a(x,\xi)|\xi|^\gamma\iota_\xi^{\mu}$ with $a, \gamma, \mu$ as in the definition of $\underline S^m$.

\begin{ass}
\label{ass.noise}
We assume that there is a positive real number $r>0$ and constants $0<C_l\leq C_u<\infty$ such that the characteristic function of $\epsilon$ is bounded from below and above by
\begin{align*}
	C_l\langle s\rangle^{-r} \leq  |\E \, e^{-is\epsilon}|
	= |\mathcal{F}(f_{\epsilon})(s)|
	\leq C_u\langle s\rangle^{-r} \quad \text{for all} \ s\in \mathbb{R}.
\end{align*}
Moreover, suppose that the second derivative of $\mathcal{F}(f_\epsilon)$ exists and
$$
	\langle s\rangle |D\mathcal{F}(f_\epsilon)(s)|+\langle  s\rangle^2 |D^2 \mathcal{F}(f_\epsilon)(s)|\leq C_u\langle s\rangle^{-r} \quad \text{for all} \ s\in \mathbb{R}.
$$
\end{ass}

These are the classical assumptions on the decay of the Fourier transform of the error density in the moderately ill-posed case (cf. Assumptions (G1) and (G3) in Fan \cite{Fan91}). Heuristically, we can think of $\mathcal{F}(f_{\epsilon})$ as an elliptic symbol in $S^{-r}$.

Let $\Re$ denote the projection on the real part. For sufficiently smooth $\phi$ consider the test statistic
\begin{align}
  T_{t,h}:= \frac 1{\sqrt{n}}\sum_{k=1}^n \Re \ v_{t,h}(Y_k)
  = \frac 1{\sqrt{n}}\sum_{k=1}^n \Re \ v_{t,h}(G^{-1}(U_k))
  \label{eq.Tthdef}
\end{align}
with
\begin{align}
  v_{t,h}=\mathcal{F}^{-1}\big(\lambda_\gamma^\mu(\cdot)\mathcal{F}\big(\Op(a^\star)(\phi\circ S_{t,h})\big)\big)
  \label{eq.vthdef}
\end{align}
and
\begin{align}
  \lambda(s)=\lambda_\gamma^\mu(s)=\frac{|s|^\gamma\iota_s^{-\mu}}{\mathcal{F}(f_\epsilon)(-s)}.
  \label{eq.lambdadef}
\end{align}
From \eqref{eq.Etthexp} and \eqref{eq.comp_adj}, we find that for $f\in H^m$,
\begin{align*}
  \E\, T_{t,h} 
  &= \sqrt{n}\int (\phi \circ S_{t,h}) (x) \Re \big(\op(p)f\big)(x) dx.
\end{align*}
Proceeding as in Section \ref{sec.genmsc} we consider the multiscale statistic
\begin{align}
  T_n=\sup_{(t,h)\in B_n} w_h\left(\frac{\big|T_{t,h}-\E [T_{t,h}]\big|}{\sqrt{\widehat g_n(t)} \ \|v_{t,h}\|_2}-\sqrt{2\log \tfrac \nu h}\right),
  \label{eq.Tndef}
\end{align}
i.e. with the notation of \eqref{eq.defgenmsc}, we set $\psi_{t,h}:=\Re v_{t,h}$ and $V_{t,h}:=\|v_{t,h}\|_2$. Define further
\begin{align*}
  T_n^\infty(W)
  :=
  \sup_{(t,h)\in B_n} w_h \left(
  \frac{\big|\int \Re v_{t,h}(s)dW_s \big|}{\|v_{t,h}\|_2}-\sqrt{2\log \tfrac \nu h }\right).
\end{align*}

\begin{thm}
\label{thm.msc}
Given an operator $\op(p)$ with symbol $p \in \underline S^m$ and let $T_n$ be as in \eqref{eq.Tndef}. Work in model \eqref{eq.deconmod} under Assumption \ref{ass.noise}. Suppose that 
\begin{itemize}
 \item[(i)] $l_n n \log^{-3}n \rightarrow \infty$ and $u_n=o(\log^{-3}n)$,
 \item[(ii)] $\phi \in H_4^{\lfloor r+m+5/2\rfloor}$, \ $\supp \phi \subset [0,1]$, and \ $\TV(D^{\lfloor r+m+5/2\rfloor} \phi)<\infty$,
 \item[(iii)] $\Op(a)$ is elliptic.
\end{itemize}
Then, there exists a (two-sided) standard Brownian motion $W$, such that for $\nu>e$,
\begin{align}  
   \sup_{G\in \mathcal{G}_{c,C,q}} \Big |T_n- T_n^\infty(W)\Big | = o_P(r_n),
    \label{eq.mainstate}
\end{align}
with
\begin{align*}
  r_n=\sup_{G\in \mathcal{G}}\big\|\widehat g_n- g\big\|_\infty \frac{\log n}{\log\log n}+l_n^{-1/2}n^{-1/2}\frac{\log^{3/2} n}{\log\log n}
  + u_n^{1/2}\log^{3/2}n.
\end{align*}
Moreover,
\begin{align}  
  \sup_{(t,h)\in\mT}
  w_h \left(
  \frac{\big|\int \Re v_{t,h}(s) dW_s \big|}{\|v_{t,h}\|_2}-\sqrt{2\log \tfrac \nu h }\right) \  <\infty, \quad \text{a.s.}
  \label{eq.limit2}
\end{align}
Hence, the approximating statistic $T_n^\infty(W)$ is almost surely bounded from above by \eqref{eq.limit2}.
\end{thm}

One can easily show using Lemma \ref{lem.limitlimit}, that if $B_n$ contains \eqref{eq.defBKncirc} and the symbol $p$ does not depend on $t$, then the approximating statistic is also bounded from below. Furthermore, the case $\epsilon=0$ can be treated as well (we can define $\mathcal{F}(f_\epsilon)=1$ in this case). In particular, our framework allows for the important case $\epsilon=0$ and $\op(p)$ the identity operator, which cannot be treated with the results from \cite{due1}.

% Moreover, given the results and the remarks in Section \ref{sec.genmsc}, we can construct confidence bands for deconvolution densities which satisfy Assumption \ref{ass.noise}. 

For special choices of $p$ and $f_\epsilon$ the functions $(v_{t,h})_{t,h}$ have a much simpler form, which allows to read off the ill-posedness of the problem from the index of the pseudo-differential operator associated with $v_{t,h}$. Let us shortly discuss this. Suppose Assumption \ref{ass.noise} holds and additionally $\langle s\rangle^k |D^k\mathcal{F}(f_\epsilon)(s)|\leq C_k \langle s\rangle^{-r}$ for all $s\in \mathbb{R}$ and $k=3,4,\ldots$ Then $(x,\xi)\mapsto \mathcal{F}(f_\epsilon)(-\xi)$ defines a symbol in $S^{-r}$. Because of the lower bound in Assumption \ref{ass.noise}, $C_l\langle \xi\rangle^{-r}\leq |\mathcal{F}(f_\epsilon)(-\xi)|$, the corresponding pseudo-differential operator is elliptic and $(x,\xi)\mapsto 1/\mathcal{F}(f_\epsilon)(-\xi)$ is the symbol of a parametrix and consequently an element in $S^r$ (cf. H\"ormander \cite{hoer}, Theorem 18.1.9). %Das is trivial. Im direkten fall ist g=Fourier^(-1) (F(f_\epsilon)(-\xi) F(f) ).  Man sieht das es eine Inversion ist, indem man die Fourier transformierte nimmt und dann durch $\mathcal{F}(f_\epsilon)(-\xi)$ teilt. Mit dem Resultat von Hoermander muss dann auch schon 1/\mathcal{F}(f_\epsilon)(-\xi) in S^r sein,
If $\phi \in H^{r+m}$ and $p\in \underline{S}^m \cap S^m$, then
\begin{align*}
  v_{t,h}(u)
  &=\frac 1{2\pi}\int \mathcal{F}\big(\Op\big(\tfrac 1{\mathcal{F}(f_\epsilon)(-\cdot)}\big)
  \circ \Op(p^\star) \big(\phi\circ S_{t,h}\big)\big) (s) e^{isu} ds \\
  &=\Op\big(\tfrac 1{\mathcal{F}(f_\epsilon)(-\cdot)}\big)
  \circ \Op(p^\star) \big(\phi\circ S_{t,h}\big)(u).
\end{align*}
Pseudo-differential operators are closed under composition. More precisely, $p_j\in S^{m_j}$ for $j=1,2$ implies that the symbol of the composed operator is in $S^{m_1+m_2}$. Therefore, there is a symbol $\widetilde p\in S^{m+r}$ such that $v_{t,h}=\Op(\widetilde p)(\phi\circ S_{t,h})$. Hence, for fixed $h$, the function $t\mapsto v_{t,h}$ can be viewed as a kernel estimator with bandwidth $h$. Furthermore, the problem is completely determined by the composition $\Op(\widetilde p)$ and this yields a heuristic argument why (as it will turn out later) the ill-posedness of the detection problem $\Re \op(p)f\lessgtr 0$ in model (\ref{eq.deconmod}) is determined by the sum $m+r$, i.e. 
\begin{center}
                                                                                                                                                                                                                                                                                                                                        ill-posedness of shape constraint $+$ ill-posedness of deconvolution problem. 
                                                                                                                                                                                                                                                                                                                                       \end{center}
Suppose further that $r$ and $m$ are integers and $\Op(p)$ is a differential operator of the form 
\begin{align}
  \sum_{k=1}^m a_k(x)D^k
  \label{eq.adiffop}
\end{align}
with smooth functions $a_k$ $k=1,\ldots,m$ and $a_m$ bounded uniformly away from zero. If $1/\mathcal{F}(f_\epsilon)(-\cdot)$ is a polynomial of degree $r$ (which is true for instance if $\epsilon$ is Exponential, Laplace or Gamma distributed) then $\Op(\widetilde p)$ is again of the form \eqref{eq.adiffop} but with degree $m+r$ and hence $v_{t,h}(u)$ is essentially a linear combination of derivatives of $\phi$ evaluated at $(u-t)/h$.
However, these assumptions on the error density are far to restrictive. In the following paragraph we will show that even under more general conditions the approximating statistic has a very simple form.

\subsection*{Principal symbol} In order to perform our test, it is necessary to compute quantiles of the approximating statistic in Theorem \ref{thm.msc}. Since the approximating statistic has a relatively complex structure let us give conditions under which it can be simplified considerably. First, we impose a condition on the asymptotic behavior of the Fourier transform of the errors. Similar conditions have been studied by Fan \cite{Fan91b} and Bissantz et {\it al.} \cite{bis}. Recall that for any $\alpha, a \in \mathbb{R}, \ s\neq 0$, $D\iota_s^\alpha |s|^a=D(is)^{a_1}(-is)^{a_2}=ai \iota_s^{\alpha-1} |s|^{a-1}$ with $a_1=(a+\alpha)/2$ and $a_2=(a-\alpha)/2$.

\begin{ass}
\label{as.Ffepsilonapprox}
Suppose that there exist $\beta_0>1/2$, $\rho \in [0,4)$ and positive numbers $A, C_\epsilon$ such that
\begin{align*}
  \big|A\iota_s^\rho |s|^r\mathcal{F}(f_\epsilon)(s)-1 \big|
  +\big|A r^{-1} i\iota_s^{\rho+1} |s|^{r+1}D\mathcal{F}(f_\epsilon)(s)-1 \big|
  \leq C_\epsilon \langle s\rangle^{-\beta_0}, \ \forall s\in \mathbb{R}.
\end{align*}
\end{ass}

For instance the previous assumption holds with $A=\theta^r$ and $\rho\equiv r \mod 4$, if $f_\epsilon$ is the density of a $\Gamma(r,\theta)$ distributed random variable. In this case $\mathcal{F}(f_\epsilon)(s)=(1+i\theta s)^{-r}$.

\begin{ass}
 \label{as.prinsymb}
Given $m=\{0\}\cup [1,\infty)$, suppose there exists a decomposition $p=p_P+p_R$ such that $p_R\in \underline S^{m'}$ for some $m'<m$, and
\begin{align*}
  p_P(x,\xi)=a_P(x)|\xi|^m\iota_\xi^\mu , \quad \text{for all} \ x,\xi\in \mathbb{R}, 
\end{align*}
with $(x,\xi)\mapsto a_P(x) \in S^0$, $a_P$ real-valued and $|a_P(\cdot)|>0$.
\end{ass}

%Note that we can choose without loss of generality $A>0$, since $\iota_s^2=-1$. 

For $s\neq 0$, $\iota_s^2=-1$. Assume that in the special case $m=0$ we have $|\rho+\mu|\leq r$. Then, we can (and will) always choose $\rho$ and $\mu$ in Assumptions \ref{as.Ffepsilonapprox} and \ref{as.prinsymb} such that $\sigma=(r+m+\rho+\mu)/2$ and $\tau=(r+m-\rho-\mu)/2$ are non-negative. The symbol $p_P$ is called principal symbol. We will see that, together with the characteristics from the error density, it completely determines the asymptotics. The condition basically means that there is a smooth function $b$ such that the highest order of the pseudo-differential operator coincides with $a_P(x)D^m$. Note that principal symbols are usually defined in a slightly more general sense, however Assumption \ref{as.prinsymb} turns out to be appropriate for our purposes. In particular, the last assumption is verified for Examples \ref{exam.A=D}-\ref{exam.noisy_w_problem}.

In the following, we investigate the approximation of the multiscale test statistic
\begin{align}
  T_n^P:=\sup_{(t,h)\in B_n} w_h\left(\frac{h^{r+m-1/2} \ \big|T_{t,h}-\E [T_{t,h}]\big|}{\sqrt{\widehat g_n(t)} \  |A a_P(t)|\ \|D_+^{r+m} \phi\|_2}-\sqrt{2\log \tfrac \nu h}\right),
  \label{eq.TnPdef}
\end{align}
by
\begin{align*}
  T_n^{P,\infty}(W)
  :=
  \sup_{(t,h)\in B_n} w_h \left(
  \frac{\big|\int D_+^{\sigma}D_{-}^\tau \phi\big(\tfrac{s-t}h\big)dW_s \big|}{\|D_+^{r+m} \phi\big(\tfrac{\cdot-t}h\big)\|_2}-\sqrt{2\log \tfrac \nu h }\right).
\end{align*}

\begin{thm}
\label{thm.speciallimits}
Work under Assumptions \ref{ass.noise}, \ref{as.Ffepsilonapprox} and \ref{as.prinsymb}. Suppose further, that
\begin{itemize}
 \item[(i)] $l_n n \log^{-3}n \rightarrow \infty$ and $u_n=o(\log^{-(3\vee (m-m')^{-1})}n)$,
 \item[(ii)] $\phi \in H_3^{\lfloor r+m+5/2\rfloor}$, \ $\supp \phi \subset [0,1]$, and \ $\TV(D^{\lfloor r+m+5/2\rfloor} \phi)<\infty$,
 \item[(iii)] If $m=0$ assume that $r>1/2$ and $|\mu+\rho|\leq r$.
\end{itemize}
Then, there exists a (two-sided) standard Brownian motion $W$, such that for $\nu>e$,
\begin{align*}  
   \sup_{G\in \mathcal{G}_{c,C,q}} \Big |T_n^P- T_n^{P,\infty}(W)\Big | = o_P(1),
\end{align*}
and the approximating statistic $T_n^{P,\infty}(W)$ is almost surely bounded from above by 
\begin{align}  
  \sup_{(t,h)\in\mT}
  w_h \left(
  \frac{\big|\int D_+^{\sigma}D_{-}^\tau \phi\big(\tfrac{s-t}h\big) dW_s \big|}{\|D_+^{r+m} \phi\big(\tfrac{\cdot-t}h\big)\|_2}-\sqrt{2\log \tfrac \nu h }\right) \  <\infty, \quad \text{a.s.}
  \label{eq.limitbd3}
\end{align}
\end{thm}

% In difference to confidence bands, our qualitative analysis, outlined in the next section, does not require an estimate of the bias of the unknown function.

\medskip

\section{Confidence statements}
\label{sec.conf_statements}

\subsection{Confidence rectangles}

Suppose that Theorem \ref{thm.msc} holds. The distribution of $T_n^\infty(W)$ depends only on known quantities. By ignoring the $o_P(1)$ term on the right hand side of \eqref{eq.mainstate}, we can therefore simulate the distribution of $T_n$. To formulate it differently, the distance between the $(1-\alpha)$-quantiles of $T_n$ and $T_n^\infty(W)$ tends asymptotically to zero, although $T_n^\infty(W)$ does not need to have a weak limit. The $(1-\alpha)$-quantile of $T_n^\infty(W)$ will be denoted by $q_\alpha(T_n^\infty(W))$ in the sequel.

In order to obtain a confidence band one has to control the bias which requires a H\"older condition on $\op(p)f$. However, since we are more interested in a qualitative analysis, it suffices to assume that $\op(p)f$ is continuous (and $f\in H^m$ in order to define the scalar product of $\op(p)f$ properly). Moreover, instead of a moment condition on the kernel $\phi$, we require non-negativity, i.e. for the remaining part of this work, assume that $\phi\geq 0$ and $\int \phi(u) du=1$. Theorem \ref{thm.msc} implies that asymptotically with probability $1-\alpha$, for all $(t,h)\in B_n$,
\begin{align}
  \langle \phi_{t,h}, \op(p)f\rangle 
  &\in \Big[\frac{T_{t,h}-d_{t,h}}{\sqrt{n}}, \frac{T_{t,h}+d_{t,h}}{\sqrt{n}}\Big],
  \label{eq.interval_bounds}
\end{align}
where
\begin{align*}
  d_{t,h}&:= \sqrt{\widehat g_n(t)}\big\|v_{t,h}\big\|_2
  \sqrt{2\log \tfrac \nu h}
  \Big(1+q_\alpha(T_n^\infty(W))\frac{\log \log \tfrac \nu h}{\log \tfrac \nu h}
  \Big).
\end{align*}
Using the continuity of $\op(p)f$, it follows that asymptotically with confidence $1-\alpha$, for all $(t,h)\in B_n$, the graph of $x\mapsto \op(p)f(x)$ has a non-empty intersection with each of the rectangles
\begin{align}
  \big[t, t+h\big]\times \Big[\frac{T_{t,h}-d_{t,h}}{h\sqrt{n}}, \frac{T_{t,h}+d_{t,h}}{h\sqrt{n}}\Big].
  \label{eq.rectangle}
\end{align}
This means we find a solution of $(iii)$ by setting
\begin{align}
  b_-(t,h,\alpha):=\frac{T_{t,h}-d_{t,h}}{h\sqrt{n}}, \quad
  b_+(t,h,\alpha):=\frac{T_{t,h}+d_{t,h}}{h\sqrt{n}}.
  \label{eq.bset}
\end{align}

If instead Theorem \ref{thm.speciallimits} holds, we obtain by similar arguments that asymptotically with confidence $1-\alpha$, for all $(t,h)\in B_n$, the graph of $x\mapsto \op(p)f(x)$ has a non-empty intersection with each of the rectangles
\begin{align}
  \big[t, t+h]\times \Big[\frac{T_{t,h}-d_{t,h}^P}{h\sqrt{n}}, \frac{T_{t,h}+d_{t,h}^P}{h\sqrt{n}}\Big]
  \label{eq.rectangleP}
\end{align}
with 
\begin{align}
  d_{t,h}^P:=\sqrt{\widehat g_n(t)}  |A a_P(t)|h^{1/2-m-r} \big\|D_+^{r+m}\phi\big\|_2
  &\sqrt{2\log \tfrac \nu h} \notag \\
  &\cdot \Big(1+q_\alpha(T_n^{P,\infty}(W))\frac{\log \log \tfrac \nu h}{\log \tfrac \nu h}
  \Big)
  \label{eq.dthPdef}
\end{align}
and $q_\alpha(T_n^{P,\infty}(W))$ the $1-\alpha$-quantile of $T_n^{P,\infty}(W)$. Therefore we find a solution with
\begin{align*}
  b_-(t,h,\alpha):=\frac{T_{t,h}-d_{t,h}^P}{h\sqrt{n}}, \quad
  b_+(t,h,\alpha):=\frac{T_{t,h}+d_{t,h}^P}{h\sqrt{n}}.
\end{align*}

Finally let us mention that instead of rectangles we can also cover $\op(p)f$ by ellipses. Note that in particular a rectangle is an ellipse with respect to the $\|\cdot\|_\infty$ vector norm on $\mathbb{R}^2$, i.e. (up to translation) a set of the form $\{(x_1,x_2): \max(a|x_1|,b|x_2|)=1\}$ for positive $a,b$.

\subsection{Comparison with confidence bands} 

Let us shortly comment on the relation between confidence rectangles and confidence bands, which for density deconvolution were studied by Bissantz et {\it al.} \cite{bis} and  Lounici and Nickl \cite{lou}. Fix one scale $h=h_n$ and consider $B_n=[0,1]\times \{h\}$. For simplicity let us further restrict to the framework of Theorem \ref{thm.msc}. From \eqref{eq.interval_bounds}, we obtain that
\begin{align}
  t\mapsto \Big[\frac{T_{t,h}-d_{t,h}}{h\sqrt{n}}, \frac{T_{t,h}+d_{t,h}}{h\sqrt{n}}\Big]
  \label{eq.a_cb}
\end{align}
is a uniform $(1-\alpha)$-confidence band for the locally averaged function $t\mapsto \tfrac 1h \langle \phi_{t,h}, \op(p)f\rangle$. Restricting to scales on which the stochastic error dominates the bias $|\op(p)f-\tfrac 1h \langle \phi_{t,h}, \op(p)f\rangle|$ (for instance by slightly undersmoothing) we can, inflating \eqref{eq.a_cb} by a small amount, easily construct asymptotic confidence bands for $\op(p)f$ as well. Note that Theorem \ref{thm.msc} does not require that $s^r\mathcal{F}(f_\epsilon)(s)$ converges to a constant and therefore we can construct confidence bands for situations which are not covered within the framework of \cite{bis}. However, the construction of confidence bands described above will not work on scales where we oversmooth or if

\begin{wrapfigure}{r}{0.55\textwidth}
  \vspace{-10pt}
 \begin{centering}
  \includegraphics[scale=0.6]{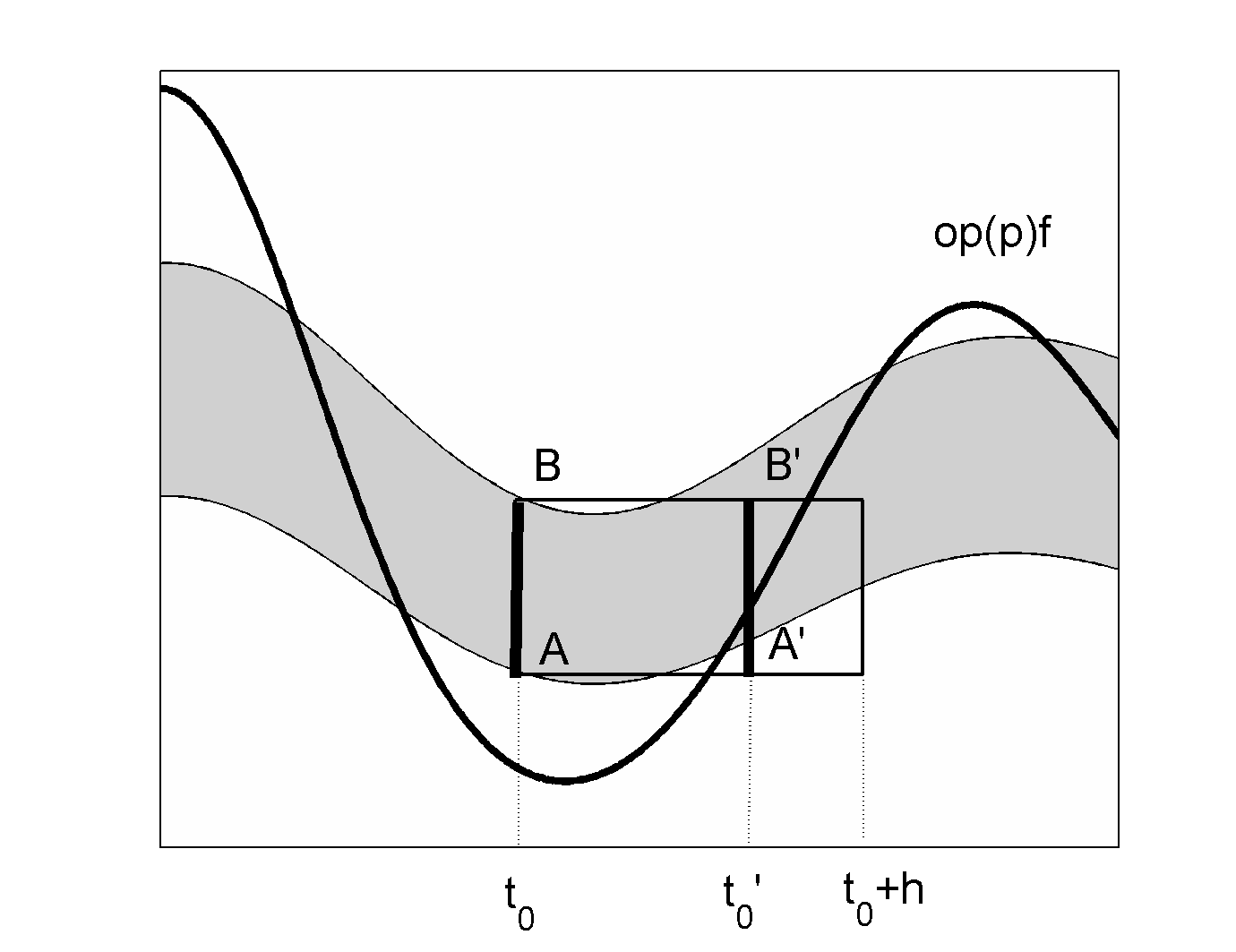}
  \end{centering}
  %\hspace{-1cm}
  \vspace{-10pt}
  \caption{Obtaining confidence rectangles from bands.}
  \label{fig.conf_rect_and_conf_bands}
  \vspace{-10pt}
\end{wrapfigure}
%Plot richtig dargestellt????

\noindent
bias and stochastic error are of the same order. The strength of the multiscale approach lies in the fact that for confidence rectangles all scales can be used simultaneously. This allows for another view on confidence rectangles. Figure \ref{fig.conf_rect_and_conf_bands} displays a band \eqref{eq.a_cb} computed for a large scale/bandwidth which obviously does not cover $\op(p)f$. Now, take a point, $t_0$ say, then \eqref{eq.rectangle} is equivalent to the existence of a point $t_0'\in[t_0,t_0+h]$ such that the confidence interval $[A, B]$ at $t_0$ shifted to $t_0'$ (and denoted by $[A',B']$ in Figure \ref{fig.conf_rect_and_conf_bands}) contains $\op(p)f(t_0')$. Thus, confidence rectangles also account for the uncertainty of $t\mapsto \op(p)f(t)$ along the $t$-axis.

\section{Choice of kernel and theoretical properties of the multiscale statistic}
\label{sec.performance}

In this section, we investigate the size/area of the rectangles constructed in the previous paragraphs. Recall that by \eqref{eq.Etthexp} the expectation of the statistic $T_{t,h}$ depends in general on $\op(p)$. In contrast, Theorem \ref{thm.speciallimits} shows that the variance of $T_{t,h}$ depends asymptotically only on the principal symbol, which acts on $\phi$ as a differentiation operator of order $m+r$. Therefore, the $m+r$-th derivative of $\phi$ appears in the approximating statistic $T_n^{P,\infty}(W)$, but no other derivative does. In fact, we shall see in this section that the scaling property of the confidence rectangles can be compared to the convergence rates appearing in estimation of the $(m+r)$-th derivative of a density.

\subsection{Optimal choice of the kernel} 

In what follows we are going to study the problem of finding an optimal function $\phi$. If $m+r\in \mathbb{N}$ and the confidence statements are formulated via the conclusions of Theorem \ref{thm.speciallimits}, this can be done explicitly.

Note that for given $(t,h)\in B_n$, the width of the rectangle \eqref{eq.rectangleP} is given by $2d_{t,h}^P/(h\sqrt{n})$. Further, the choice of $\phi$ influences the value of $d_{t,h}^P$ in two ways, namely by the factor $\big\|D_+^{r+m}\phi\big\|_2=\big\|D^{r+m}\phi\big\|_2$ as well as the quantile $q_\alpha(T_n^{P,\infty}(W))$ (cf.\ the definition of $d_{t,h}^P$ given in \eqref{eq.dthPdef}). Since $\alpha$ is fixed, we have
\begin{align*}
  q_\alpha(T_n^{P,\infty}(W))\frac{\log \log \tfrac \nu h}{\log \tfrac \nu h}=o(1).
\end{align*}
Therefore, $d_{t,h}^P$ depends in first order on $\big\|D^{r+m}\phi\big\|_2$ and our optimization problem can be reformulated as
\begin{align*}
  \text{minimize} \ \big\|D^{r+m}\phi\big\|_2, \quad \text{subject to} \ \int \phi(u)du=1. 
\end{align*}
This is in fact easy to solve if we additionally assume that $\phi\in H^q$ with $r+m\leq q<r+m+1/2$. By Lagrange calculus, we find that on $(0,1)$, $\phi$ has to be a polynomial of order $2m+2r$. Under the induced boundary conditions $\phi^{(k)}(0)=\phi^{(k)}(1)=0$ for $k=0,\ldots,r+m-1$, the solution $\phi_{m+r}$ is of the form
\begin{align}
  \phi_{m+r}(x)= c_{m+r}x^{m+r}(1-x)^{m+r}\mathbb{I}_{(0,1)}(x).
  \label{eq.phimrdef}
\end{align}
Due to the normalization constraint $\int \phi_{m+r}(u)du=1$, it follows that $\phi_{m+r}$ is the density of a beta distributed random variable with parameters $\alpha=m+r+1$ and $\beta=m+r+1$, implying, $c_{m+r}= (2m+2r+1)!/((m+r)!)^2$. It is  worth mentioning that $\phi_{m+r}^{(m+r)}$, restricted to the domain $[-1,1)$, is (up to translation/scaling) the $(m+r)$-th Legendre polynomial $L_{m+r}$, i.e. $$\phi_{m+r}^{(m+r)}=(-1)^{m+r} \frac{(2m+2r+1)!}{(m+r)!} L_{m+r}(2\cdot-1)$$ (this is essentially Rodrigues' representation, cf.\ Abramowitz and Stegun \cite{abr}, p. 785). For that reason, we even can compute
\begin{align*}
  \big\|\phi_{m+r}^{(m+r)}\big\|_{L^2}= \frac{(2m+2r)!}{(m+r)!} \sqrt{2m+2r+1}.
\end{align*}

In the particular case $r=0, \ m=1$ we obtain $\phi_1^{(1)}(x) \propto 1-2x$. This is known from the work of D\"umbgen and Walther \cite{due1} who considered locally most powerful tests to derive $\phi_1^{(1)}$.

To summarize, we can find the ``optimal'' kernel but it turns out that it has less smoothness than it is required by the conditions for Theorem \ref{thm.speciallimits} due to its behavior on the boundaries $\{0,1\}$. However, if the operator defining the shape constraint and the inversion operator $g\mapsto f$ are both differential operators (for an example see Section \ref{subs.example_Laplace_deconv}), then, the theorems can be proved under weaker assumptions on $\phi$ including as a special case the optimal beta kernels.

% Let us also mention that beta kernels are the natural minimizers of these type of problems. For instance the Epanechnikov kernel, which is in an $L^2$-sense the optimal kernel for density estimation (cf.\ Bartlett \cite{bar} and Epanechnikov \cite{epa}), is a (shifted) beta kernel.

\subsection{Theoretical properties of the method}

In this part, we give some theoretical insights. We start by investigating Problem $(iii)$ (cf.\ Section \ref{sec.mono}). After that, we will address issues related to $(ii)$ and $(i)$. It is easy to see that $\|v_{t,h}\|_2\lesssim h^{1/2-m-r}$ and thus, $d_{t,h}$ and $d_{t,h}^P$ are of the same order. We can therefore restrict ourselves in the following to the situation, where the confidence statements are constructed based on the approximation in Theorem \ref{thm.msc}. In the other case, similar results can be derived.

\medskip
 
{\it Problem (iii):} Recall that with confidence $1-\alpha$, for all $(t,h)\in B_n$,
\begin{align*}
  \graph(\op(p)f)\cap \big[t, t+h\big]\times \Big[\frac{T_{t,h}-d_{t,h}}{h\sqrt{n}}, \frac{T_{t,h}+d_{t,h}}{h\sqrt{n}}\Big]\neq \varnothing.
\end{align*}
The so constructed rectangles localize $\op(p)f$, where the amount of information is directly linked to the size of the rectangle. Therefore, it is natural to think of the length of the diagonal as a measure of localization quality. This length behaves like $h \vee h^{-m-r-1/2}n^{-1/2}\sqrt{\log 1/h}$. In particular, if the rectangle is a square, then, $h\sim (\log n/n)^{1/(3+2m+2r)}$ and this coincides with the optimal bandwidth for a kernel density estimator under a Lipschitz assumption on $f$. This is no surprise, of course, since Lipschitz continuity allows a function to oscillate over an interval $I$ by an amount that is proportional to the length $|I|$.  

\medskip

{\it Problem $(ii)$,  $(ii')$:} The following lemma gives a necessary condition in order to solve $(ii)$. Loosely speaking, it states that whenever 
\begin{align*}
  \op(p)f \big |_{[t,t+h]} \gtrsim n^{-1/2}h^{-m-r-1/2}\sqrt{\log 1/h},
\end{align*}
the multiscale test returns a rectangle $[t, t+h]\times [b_-(t,h,\alpha), b_+(t,h,\alpha)]$ which is in the upper half-plane with high-probability. Or, to state it differently, we can reject that $\op(p)f \big |_{[t,t+h]}<0$. 

In order to formulate the next theorem, recall the definition of $b_{\pm}(t,h,\alpha)$ given in \eqref{eq.bset}. Further, set $r_{t,h,n}:= 2d_{t,h}/(h\sqrt{n})$ and denote by $M_n^-$ the set of tupels $(t,h)\in B_n$ for which $\op(p)f \big |_{[t,t+h]}>r_{t,h,n}$. Similarly define $M_n^{+}:=\{(t,h)\in B_n \ | \ \op(p)f  |_{[t,t+h]} < -r_{t,h,n}\}$.
  
\begin{thm}
\label{thm.rates}
Work under the assumptions of Theorem \ref{thm.msc}. If $\phi\geq 0$, then
$$
	\lim_{n\rightarrow \infty}\P\Big((-1)^\mp b_\pm(t,h,\alpha)>0, \text{\ for all \ } (t,h)\in M_n^\pm\Big)\geq 1-\alpha.
$$
\end{thm}

\begin{proof}
For all $(t,h)\in M_n^-$, conditionally on the event given by \eqref{eq.interval_bounds},
\begin{align*}
  \op(p)f \big |_{[t,t+h]} > r_{t,h,n}
  \ &\Rightarrow \ \langle \phi_{t,h}, \op(p)f\rangle > hr_{t,h,n} \\
  \ &\Rightarrow \ T_{t,h}>d_{t,h} \ \Rightarrow \ b_-(t,h,\alpha)>0.
\end{align*}
One can argue similarly for $M_n^+$.
\end{proof}

Define
\begin{align}
 C_\alpha:=\big(\sqrt{8\|f_\epsilon\|_{\infty}}h^{m+r-1/2}\|v_{t,h}\|_{2}(1+q_\alpha(T_n^\infty(W)))\big)^{2/(2m+2r+1)}
  \label{eq.Cdef}
\end{align}
and let $\widetilde M^\pm$ be the set of tupels $(t,h) \in B_n$ satisfying the pair of constraints
\begin{align*}
  h\geq C_\alpha\left(\frac{\log n}n\right)^{1/(2\beta+2m+2r+1)}
\end{align*}
and 
\begin{align}
  \op(p)f \big |_{[t,t+h]} \lessgtr \left(\frac{\log n}n\right)^{\beta/(2\beta+2m+2r+1)}
  \label{eq.Af>}
\end{align}
(with $>$ in the last equality corresponding to $\widetilde M_n^-$ and $<$ to $\widetilde M_n^+$).

\begin{cor}
\label{cor.roc}
Work under the assumptions of Theorem \ref{thm.msc}. If $\phi\geq 0$ and $\beta\in \mathbb{R}$, then
$$
	\lim_{n\rightarrow \infty}\P\Big((-1)^\mp b_\pm(t,h,\alpha)>0,
	\text{\ for all \ } (t,h)\in \widetilde M_n^\pm\Big)\geq 1-\alpha.
$$
\end{cor}

\begin{proof}
It holds that
\begin{align*}
  d_{t,h}\leq \left\|f_\epsilon\right\|_\infty^{1/2}\big\|v_{t,h}\big\|_2 \sqrt{2\log \nu/h}\big(1+q_\alpha(T_n^\infty(W))\big).
\end{align*}
For sufficiently large $n$, $h\geq l_n\geq  \nu/n$. Therefore we have for every $(t,h)\in \widetilde M_n^-$,
\begin{align*}
  r_{t,h,n}\leq \sqrt{8\left\|f_\epsilon\right\|_\infty}\big\|v_{t,h}\big\|_2 \big(1+q_\alpha(T_n^\infty(W))\big)h^{-1/2}n^{-1/2}\sqrt{\log n}< \op(p)f \big |_{[t,t+h]}.
\end{align*}
Similar for $\widetilde M_n^+$. Since $\widetilde M_n^\pm \subset M_n^\pm$, the result follows directly from Theorem \ref{thm.rates}. 
\end{proof}

 Roughly speaking, the last result shows that if $h\sim (\log n/n)^{1/(2\beta+2m+2r+1)}$ and $\op(p)f \big |_{[t,t+h]} \sim (\log n/n)^{\beta/(2\beta+2m+2r+1)}=h^\beta$, then with probability $1-\alpha$, our method returns a rectangle in the upper half-plane. We have three distinct regimes
\begin{align*}
 \begin{array}{lll}
  \beta>0: \ & \op(p)f \big |_{[t,t+h]} \rightarrow 0, \\
  \beta=0: \ & \op(p)f \big |_{[t,t+h]} =O(1), \\
  -m-r-\tfrac 12<\beta<0: \ & \op(p)f \big |_{[t,t+h]} \rightarrow \infty.
 \end{array}
\end{align*}
It is insightful to compare the previous result to derivative estimation of a density if $m+r$ is a positive integer. As it is well known, $D^{m+r}f$ can be estimated with rate of convergence $$\Big(\frac{\log n}n\Big)^{\beta/(2\beta+2m+2r+1)}$$ under $L^\infty$-risk assuming that $\op(p)f$ is H\"older continuous with index $\beta>0$ and $h\sim (\log n/n)^{1/(2\beta+2m+2r+1)}$. This directly relates to the first case considered above.

\medskip
 
{\it Problem $(i)$:} At the beginning of Section \ref{sec.mono} we shortly addressed construction of confidence statements for the number of roots and their location. Note that estimators derived in this way, have many interesting features. On the one hand, we know that with probability $1-\alpha$ the estimated number of roots is a lower bound for the true number of roots. Therefore, these estimates do not come from a trade-off between bias and variance but they allow for a clear control on the probability to observe artefacts. In order to show that the lower bound for the number of roots is not trivial, we need to prove that whenever two roots are well-separated (for instance the distance between them shrinks not too fast), they will be detected eventually by our test. This property follows if we can show that the simultaneous confidence intervals for a fixed number of roots, say, shrink to zero.

Therefore, assume for simplicity that the number $K$ and the locations $(x_{0,j})_{j=1,\ldots,K}$ of the zeros of $\op(p)f$ are fixed (but unknown) and $x_{0,j}\in (0,1)$ for $j=1,\ldots,K$. For example, these roots can be extreme/saddle points if $\op(p)=D$ or points of inflection if $\op(p)=D^2$.

In order to formulate the result, we need that $B_n$ is sufficiently rich. Therefore, we assume that for all $n$, there exists a sequence $(N_n), \ N_n\gtrsim n^{1/(2m+2r+1)}\log^4 n $, such that
\begin{align*}
  \Big\{\Big(\frac k{N_n}, \frac l{N_n} \Big) \ \big | \ k=0,1,\ldots, \ l=1,2,\ldots, \  k+l\leq N_n\Big\} \subset B_n.
\end{align*}
Assume further that in a neighborhood of the roots $x_{0,j}$, $\op(p)f$ behaves  like
\begin{align*}
  \op(p)f(x)= \gamma \sign(x-x_{0,j})|x-x_{0,j}|^\beta + o(|x-x_{0,j}|^\beta),
\end{align*}
for some positive $\beta \in (0,1]$. Let $\rho_n=(\log n/n)^{1/(2\beta+2m+2r+1)}2/\gamma^{1/\beta}$ and $C_\alpha, M_n^\pm$ as defined in Corollary \ref{cor.roc}. There exist integer sequences $(k^-_{j,n})_{j,n}$, $(k^+_{j,n})_{j,n}$, $(l_{n})_{n}$ such that for all sufficiently large $n$,
$$
	\rho_n\leq \frac {k^-_{j,n}}{N_n}-x_{0,j}
	\leq 2\rho_n, \quad
	-2\rho_n\leq \frac{k^+_{j,n}}{N_n}-x_{0,j}\leq -\rho_n,
$$
and
$$
	C_\alpha\gamma^{1/\beta}\rho_n\leq \frac  {l_n}{N_n}\leq 2C_\alpha\gamma^{1/\beta}\rho_n.
$$
Direct calculations show $(k_{j,n}^-/N_n,l_n/N_n)\in M_n^-$ and $((k_{j,n}^+-l_n)/N_n,l_n/N_n)\in M_n^+$  for $j=1,\ldots,K$. We can conclude from Corollary \ref{cor.roc} and the construction that for $j=1,\ldots,K$, the confidence intervals have to be a subinterval of
\begin{align*}
  \left[\frac{k^+_{j,n}-l_n}{N_n}, \frac{k_{j,n}^-+l_n}{N_n}\right].
\end{align*}
Hence, the length for each confidence interval is bounded from above by
\begin{align*}
  4(C_\alpha \gamma^{1/\beta}+1)\rho_n 
  \sim
  \left(\frac {\log n}n\right)^{1/(2\beta+2m+2r+1)}.
\end{align*}
As $n\rightarrow\infty$ the confidence intervals shrink to zero, and will therefore become disjoint eventually. This shows that our estimator for the number of roots picks asymptotically the correct number with high probability. Observe, that for localization of modes in density estimation $(m,r,\beta)=(1,0,1)$ the rate $(\log n/n)^{1/5}$ is indeed optimal up to the log-factor (cf.\ Hasminskii \cite{has}). The rate $(\log n/n)^{1/7}$ for localization of inflection points in density estimation $(m,r,\beta)=(2,0,1)$ coincides with the one found in Davis {\it et al.}\ \cite{dav}.

For the special case of mode estimation in density deconvolution (here: $(m,r,\beta)=(1,r,1)$) let us shortly comment on related work by Rachdi and Sabre \cite{rac} and Wieczorek \cite{wie}. In \cite{wie} optimal estimation of the mode under relatively restrictive conditions on the smoothness of $f$ is considered. In contrast, Rachdi and Sabre find the same rates of convergence $n^{-1/(2r+5)}$ (but with respect to the mean-square error). Under the stronger assumption that $D^3 f$ exists they also provide confidence bands which converge at a different rate, of course.

\subsection{On calibration of multiscale statistics}
\label{subsec.cali}

Let us shortly comment on the type of multiscale statistic, derived in Theorems \ref{thm.gmsc}-\ref{thm.speciallimits}. Following \cite{due2}, p.139, we can view the calibration of the multiscale statistics \eqref{eq.defgenmsc}, (\ref{eq.Tndef}), and (\ref{eq.TnPdef}) as a generalization of L\'evy's modulus of continuity. In fact, the supremum is attained uniformly over different scales, making this calibration in particular attractive for construction of adaptive methods.

One of the restrictions of our method, compared to other works on multiscale statistics, is that we exclude the coarsest scales, i.e. $h> u_n=o(1)$ (cf.\ Theorem \ref{thm.gmsc}). Otherwise the approximating statistic would not be distribution-free. However, excluding the coarsest scales is a very weak restriction since the important features of $\op(p)f$ can be already detected at scales tending to zero with a certain rate. For instance in view of Corollary \ref{cor.roc}, the multiscale method detects a deviation from zero, i.e. $\op(p)f \big |_I\geq C>0$, provided the length of the interval $I$ is larger than $\mathrm{const.}\times (\log n/n)^{1/(2m+2r+1)}$. This can be also seen by numerical simulations, as outlined in the next section.

\section{Numerical simulations}
\label{sec.numsim}

\begin{figure}
\begin{center}
  \includegraphics[scale=0.6]{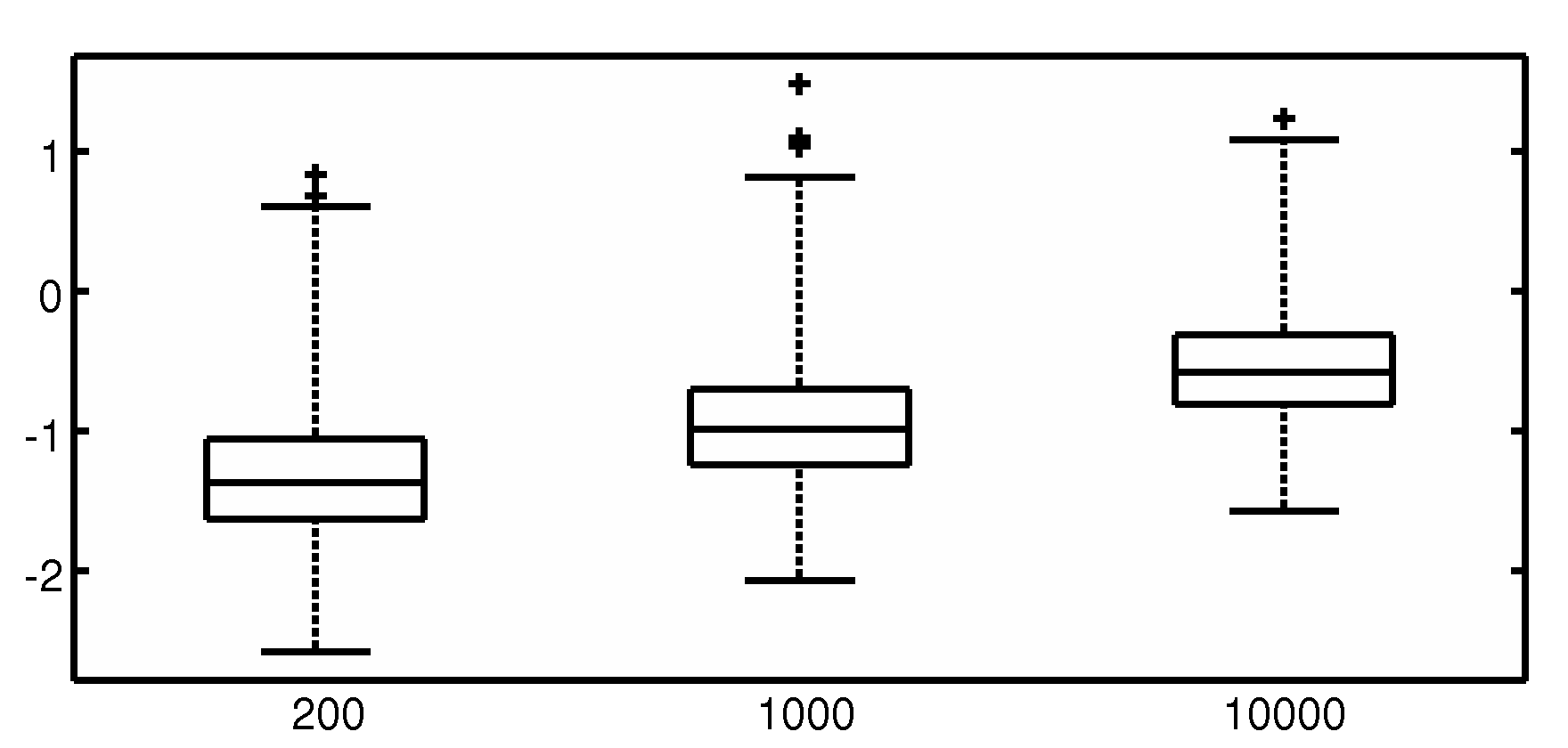}
\caption{Boxplots for three different values ($n=200, \ n=1000, \ n=10.000$) of the approximating statistic \eqref{eq.TnPinfLaplace}.}
\label{fig.box_plot}
\end{center}
\end{figure} 
 
In this section we provide further simulation results and discussion to the example from Section \ref{subs.example_Laplace_deconv} (cf. also Example \ref{exam.A=D}, Section \ref{sec.mono}), that is, studying monotonicity of the density $f$ under Laplace-deconvolution. More precisely, the error density is $f_\epsilon(x)=\theta^{-1}e^{-|x|/\theta}$ with $\theta=0.075$. In this case,
\begin{align*}
  \mathcal{F}(f_\epsilon)(t)= \langle \theta t\rangle^{-2} \quad \text{and} \quad \op(p)^\star f=-Df.
\end{align*}
One should notice that for Laplace deconvolution the inversion operator, mapping $g$ to $f$, is given by $1-\theta^2D^2$ and therefore the statistic \eqref{eq.Tthdef} takes the 
simple form \eqref{eq.def_Tth_Lapl_mono} (cf. also the discussion following Theorem \ref{thm.msc}). The ill-posedness of the shape constraint and the deconvolution problem give $m=1$, $r=2.$ Together with \eqref{eq.phimrdef} it is therefore natural to choose $\phi$ as the density of a $\Beta(4,4)$ random variable. Further, recall that $u_n=1/\log\log n$, $N_n=[n^{3/5}]$, and
\begin{align*}
  B_n=\Big\{\Big(\frac k{N_n}, \frac l{N_n} \Big) \ \big | \ k=0,1,\ldots, \ l=1,2,\ldots,[N_nu_n], \  k+l\leq N_n\Big\}. 
\end{align*}
Note that Assumptions \ref{as.Ffepsilonapprox} and \ref{as.prinsymb} hold for $(A,\rho,r,\beta_0)=(\theta^2,0,2,2)$ and $(\mu,m)=(1,1)$, respectively. Thus, we might work in the framework of Theorem \ref{thm.speciallimits}. The multiscale statistics
\begin{align*}
  T_n^P
  =
  \sup_{(t,h)\in B_n} w_h\left(\frac{|T_{t,h}-\E T_{t,h}|}{\sqrt{\widehat g_n(t)} \ \theta^2 \ \|\phi^{(3)}\|_2} - \sqrt{2\log\big(\tfrac{\nu}h\big)}\right)
\end{align*}
and
\begin{align}
  T_n^{P,\infty}(W)
  =
  \sup_{(t,h)\in B_n} w_h\left(\frac{\big|\int \phi^{(3)}\big(\tfrac{s-t}h\big)dW_s\big|}{\sqrt{h} \ \|\phi^{(3)}\|_2}-\sqrt{2\log\big(\tfrac {\nu}h\big)}\right)
  \label{eq.TnPinfLaplace}
\end{align}
have a particular simple form as well and the rectangles in \eqref{eq.rectangleP} can be computed via
\begin{align}
  d_{t,h}^P=h^{-5/2} \sqrt{\widehat g_n(t)} \theta^2 \|\phi^{(3)}\|_2 \sqrt{2\log \tfrac \nu h}\big(1+q_{\alpha,n} \tfrac {\log \log \tfrac \nu h}{\log \tfrac \nu h}\big).
  \label{eq.dthP_Laplace_explicit}
\end{align}
Boxplots for the distributions $T_{200}^{P,\infty}(W)$, $T_{1000}^{P,\infty}(W)$ and $T_{10.000}^{P,\infty}(W)$ are displayed in Figure \ref{fig.box_plot} based on $10.000$ repetitions each. The plot shows that the distribution is well-concentrated with a few outliers only. Although our theoretical results imply boundedness of the multiscale statistic as $n\rightarrow \infty$, Figure \ref{fig.box_plot} indicates that if $n$ is in the range of a few thousands $T_n^{P,\infty}(W)$ increases slowly.

\begin{figure}[hp]
\begin{center}
  \includegraphics[scale=0.8]{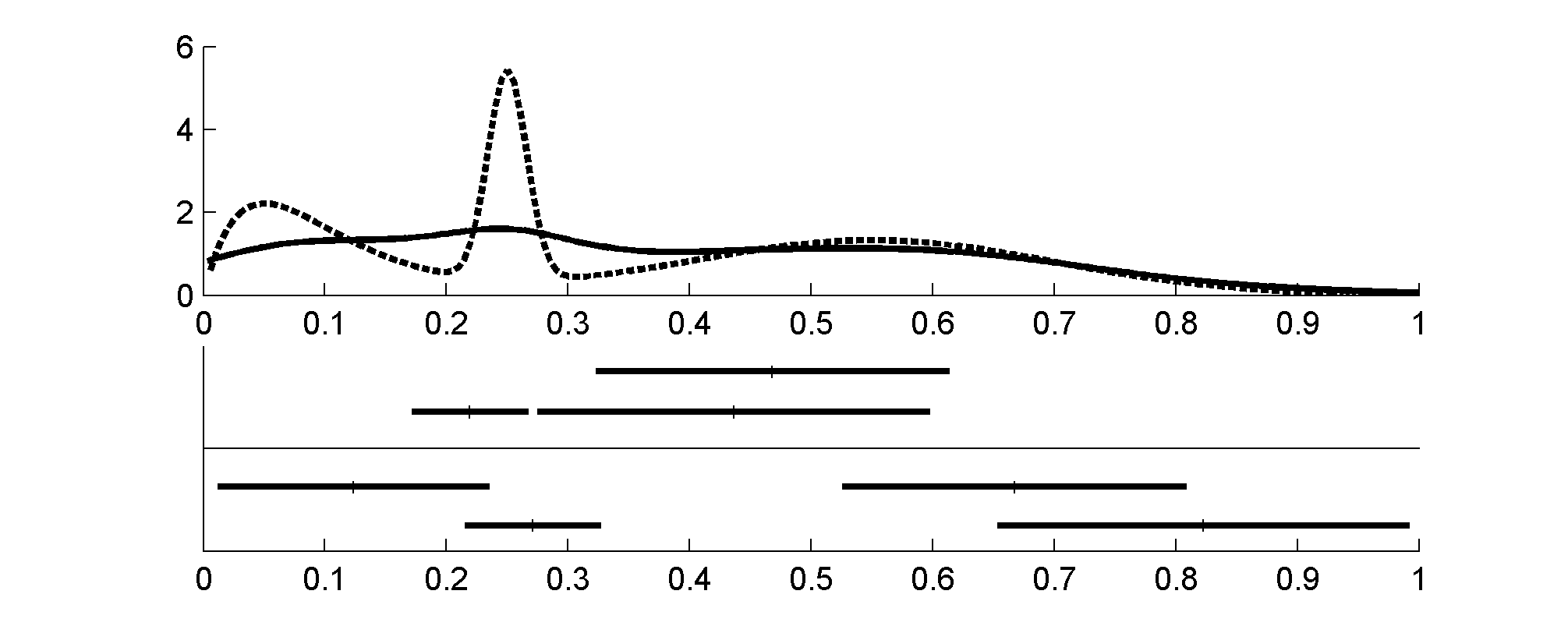}
\caption{Simulation for sample size $n=10.000$ and $90\%$-quantile. {\it Upper display:} True density $f$ (dashed) and convoluted density $g$ (solid). {\it Lower display:} Subset of minimal solutions to $(ii)$ and $(ii')$ (horizontal lines above/below the thin line)}
\label{fig.truedens}

    \includegraphics[scale=0.8]{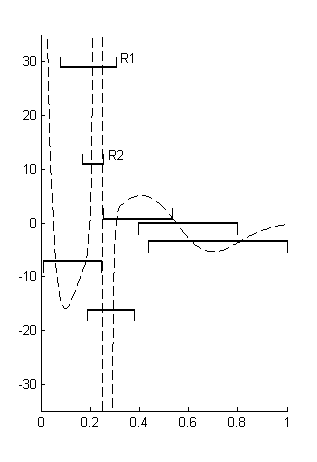}

  %\hspace{-1cm}
  \vspace{-10pt}
  \caption{True (unobserved) derivative $f'$ (dashed) and confidence statements for the level of $f'$. Computed for the same data set as in Figure \ref{fig.truedens}.}
  \label{fig.reconstr1000}

\end{center}
\end{figure} 

In Section \ref{subs.example_Laplace_deconv} we showed confidence statements for a simulated sample of size $n=2000$. To complement our study, let us now investigate the case of large $n$, i.e. $n=10.000$. Again we choose the confidence level equal to $90\%$. The estimated quantile is $q_{0.1}(T_{10.000}^{P,\infty}(W))=-0.04$. For all simulations, we use $\nu=\exp(e^2)$ because then, $h\mapsto \sqrt{\log \nu/h}/(\log\log \nu/h)$ is monotone as long as $0< h\leq 1$ (cf.\ Lemma \ref{lem.delta1_tech} (i)). The density $f$ has been designed in order to investigate Corollary \ref{cor.roc} numerically. Indeed, on $[0, 0.35]$ the signal $|f'|$ is large on average, but the intervals on which $f$ increases/decreases are comparably small. By way of contrast, on $[0.35, 1]$ the signal $|f'|$ is small and there is only one increase/decrease.

The test is able to find all increases and decreases of $f$ besides the increase on $[0,0.04]$, which is not detected (cf. Figure \ref{fig.truedens}). In contrast to the simulation in Figure \ref{fig.intro}, we see now a much better localization of the sharp increase/decrease on $[0.2,0.25]$ and $[0.25,0.3]$.

With the confidence rectangles at hand, we are able to say more about $f$ than localizing regions of increase/decrease only. In fact, we also can provide some confidence statements about the value of $f'$ close to a given point. Instead of plotting all confidence statements, we have displayed in Figure \ref{fig.reconstr1000} the most prominent ones, allowing for a good characterization of the derivative $f'$ and telling us something about the strength of the increases/decreases of $f$.

\noindent
A bracket of type  ``$\sqcup$'' means that $f'$ has to be above the horizontal line, somewhere. To give an example, from the bracket $R1$ we can conclude that at least on a subset of $[0.07, 0.3]$, the derivative $f'$ exceeds $29$. Similarly,  ``$\sqcap$'' means that somewhere $f'$ has to be below the corresponding horizontal line. As always, these statements hold simultaneously with confidence $90\%$.

% \begin{wrapfigure}{r}{0.55\textwidth}
%   \vspace{-5pt}
%  \begin{centering}
%     \includegraphics[scale=0.8]{Plots/num_sim_fin2_2v3.png}
%   \end{centering}
%   %\hspace{-1cm}
%   \vspace{-10pt}
%   \caption{True (unobserved) derivative $f'$ (dashed) and confidence statements for the level of $f'$. Computed for the same data set as in Figure \ref{fig.truedens}.}
%   \label{fig.reconstr1000}
%   \vspace{-5pt}
% \end{wrapfigure}

What we find is that in regions where the derivative does not oscillate much, we can achieve rather precise confidence statements about the value of $f'$. For example, from the rightmost bracket we can infer that with confidence $90\%$ the minimum of $f'$ on $[0.45, 1]$ has to be below $-4$, coming close to the true minimum, which is approximately $-6$.

Figure \ref{fig.reconstr1000} also shows nicely why a multiscale approach can provide additional insight compared to a one-scale method. Consider $R1$ and $R2$ in Figure \ref{fig.reconstr1000} and denote by $(t_1,h_1)$ and $(t_2,h_2)$ the corresponding indices in $B_n$ (as in \eqref{eq.rectangleP}). Note that $R1$ and $R2$ belong to similar time points in the sense that $R2 \subset R1$ but different bandwidths $h_1, h_2$. Therefore we may view $R1$ and $R2$ as a superposition of confidence statements on different scales. This allows to infer different qualitative and quantitative statements close to the same time point. We would use $R2$ in order to detect and localize an increase (as in Figure \ref{fig.truedens}) or to construct a confidence band for a mode, whereas from $R1$ we obtain a better lower bound for $\sup f'$. Thus, for a qualitative analysis there is a real gain by taking into account all scales simultaneously.

\section{Outlook and Discussion}

Given a density deconvolution model, we have investigated multiscale methods in order to analyze qualitative features of the unknown density which can be expressed as pseudo-differential operator inequalities. Compared to previous work, a more refined multiscale calibration has been considered using an idea of proof based on KMT results together with tools from the theory of pseudo-differential operators. We believe that the same strategy can be applied to a variety of other problems. In particular, it is to be expected that similar results will hold for regression and spectral density estimation.

In the formulation of the problem but also in the proofs it becomes apparent that modern tools from functional and harmonic analysis such as pseudo-differential operators are very helpful and to a certain extent unavoidable. In the same spirit, very recently, Nickl and Rei\ss \ \cite{nic} as well as S\"ohl and Trabs \cite{soe} used singular integral theory in order to prove Donsker theorems in deconvolution-type models. It is expected that reconsidering deconvolution theory from the viewpoint of harmonic analysis will lead to an improved understanding of the field.

Our multiscale approach allows us to identify intervals such that for given significance level we know that $\op(p)f>0$ at least on a subinterval. As outlined in Section \ref{sec.performance}, these results allow for qualitative inference as for example construction of confidence bands for the roots of $\op(p)f$. Since we only required that $\op(p)f$ is continuous, $\op(p)f$ can be highly oscillating. In this framework, it is therefore impossible to obtain strong confidence statements in the sense that we find intervals on which $\op(p)f$ is always positive. By adding bias controlling smoothness assumptions such as for instance H\"older conditions stronger results can be obtained resulting for instance in uniform confidence bands.

Obtaining multiscale results for error distributions as in Assumption \ref{ass.noise} is already a very difficult topic on its own and extension to the severely ill-posed case, including Gaussian deconvolution, becomes technically challenging since the theory of pseudo-differential operators has to the best of our knowledge not been formulated on the induced function spaces so far. Therefore we intend to treat this in a subsequent paper.

% 
% At the moment the proposed method is restricted to the class of blurring distributions introduced in Assumption \ref{ass.noise} and extension to the severely ill-posed case is not straightforward. 
% 
% 
% $r=\infty$ is not straightforward. Of particular interest is the case of Gaussian deconvolution. In this case the inversion formula is well known. It is basically the inverse Weierstrass transform (cf.\ Eddington \cite{edd}, Pollard, \cite{pol}, Widder \cite{wid}). Van Es and Kok \cite{es} derive some heuristic arguments indicating that the inversion formula of a Gaussian can be approximated by the inversion formula of scaled sums of Laplacian distributed random variables satisfying Assumption \ref{ass.noise}. 

Restricting to shape constraint which are associated with pseudo-differential operators appears to be a limitation of our method, since important shape constraints as for instance curvature cannot be handled within this framework and we may only work with linearizations (which is quite common in physics and engineering). Allowing for non-linearity is a very challenging task for further investigations. We are further aware of the fact that many other important qualitative features are related to integral transforms (that are in general not of convolution type) and they do not have a representation as pseudo-differential operator. For instance complete monotonicity and positive definiteness are by Bernstein's and Bochner's Theorem connected to the Laplace transform and Fourier transform, respectively. They cannot be handled with the methods proposed here and are subject to further research.

%\vspace{-0.2cm}

\bigskip

{\bf Acknowledgments.}
This research was supported by the joint research grant FOR 916 of the German Science Foundation (DFG) and the Swiss National Science Foundation (SNF). The first author was partly funded by DFG postdoctoral fellowship SCHM 2807/1-1. The second author would like to acknowledge support by DFG grants CRC 755 and CRC 803.  The authors are grateful for very helpful comments by Steve Marron, Markus Rei\ss \ , Jakob S\"ohl, Mathias Trabs, and G\"unther Walther as well as two referees and an associate editor which led to a more general version of previous results.

% \begin{supplement}[id=suppA]
%  \sname{Supplement}
%  \stitle{Multiscale Methods for Shape Constraints in Deconvolution: Confidence Statements for Qualitative Features}
%  \slink[url]{http://lib.stat.cmu.edu/aos/???/???}
%  \sdescription{All proofs can be found in the supplementary part, which contains additionally various lemmas, enumerated by $B.1, B.2,\ldots,C.1, C.2,\ldots$}
% \end{supplement}

\begin{appendices}

\section{Proofs of the main theorems}
\label{eq.secproofofmsc}
Throughout the appendix, let
\begin{align*}
  w_h=\frac{\sqrt{\tfrac 12\log \tfrac \nu h}}{\log\log\tfrac \nu h},
  \quad \widetilde w_h=\frac{\log \tfrac \nu h}{\log\log \tfrac \nu h}.
\end{align*}
Furthermore, we often use the normalized differential $\rqm \xi:=(2\pi)^{-1} d\xi$ 

% \begin{align*}
%   \delta_1(h):= \frac{\sqrt{\tfrac 12 \log \tfrac \nu h}}{\log \log \tfrac \nu h},
%   \quad\quad
%   \delta_2(h):= \frac{\log \tfrac \nu h}{\log\log \tfrac \nu h}.
% \end{align*}

\begin{proof}[Proof of Theorem \ref{thm.gmsc}] 
In a first step we study convergence of the statistic
\begin{align*}
  T_n^{(1)}=\sup_{(t,h)\in B_n}
  w_h\frac{\big|T_{t,h}-\E T_{t,h}\big|}{V_{t,h} \ \sqrt{g(t)}}-\widetilde w_h.
\end{align*}
Note that $T_n^{(1)}$ is the same as $T_n$, but $\widehat g_n$ is replaced by the true density $g$. We show that there exists a (two-sided) Brownian motion $W$, such that with
\begin{align*}
  T_n^{(2)}(W):=\sup_{(t,h)\in B_n} w_h\frac{\big| \int \psi_{t,h}(s) \sqrt{g(s)} dW_s\big|}{V_{t,h} \ \sqrt{g(t)}}-\widetilde w_h,
\end{align*}
we have
\begin{align}
  \sup_{G\in \mathcal{G}_{c,C,q}} \big |T_n^{(1)}-T_n^{(2)}(W)\big|=o_P(r_n).
  \label{eq.Tn1Tn2}
\end{align}
The main argument is based on the standard version of KMT (cf.\ \cite{KMT}). This is a fairly classical result, but has never been used to describe the asymptotic distribution of a multiscale statistic, the only exception being Walther \cite{wal}. In order to state the result, let us define a Brownian bridge on the index set $[0,1]$ as a centered Gaussian process $(B(f))_{\{f\in \mathcal{F}\}}, \ \mathcal{F}\subset L^2([0,1])$ with covariance structure
\begin{align*}
  \Cov\big(B(f),B(g)\big)= \langle f,g\rangle -\langle f,1 \rangle \langle g,1 \rangle.
\end{align*}
For $\mathcal{F}_0:=\{x\mapsto\mathbb{I}_{[0,s]}(x): s\in [0,1]\},$ the process $(B(f))_{\{f\in \mathcal{F}_0\}}$ coincides with the classical definition of a  Brownian bridge. If $U_i\sim \mathcal{U}[0,1]$, i.i.d., the uniform empirical process on the function class $\mathcal{F}$ is defined as
\begin{align*}
  \U_n(f)= \sqrt{n} \ \Big(\frac 1n \sum_{i=1}^n f(U_i)- \int f(x) dx \Big), \quad f\in \mathcal{F}.  
\end{align*}
In particular
\begin{align*}
  T_{t,h}-\E \, T_{t,h} = \U_n\big(\psi_{t,h} \circ G^{-1}\big),
\end{align*}
where $G^{-1}$ denotes the quantile function of $Y$. For convenience, we restate the celebrated KMT inequality for the uniform empirical process. %Hier FUnktionenklasse auf caglad geaendert, soll man noch zeigen, dass das Supremum stets das Gleiche ist????

\begin{thm}[{KMT on $[0,1]$}, cf.\ \cite{KMT}]
\label{thm.KMT01}
There exist versions of $\U_n$ and a Brownian bridge $B$ such that for all $x$
\begin{align*}
  \P\Big(\sup_{f\in \mathcal{F}_0} \big|\U_n(f)-B(f)\big|> n^{-1/2} (x+C\log n)\Big)<Ke^{-\lambda x},
\end{align*}
where $C, K, \lambda>0$ are universal constants.
\end{thm}

However, we need a functional version of KMT. We shall prove this by using the theorem above in combination with a result due to Koltchinskii \cite{kol}, (Theorem 11.4, p. 112) stating that the supremum over a function class $\mathcal{F}$ behaves as the supremum over the symmetric convex hull $\bsc(\mathcal{F})$, defined by
\begin{align*}
  \bsc(\mathcal{F}):= \Big\{\sum_{i=1}^\infty \lambda_i f_i: f_i\in \mathcal{F}, \lambda_i \in [-1,1], \sum_{i=1}^\infty |\lambda_i|\leq 1\Big\}.
\end{align*}

\begin{thm}
\label{thm.sconv}
Assume there exists a version $B$ of a Brownian bridge, such that for a sequence $(\widetilde\delta_n)_n$ tending to $0$,
\begin{align*}
  \P^*\Big(\sup_{f\in \mathcal{F}} |\U_n(f)-B(f)| \geq \widetilde\delta_n (x+C\log n)\Big)
  \leq Ke^{-\lambda x},
\end{align*}
where $C, K, \lambda>0$ are constants depending only on $\mathcal{F}$. Then, there exists a version $\widetilde B$ of a Brownian bridge, such that
\begin{align*}
  \P^*\Big(\sup_{f\in \bsc(\mathcal{F})} |\U_n(f)- \widetilde B(f)| \geq \widetilde\delta_n (x+C'\log n)\Big)
  \leq K'e^{-\lambda' x}
\end{align*}
for constants $C', K',\lambda'>0$.
\end{thm}

In Theorem \ref{thm.sconv}, $\P^\star$ refers to the outer measure, however, for the function class considered in this paper, we have measurability of the corresponding event and hence may replace $\P^\star$ by $\P$.  It is well-known (cf.\ Gin\'e {\it et al.}\ \cite{gin2}, p. 172) that
\begin{align}
  \big\{ \rho \ \big| \ \rho:\mathbb{R}\rightarrow \mathbb{R},\ \supp \rho \subset [0,1], \  \rho(1)=0, \ \TV(\rho)\leq 1\big\} \subset \bsc(\mathcal{F}_0).
  \label{eq.inclusion}
\end{align}
Now, assume that $\rho:\mathbb{R}\rightarrow \mathbb{R}$ is such that $\TV(\rho)+3|\rho(1)|\leq 1$. Define $\widetilde \rho=(\rho-\rho(1)\mathbb{I}_{[0,1]})/(1-|\rho(1)|)$ and observe that $\TV(\widetilde \rho)\leq 1$ and $\widetilde \rho(1)=0$. By \eqref{eq.inclusion} there exists $\lambda_1,\lambda_2,\ldots \in \mathbb{R}$ and $t_1,t_2,\ldots \in [0,1]$ such that $\widetilde \rho=\sum \lambda_i\mathbb{I}_{[0,t_i]}$ and $\sum |\lambda_i|\leq 1$. Therefore, $\rho=(1-|\rho(1)|)\widetilde \rho+\rho(1)\mathbb{I}_{[0,1]}$ can be written as linear combination of indicator functions, such that the sum of the absolute values of weights is bounded by $1$. This shows
\begin{align*}
  \big\{ \rho \ \big| \ \rho:\mathbb{R}\rightarrow \mathbb{R}, \ \supp \rho \subset [0,1], \  \TV(\rho)+3|\rho(1)|\leq 1\big\} \subset \bsc(\mathcal{F}_0).
\end{align*}
Since $\TV(\psi_{t,h}\circ G^{-1})\leq \TV(\psi_{t,h})$ it follows by Assumption \ref{as.testfcts} (ii) that the function class
\begin{align*}
  \mathcal{F}_n:= \Big\{ C_\star V_{t,h}^{-1} \sqrt{h} \ \psi_{t,h} \circ G^{-1} : (t,h)\in B_n, \ G \in \mathcal{G}_{c,C,q} \Big\}
\end{align*}
is a subset of $\bsc(\mathcal{F}_0)$ for sufficiently small constant $C_\star$. Combining Theorems \ref{thm.KMT01} and \ref{thm.sconv} shows for $\widetilde \delta_n=n^{-1/2}$ that there are constants $C',K',\lambda'$ and a Brownian bridge $(B(f))_{f\in \bsc(\mathcal{F}_0)}$ such that for $x>0$, the probability of
\begin{align*}
    \Big\{\sup_{(t,h)\in B_n, \ G \in \mathcal{G}}
    C_\star  \frac{\sqrt h \big|\U_n\big(\psi_{t,h}\circ G^{-1}\big)-B\big(\psi_{t,h}\circ G^{-1}\big)\big|}{V_{t,h}} \geq \frac1 {\sqrt{n}}(x+C'\log n)\Big\}
\end{align*}
is bounded by $K'e^{-\lambda'x}$. Due to Lemma \ref{lem.delta1_tech} (i) and $l_n\geq \nu/n$ for sufficiently large $n$, we have that $w_{l_n}\leq w_{\nu/n}$. This readily implies with $x= \log n$ that
\begin{align*}
  \sup_{(t,h)\in B_n, \ G \in \mathcal{G}}
  w_h \frac{\Big|\big|T_{t,h}-\E\,T_{t,h}\big|
  - \big|B\big(\psi_{t,h}\circ G^{-1}\big)\big|\Big|}{V_{t,h} \ \sqrt{g(t)}}
   =O_P\Big(\frac 1{\sqrt{l_n n}} w_{\nu/n}\log n\Big).&
\end{align*}
Now, let us introduce the (general) Brownian motion $W(f)$ as a centered Gaussian process with covariance $\E[W(f)W(g)]=\langle f, g\rangle$. In particular, $W(f)=B(f)+(\int f) \xi$, $\xi \sim \mathcal{N}(0,1)$ and independent of $B$, defines a Brownian motion and hence there exists a version of $(W(f))_{f\in \bsc(\mathcal{F}_0)}$ such that $B(f)=W(f)-(\int f) W(1)$. We have
\begin{align*}
  \sup_{(t,h)\in B_n, \ G \in \mathcal{G}}
  w_h\frac{\big|\int \psi_{t,h} (u) \ dG(u)\big|}{V_{t,h} \ \sqrt{g(t)}}
  &\leq c^{-1}
  \sup_{(t,h)\in B_n, \ G \in \mathcal{G}}
  w_h\frac{\|\psi_{t,h}\|_1}{V_{t,h} \ \sqrt{g(t)}} \\
  &\lesssim \sup_{h\in[l_n,u_n]} w_h h^{1/2}\leq w_{u_n} u_n^{1/2},
\end{align*}
where the second inequality follows from Assumption \ref{as.testfcts} (ii) and the last inequality from Lemma \ref{lem.delta1_tech} (ii). This implies further
\begin{align*}
  \E\Big[\Big\|\frac{w_h}{V_{t,h} \ \sqrt{g(t)}}
  \Big[ \big|B\big(\psi_{t,h}\circ G^{-1}\big)\big|
  - \big|W\big(\psi_{t,h}\circ G^{-1}\big)\big|\Big]\Big\|_{\mathcal{F}_{n}}\Big]
  = O(w_{u_n}u_n^{1/2}),
\end{align*}
and therefore
\begin{align*}
  &\sup_{G \in \mathcal{G}} \Big|T_n^{(1)} -
  \sup_{(t,h)\in B_n}
  w_h\frac{\big|W\big(\psi_{t,h}\circ G^{-1}\big)\big|}{V_{t,h} \ \sqrt{g(t)}}-\widetilde w_h \Big|
  = O_P(\frac{w_{1/n} \log n}{\sqrt{l_n n}} + w_{u_n}u_n^{1/2}),
\end{align*}
and
\begin{align*}
  \sup_{G \in \mathcal{G}} \Big|T_n^{(1)} 
 -T_n^{(2)}(W) \Big|=O_P(l_n^{-1/2}n^{-1/2}w_{1/n}\log n + w_{u_n}u_n^{1/2}).
\end{align*}
In the last equality we used that $(W^{(1)}_t)_{t\in [0,1]}=(W(\mathbb{I}_{[0,t]}(\cdot)))_{t\in [0,1]}$ and $$(W_t)_{t\in \mathbb{R}}=\Big(\int_0^t \frac{\mathbb{I}_{\{g>0\}}(s)}{\sqrt{g(s)}}dW_{G(s)}^{(1)}\Big)_{t\in \mathbb{R}}$$ are (two-sided) standard Brownian motions, proving $W(\psi_{t,h}\circ G^{-1})= \linebreak \int \psi_{t,h}(s) \sqrt{g(s)} dW_s$ and hence \eqref{eq.Tn1Tn2}. Further note that Assumption \ref{as.testfcts} (iii) together with Lemma \ref{lem.empprocbd} shows that
\begin{align*}
  \sup_{G \in \mathcal{G}} \Big|T_n^{(2)}(W) -
  \sup_{(t,h)\in B_n}
  w_h
  \frac{\big|\int \psi_{t,h} (s) dW_s \big|}{V_{t,h}}-\widetilde w_h
  \Big|
  =O_P(\kappa_n).
%   \label{eq.Tn1DtoTNM}
\end{align*}
In a final step let us show that \eqref{eq.limitbd1} is almost surely bounded. In order to establish the result, we use Theorem 6.1 and Remark 1 of D\"umbgen and Spokoiny \cite{due2}. We set $\rho\big((t,h),(t',h')\big)=(|t-t'|+|h-h'|)^{1/2}$. Further, let $X(t,h)=\sqrt h V_{t,h}^{-1} \int \psi_{t,h}(s)dW_s$ and $\sigma(t,h)=h^{1/2}$.

By assumption, $X$ has continuous sample paths on $\mathcal{T}$ and obviously, for all $(t,h),(t',h')\in \mathcal{T}$,
\begin{align*}
	\sigma^2(t,h) \le \sigma^2(t',h')+\rho^2((t,h),(t',h')).
\end{align*}
Let $Z\sim \mathcal{N}(0,1)$. Since $X(t,h)$ is a Gaussian process and $V_{t,h}\geq \|\phi_{t,h}\|_2$, $\P(X(t,h)>\sigma(t,h)\eta)\leq \P(Z>\eta)\leq \exp(-\eta^2/2)$ for any $\eta>0$. Further, denote by
\begin{align}
  A_{t,t',h,h'}:= \left\|\frac{\psi_{t,h}\sqrt h}{V_{t,h}}
  -\frac{\psi_{t',h'}\sqrt h'}{V_{t',h'}}\right\|_2.
  \label{eq.defAtthh}
\end{align}
Because of $\P(|X(t,h)-X(t',h')\big|\geq A_{t,t',h,h'}\eta\Big)\leq 2\exp\big(-\eta^2/2\big)$ we have by Lemma \ref{lem.Atthh} for a universal constant $K>0$,
\begin{align*}
  \P\Big(\big|X(t,h)-X(t',h')\big|\geq
  \rho((t,h),(t',h')) \eta\Big)\leq 2\exp\big(-\eta^2/(2K^2)\big).
\end{align*}
Finally, we can bound the entropy $\mathcal{N}((\delta u)^{1/2}, \{(t,h)\in \mathcal{T} :h\leq \delta\})$ similarly as in \cite{due2}, p.\ 145. Therefore, application of Remark 1 in \cite{due2} shows that
\begin{align*}
	S := \sup_{(t,h) \in \mathcal{T}}
  \frac{\sqrt{\tfrac 12 \log \tfrac eh} \ \big|\int \psi_{t,h}(s) dW_s \big|}{\log \big(e\log \tfrac eh\big) \ V_{t,h} }-\frac{\sqrt{\log(\tfrac 1h)\log(\tfrac eh)}}{\log\big(e \log \tfrac eh\big)}
\end{align*}
is almost surely bounded from above. Define
\begin{align*}
 S':=\sup_{(t,h)\in \mathcal{T}}
  \frac{\sqrt{\tfrac 12 \log \tfrac \nu h} \ \big|\int \psi_{t,h}(s) dW_s \big|}{\log \log \tfrac \nu h \ V_{t,h}}-\frac{\sqrt{\log(\tfrac 1h)\log(\tfrac \nu h)}}{\log \log \tfrac \nu h }.
\end{align*}
If $e<\nu\leq e^e$, then 
\begin{align*}
  \log\log \tfrac \nu h=\log\Big(\tfrac{\log \nu}e\log \tfrac {e^e}{h^{e/\log \nu}}\Big)
  \geq \log\log \nu -1+\log\big(e\log \tfrac eh\big)
\end{align*}
implies
\begin{align*}
  \frac{\log\big(e\log\tfrac eh\big)}{\log\log \tfrac \nu h}\leq \frac 1{\log\log \nu} +1.
\end{align*}
Furthermore, $\log \nu/h\leq (\log \nu) (\log e/h)$. Suppose now that $S'>0$ (otherwise $S'$ is bounded from below by $0$). Then, $S'\lesssim S$ and hence $S'$ is almost surely bounded. Finally,
\begin{align*}
  \sqrt{\log \tfrac \nu h}\big | \sqrt{\log\tfrac 1h}-\sqrt{\log\tfrac \nu h}\big |\leq \log \nu.
\end{align*}
Therefore, \eqref{eq.limitbd1} holds, i.e.
\begin{align*}
 \sup_{(t,h)\in \mT}
  w_h\frac{\big|\int \psi_{t,h}(s) dW_s \big|}{V_{t,h}}-\widetilde w_h
\end{align*}
is almost surely bounded.

In the last step, it remains to prove that $\sup_{G\in \mathcal{G}_{c,C,q}}|T_n-T_n^{(1)}|=O_P(\sup_{G\in \mathcal{G}} \|\widehat g_n-g\|_\infty \log n/\log\log n)$. For sufficiently large $n$ and because $G \in \mathcal{G}$, $\widehat g_n\geq c/2$ for all $t\in [0,1]$. Therefore using Lemma \ref{lem.delta1_tech} (i),
\begin{align}
  \sup_{G\in \mathcal{G}} \big|T_n-T_n^{(1)}|
  &\leq 
  \sup_{(t,h)\in B_n, \ G\in \mathcal{G}}w_h \frac{\big|T_{t,h}-\E [T_{t,h}]\big|}{V_{t,h} \ \sqrt{g(t)}}
  \frac{\sup_{G\in \mathcal{G}}\big\|\widehat g_n- g\big\|_\infty}{\widehat g_n(t)} \notag \\
  &\leq \frac{2\sup_{G\in \mathcal{G}}\big\|\widehat g_n- g\big\|_\infty}{c}
  \sup_{(t,h)\in B_n, \ G\in \mathcal{G}} w_h \frac{\big|T_{t,h}-\E [T_{t,h}]\big|}{V_{t,h} \ \sqrt{g(t)}} \notag\\
  &\leq \frac{2\sup_{G\in \mathcal{G}}\big\|\widehat g_n- g\big\|_\infty}{c}(T_n^{(1)}+\sup_{h\in [l_n,u_n]} \widetilde w_h ) \notag\\
  &\leq \frac{2\sup_{G\in \mathcal{G}}\big\|\widehat g_n- g\big\|_\infty}{c}\big(T_n^{(1)}+O(\frac{\log n}{\log \log n})\big).
  \label{eq.TnTn1diff}
\end{align}
Since $T_n^{(1)}$ is a.s. bounded by Theorem \ref{thm.gmsc}, the result follows.
\end{proof}

\begin{rem}
\label{rem.thmmsc}
Next, we give a proof of Theorem \ref{thm.msc}. In fact we proof a slightly stronger version, which does not necessarily require the symbol $a$ to be elliptic and $V_{t,h}=\|v_{t,h}\|_2$. It is only assumed that 
\begin{itemize}
 \item[(i)]  $V_{t,h}\geq \|v_{t,h}\|_2$,
 \item[(ii)] there exists constants $c_V, C_V$ with $0<c_V\leq h^{m+r-1/2}V_{t,h}\leq C_V<\infty$
 \item[(iii)] for all $(t,h), (t',h')\in \mathcal{T}$ and whenever $h\leq h'$ it holds that $h^{m+r}|V_{t,h}-V_{t',h'}|\leq C_V(|t-t'|+|h-h'|)^{1/2}$.
\end{itemize}
As a special
case these conditions are satisfied for $V_{t,h}=\|v_{t,h}\|_2$ and $\op(a)$ elliptic. This follows directly from Lemmas \ref{lem.vthL2} and \ref{lem.vthvt'h'}.
\end{rem}

\begin{proof}[Proof of Theorem \ref{thm.msc}]
In order to prove the statements it is sufficient to check the conditions of Theorem \ref{thm.gmsc}. For $h>0$ define the symbol 
\begin{align}
 a_{t,h}^\star(x,\xi):=h^{\om} a^\star(xh+t,h^{-1}\xi).
  \label{eq.defathstar}
\end{align}
Under the imposed conditions and by Remark \ref{rem.vth=Kth} we may apply Lemma \ref{lem.Kthestimates} for $\mfp=a_{t,h}^\star$ and therefore, uniformly over $(t,h)\in \mT$ and $u,u'\in \mathbb{R}$,
\begin{itemize}
 \item[(I)] $|v_{t,h}(u)|\lesssim h^{-m-r}\min\big(1,\tfrac {h^2}{(u-t)^2}\big)$.
 \item[(II)] $|v_{t,h}(u)-v_{t,h}(u')|\lesssim h^{-m-r-1}|u-u'|$ and if $u,u'\neq t$,
$$
	|v_{t,h}(u)-v_{t,h}(u')|\lesssim h^{1-m-r}\frac{|u-u'|}{|u'-t| \ |u-t|}=h^{1-m-r}\big|\int_{u'}^u \frac 1{(x-t)^2} dx\big|.
$$
\end{itemize}

Using (I), we obtain $\|v_{t,h}\|_\infty\lesssim h^{-m-r}$ and $\|v_{t,h}\|_1\lesssim h^{1-m-r}$. In order to show that the total variation is of the right order, let us decompose $v_{t,h}$ further into $v_{t,h}^{(1)}=v_{t,h}\mathbb{I}_{[t-h,t+h]}$ and $v_{t,h}^{(2)}=v_{t,h}-v_{t,h}^{(1)}$. By (II), $\TV(v_{t,h}^{(1)})\lesssim h^{-m-r}$ and
\begin{align*}
  \TV(v_{t,h}^{(2)})\lesssim h^{-m-r}+h^{1-m-r} \int_{t+h}^\infty \frac 1{(x-t)^2} dx\lesssim h^{-m-r}. 
\end{align*}
Since $\TV(v_{t,h})\leq \TV(v_{t,h}^{(1)})+\TV(v_{t,h}^{(2)})\lesssim h^{-m-r}$, this shows together with Remark \ref{rem.thmmsc} that part (ii) of Assumption \ref{as.testfcts} is satisfied. 

Next, we verify Assumption \ref{as.testfcts}, (iii) with $\kappa_n= \sup_{(t,h)\in B_n} w_hh^{1/2}\log(1/h) \lesssim u_n^{1/2} \log^{3/2} n$ (cf. Lemma \ref{lem.delta1_tech}, (ii)), i.e. we show
\begin{align*}
  \sup_{(t,h)\in B_n, \ G\in \mathcal{G}}
  w_h \frac{\TV\big(v_{t,h}(\cdot) [\sqrt{g(\cdot)}-\sqrt{g(t)}] \langle \cdot \rangle^\alpha \big)}{V_{t,h}} \lesssim u_n^{1/2} \log^{3/2} n.
\end{align*}
By Lemma \ref{lem.condTVreplace}, we see that this holds for $v_{t,h}$ replaced by $v_{t,h}^{(1)}$. Therefore, it remains to prove the statement for $v_{t,h}^{(2)}$.
% 
% In the next step we show that $\TV(v_{t,h}(\cdot)[\sqrt{g(\cdot)}-\sqrt{g(t)}])\lesssim h^{1-m-r}\log n$ uniformly over $(t,h)\in B_n, G\in \mathcal{G}$, i.e. we verify Assumption \ref{as.testfcts}, (iii) with $\kappa_n=u_n^{1/2}\log^{3/2}n$ (cf. Lemma \ref{lem.delta1_tech}, (ii)). By Lemma \ref{lem.condTVreplace},
% \begin{align}
%   \TV\big(v_{t,h}^{(1)}(\cdot)[\sqrt{g(\cdot)}-\sqrt{g(t)}]\big)\lesssim h^{1-m-r}.
%   \label{eq.TVvth1}
% \end{align}
Let us decompose $v_{t,h}^{(2)}$ further into $v_{t,h}^{(2,1)}=v_{t,h}\mathbb{I}_{[t-1,t+1]\cap [t-h,t+h]^c}$ and $v_{t,h}^{(2,2)}=v_{t,h}^{(2)}- v_{t,h}^{(2,1)}=v_{t,h}\mathbb{I}_{[t-1,t+1]^c}$. For the remaining part, let $u,u'$ be such that $|u-t|\geq |u'-t|\geq h$. We have
\begin{align}
  \TV\big(v_{t,h}^{(2,1)}(\cdot)
  \big[\sqrt{g(\cdot)}-\sqrt{g(t)}\big]\langle \cdot \rangle^\alpha\big)
  &\lesssim
  \big\|v_{t,h}^{(2,1)}(\cdot)
  \big[\sqrt{g(\cdot)}-\sqrt{g(t)}\big]\big\|_\infty \notag \\
  &\quad +\TV\big(v_{t,h}^{(2,1)}(\cdot)
  \big[\sqrt{g(\cdot)}-\sqrt{g(t)}\big]\big).
  \label{eq.TVvth21bd}
\end{align}
Using (I) and (II) together with the properties of the class $\mathcal{G}$ we can bound the variation $\big|v_{t,h}^{(2,1)}(u)
 \big[\sqrt{g(u)}-\sqrt{g(t)}\big]-v_{t,h}^{(2,1)}(u') \big[\sqrt{g(u')}-\sqrt{g(t)}\big]\big|$ by
\begin{align*}
  &\big|v_{t,h}^{(2,1)}(u)-v_{t,h}^{(2,1)}(u')\big|
  \cdot \big|\sqrt{g(u')}-\sqrt{g(t)}\big|
  +\big|v_{t,h}^{(2,1)}(u)\big|
  \cdot \big|\sqrt{g(u)}-\sqrt{g(u')}\big| \\
  &\lesssim h^{1-m-r}\tfrac{|u-u'|}{|u-t|}+h^{2-m-r}\tfrac{|u-u'|}{|u-t|^2}
  \lesssim h^{1-m-r}\tfrac{|u-u'|}{|u-t|}
  \leq h^{1-m-r} \big|\int_{u'}^u \tfrac 1{|x-t|} dx\big|.
\end{align*}
Due to $h\geq l_n\gtrsim 1/n$ this yields
\begin{align*}
  \TV\big(v_{t,h}^{(2,1)}(\cdot)[\sqrt{g(\cdot)}-\sqrt{g(t)}]\big)
  &\lesssim h^{1-m-r}+h^{1-m-r}\int_{t+h}^{t+1} \frac {du}{|u-t|} \\
  &\lesssim h^{1-m-r}\log \tfrac 1h\lesssim h^{1-m-r}\log n
\end{align*}
and with \eqref{eq.TVvth21bd} also
\begin{align}
  \TV\big(v_{t,h}^{(2,1)}(\cdot)
  \big[\sqrt{g(\cdot)}-\sqrt{g(t)}\big]\langle \cdot \rangle^\alpha\big)
  \lesssim  h^{1-m-r}\log n.
  \label{eq.TVvth21}
\end{align}
Finally, let us address the total variation term involving $v_{t,h}^{(2,2)}$. Given $\mathcal{G}_{c,C,q}$ we can choose $\alpha$ such that $\alpha>1/2$ and $\alpha +q<1$ (recall that $0\leq q<1/2$). By Lemma \ref{lem.rthlem}, we find that 
\begin{align*}
  \big|v_{t,h}^{(2,2)}(u)
  \langle u \rangle^\alpha
  -v_{t,h}^{(2,2)}(u')\langle u' \rangle^\alpha\big|
  &\lesssim
  h^{1-m-r}\Big|\int _{u'}^u \frac1{(x-t)^{2-\alpha}}+\frac 1{(x-t)^2} dx\Big|.
\end{align*}
Moreover
\begin{align*}
  \langle u\rangle ^\alpha (1+|u'|+|u|)^{q}
  &\leq (1+|u'|+|u|)^{q+\alpha} \\ &\leq (3+2|u-t|)^{q+\alpha}\leq 3+2|u-t|^{q+\alpha}
\end{align*}
and thus
\begin{align*}
 \big|v_{t,h}^{(2,2)}(u)\langle u \rangle^\alpha \big|
  \  \big|\sqrt{g(u)}-\sqrt{g(u')} \big|
  \lesssim 
  h^{2-m-r}\frac{|u-t|^{q+\alpha}+1}{|u-t|^2}|u-u'|.
\end{align*}
This allows us to bound the variation by
\begin{align*}
  &\big|v_{t,h}^{(2,2)}(u)
  \big[\sqrt{g(u)}-\sqrt{g(t)}\big]\langle u \rangle^\alpha
  -v_{t,h}^{(2,2)}(u')
  \big[\sqrt{g(u')}-\sqrt{g(t)}\big]\langle u' \rangle^\alpha\big| \notag \\
  &
  \leq \big|v_{t,h}^{(2,2)}(u)\langle u \rangle^\alpha \big|
  \  \big|\sqrt{g(u)}-\sqrt{g(u')} \big|
  +\frac 2{\sqrt{c}}\big|v_{t,h}^{(2,2)}(u)
  \langle u \rangle^\alpha
  -v_{t,h}^{(2,2)}(u')\langle u' \rangle^\alpha\big| \notag \\
  &
  \lesssim h^{1-m-r}\Big|\int_{u'}^u \frac1{(x-t)^{2-q-\alpha}}+\frac 1{(x-t)^{2-\alpha}}
  +\frac 1{(x-t)^2} \ dx \ \Big|
\end{align*}
and therefore we conclude that
\begin{align*}
  &\TV\big(v_{t,h}^{(2,2)}(\cdot)
  \big[\sqrt{g(\cdot)}-\sqrt{g(t)}\big]\langle \cdot \rangle^\alpha\big) \\
  &\lesssim h^{1-m-r}+
  h^{1-m-r}\int_{t+1}^\infty
  \frac{1}{(x-t)^{2-q-\alpha}}+\frac 1{(x-t)^{2-\alpha}}+\frac 1{(x-t)^2} dx \\
  &\leq h^{1-m-r}.
\end{align*}
Together with the bound for $v_{t,h}^{(1)}$ and \eqref{eq.TVvth21} this yields Assumption \ref{as.testfcts}, (iii).

Finally, Assumption \ref{as.testfcts} (iv) follows from Lemma \ref{lem.vthvt'h'} and Remark \ref{rem.thmmsc} due to $\phi \in H^{\lceil r+m\rceil}\cap H^{r+m+1/2}$, $\supp \phi \subset [0,1]$ and $\phi \in \TV(D^{\lceil r+m\rceil}\phi)<\infty$. This shows that Assumption \ref{as.testfcts} holds for $(v_{t,h},V_{t,h})$.

In the next step, we verify that $(t,h)\mapsto X(t,h)=\sqrt{h}V_{t,h}^{-1}\int v_{t,h}(s) dW_s$ has continuous sample paths. Note that in view of Lemma \ref{lem.empprocbd}, it is sufficient to show that there is an $\alpha$ with $1/2<\alpha<1$ such that
\begin{align*}
  \TV\big(\big(\sqrt{h}V_{t,h}^{-1}v_{t,h}-\sqrt{h'}V_{t',h'}^{-1}v_{t',h'}\big)\langle \cdot \rangle ^\alpha \big)\rightarrow 0,
\end{align*}
whenever $(t',h')\rightarrow (t,h)$ on the space $\mT$. Since Assumption \ref{as.testfcts} (iv) holds, we have 
\begin{align*}
  \big|\sqrt{h}V_{t,h}^{-1}-\sqrt{h'}V_{t',h'}^{-1}\big|
  \leq 
  \frac{\sqrt{|h-h'|}}{V_{t,h}}+V_{t,h}^{-1}\frac{\sqrt{h'}|V_{t',h'}-V_{t,h}|}{V_{t',h'}}\rightarrow 0,
\end{align*}
for $(t',h')\rightarrow (t,h)$. By Lemma \ref{lem.rthlem}, $\TV(v_{t,h}(\cdot)\langle \cdot \rangle^\alpha)<\infty$. Therefore, it is sufficient to show that 
\begin{align}
  \TV\big((v_{t,h}-v_{t',h'})\langle \cdot \rangle ^\alpha \big)\rightarrow 0,
  \quad \text{whenever} \ (t',h')\rightarrow (t,h).
  \label{eq.TVvthdiffcond}
\end{align}
Using \eqref{eq.vthreprs}, we obtain
\begin{align*}
  &(K_{t,h}^{\gamma,\om}a_{t,h}^\star)(u) \\
  &=v_{t,h}-v_{t',h'} \\
  &= h^{-\om}
  \int \lambda_\gamma^\mu\big(\tfrac sh\big)\mathcal{F}\big(\Op(a_{t,h}^\star)(\phi-\phi\circ S_{t',h'}\circ S_{t,h}^{-1})\big)(s)e^{is(u-t)/h} \rqm s
\end{align*}
and by Remark \ref{rem.vth=Kth}, we can apply Lemma \ref{lem.Kthestimates} again (here $\phi$ should be replaced by $\phi-\phi\circ S_{t',h'}\circ S_{t,h}^{-1}$). In order to verify \eqref{eq.TVvthdiffcond}, observe that by Lemma \ref{lem.rthlem} it is enough to show $\|\phi-\phi\circ S_{t',h'}\circ S_{t,h}^{-1}\|_{H_4^{\overline q}}\rightarrow 0$ for some $\overline q>r+m+3/2$ whenever $(t',h')\rightarrow (t,h)$ in $\mathcal{T}$. Note that
\begin{align}
  &\big\|\phi-\phi\circ S_{t',h'}\circ S_{t,h}^{-1}\big\|_{H_4^{\overline q}}^2 \notag \\
  &=\tfrac 1h
  \sum_{j=0}^4 \int \langle s\rangle ^{2\overline q}
  \Big|\mathcal{F}\big((x^j\phi)\circ S_{t,h}\big)(s)
  -\mathcal{F}\big((S_{t,h}(\cdot))^j(\phi\circ S_{t',h'})\big)(s)\Big|^2ds\notag \\
  &\leq
  \frac 2h\sum_{j=0}^4 \big\|(x^j\phi) \circ S_{t,h}-(x^j\phi)\circ S_{t',h'}\big\|_{H^{\overline q}}^2 \notag \\
  &\quad +
  \int \langle s\rangle ^{2\overline q}\Big|
  \mathcal{F}\big(\big[(S_{t',h'}(\cdot))^j-(S_{t,h}(\cdot))^j\big](\phi\circ S_{t',h'})\big)(s)\Big|^2ds
  \label{eq.phiH4bd}
\end{align}
with $(S_{t,h}(\cdot))^j:=\big(\tfrac {\cdot-t}h\big)^j$. For real numbers $a,b$ we have the identity $a^j-b^j= \sum_{\ell=1}^k \binom{k}{\ell} b^{k-\ell}(a-b)^\ell$. Moreover, we can apply Lemma \ref{lem.vthvt'h'} for $\overline q$ with $m+r+3/2<\overline q < \lfloor r+m+5/2\rfloor$ (and such a $\overline q$ clearly exists). Thus, with $a=S_{t,h}(\cdot), \ b=S_{t',h'}(\cdot)$ and $S_{t,h}-S_{t',h'}=(h/h'-1)S_{t',h'}-(t'-t)/h$ the r.h.s. of \eqref{eq.phiH4bd} converges to zero if $(t',h')\rightarrow(t,h)$.
% and this converges to zero if $(t',h')\rightarrow(t,h)$ thanks to Lemma , $\phi \in H_4^{\lfloor r+m+5/2\rfloor}$, $\TV(D^{\lfloor r+m+5/2\rfloor} \phi)<\infty$, and $S_{t,h}-S_{t',h'}=(h/h'-1)S_{t',h'}-(t'-t)/h$ as well as $(S_{t,h}(\cdot))^2-(S_{t',h'}(\cdot))^2=(S_{t,h}(\cdot)-S_{t',h'}(\cdot))^2+2S_{t',h'}(\cdot)(S_{t,h}(\cdot)-S_{t',h'}(\cdot))$.
% 

\end{proof}

\begin{proof}[Proof of Theorem \ref{thm.speciallimits}] By assumption, $p_R(x,\xi)=a_{R}(x,\xi)|\xi|^{\gamma_1}\iota_\xi^{\mu_1}$ with $a_{R}\in S^{\overline{m}_1}$ and $\overline m_1+\gamma_1=m'$. Recall that $p_P(x,\xi)=a_P(x)|\xi|^m\iota_\xi^\mu$. Since $a_P$ is real-valued, $\Op(a_P)$ is self-adjoint.  Taking the adjoint is a linear operator and therefore arguing as in \eqref{eq.comp_adj} yields  
\begin{align*}
  \mathcal{F}\big(\op(p)^\star (\phi\circ S_{t,h})\big)(s)
  =&|s|^m\iota_s^{-\mu} \mathcal{F}\big(a_P(\phi\circ S_{t,h})\big)(s) \\
  &+|s|^{\gamma_1}\iota_s^{-\mu_1}\mathcal{F}\big(\Op(a_R^\star)(\phi\circ S_{t,h})\big)(s).
\end{align*}
Decompose $v_{t,h}=v_{t,h}^{(1)}+v_{t,h}^{(2)}$ with
\begin{align*}
  v_{t,h}^{(1)}(u)
  &:=
  \int \lambda_m^\mu(s) \mathcal{F}\big(a_P(\phi\circ S_{t,h})\big)(s) e^{isu} \rqm s \\
  &=  \int \lambda_m^\mu\big(\tfrac sh\big) 
  \mathcal{F}\big(a_P(\cdot h+t)\phi\big)(s)e^{is(u-t)/h} \rqm s \\
  v_{t,h}^{(2)}(u)
  &:= \int \lambda_{\gamma_1}^{\mu_1}(s)\mathcal{F}\big(\Op(a_R^\star)(\phi\circ S_{t,h})\big)(s) e^{isu} \rqm s \\
  &= h^{-\om_1}\int \lambda_{\gamma_1}^{\mu_1}\big(\tfrac sh\big) 
  \mathcal{F}\big(\Op(a^{(1)}_{t,h})\phi\big)(s)e^{is(u-t)/h} \rqm s
\end{align*}
using similar arguments as in \eqref{eq.vthreprs} and $a_{t,h}^{(1)}(x,\xi):=h^{\om_1} a_R^\star(xh+t,h^{-1}\xi)$. For $j=1,2$ we denote by $T_{t,h}^{(j)}$ and $T_n^{P,(j)}$ the statistics $T_{t,h}$ and $T_n^P$ with $v_{t,h}$ replaced by $v_{t,h}^{(j)}$, $j=1,2$, respectively. Recall the definitions of $\sigma$ and $\tau$ and set
\begin{align}
  v_{t,h}^P(u)
  &:= Aa_P(t)\int |s|^{r+m}\iota_s^{-\rho-\mu} \mathcal{F}(\phi\circ S_{t,h})(s)e^{isu} \rqm s \notag \\
  &= Ah^{-r-m}a_P(t)\int |s|^{r+m}\iota_s^{-\rho-\mu} \mathcal{F}(\phi)(s)e^{is(u-t)/h} \rqm s \notag \\
  &= Aa_P(t) D_+^{\sigma}D_{-}^{\tau} \phi\big(\tfrac{u-t}h\big).
  \label{eq.vthPdef}
\end{align}
Further let
$$
	V_{t,h}^P:=\|v_{t,h}^P\|_2=|A a_P(t)|\big\|D_+^{r+m}\phi((\cdot-t)/h)\big\|_2=h^{1/2-r-m}|Aa_P(t)| \big\|D_+^{r+m}\phi\big\|_2,
$$
and
\begin{align*}
  T_n^{P,(1),\infty}(W)
  :=
  \sup_{(t,h)\in B_n} w_h \left(\frac{\big |\int \Re v_{t,h}^{(1)}(s) dW_s \big|}{V_{t,h}^P}-\sqrt{2\log \tfrac{\nu}h}\right).
\end{align*}
Note that for the approximation of $T_n^P$ we can write
\begin{align*}
  T_n^{P,\infty}(W)
  =
  \sup_{(t,h)\in B_n} w_h \left(\frac{\big |\int \Re v_{t,h}^P(s) dW_s \big|}{V_{t,h}^P}-\sqrt{2\log \tfrac{\nu}h}\right).
\end{align*}
Since $|T_n^P-T_n^{P,\infty}(W)|\leq |T_n^P-T_n^{P,(1)}|+|T_n^{P,(1)}-T_n^{P,(1),\infty}(W)|+|T_n^{P,(1),\infty}(W)-T_n^{P,\infty}(W)|$ it is sufficient to show that there exists a Brownian motion $W$ such that the terms on the right hand side converge to zero in probability. This will be done separately, and proofs for the single terms are denoted by $(I), (II)$ and $(III)$. From $(II)$ and $(III)$ we will be able to conclude the boundedness of the approximating statistic.

{\it (I):} It is easy to see that for a constant $K$, $\|v_{t,h}^{(2)}\|_2\leq Kh^{1/2-m'-r}=:V_{t,h}^R$. By Remark \ref{rem.thmmsc} and
\begin{align*}
  \big|T_n^P-T_n^{P,(1)}\big| 
  \leq 
  \sup_{h\in [l_n,u_n]} \frac{V_{t,h}^R}{V_{t,h}^P} 
  \Big(\sup_{(t,h)\in B_n} w_h\Big(\frac{|T_{t,h}^{(2)}-\E T_{t,h}^{(2)}|}{\sqrt{\widehat g_n(t)} \ V_{t,h}^R}-\sqrt{2\log\big(\tfrac {\nu}h\big)}\Big) & \notag \\
  +\sup_{h\in [l_n,u_n]} w_h \sqrt{2\log\big(\tfrac {\nu}h\big)}&\Big),
\end{align*}
we can apply Theorem \ref{thm.msc} where $m$ should be replaced by $m'$, of course. Because of $u_n^{m-m'}\log n \rightarrow 0$, $(I)$ is proved.

{\it (II):} We show that there is a Brownian motion $W$ such that $|T_n^{P,(1)}-T_n^{P,(1),\infty}(W)|\leq |T_n^{P,(1)}-\widetilde T_n^{(1)}|+|\widetilde T_n^{(1)} - \widetilde T_n^{(1),\infty}(W)|+|\widetilde T_n^{(1),\infty}(W)-T_n^{P,(1),\infty}(W)|$ $=o_P(1)$ with
\begin{align*}
  \widetilde T_n^{(1)} := \sup_{(t,h)\in B_n} w_h \Big(\frac{\big|T_{t,h}^{(1)}-\E T_{t,h}^{(1)} \big|}{\sqrt{\widehat g_n(t)} \ \|v_{t,h}^{(1)}\|_2} -\sqrt{2\log\big(\tfrac {\nu}h\big)}\Big)
\end{align*}
and
\begin{align*}
  \widetilde T_n^{(1),\infty}(W) := \sup_{(t,h)\in B_n} w_h \Big(\frac{\big|\int \Re v_{t,h}^{(1)}(s) dW_s\big |}{\sqrt{\widehat g_n(t)} \ \|v_{t,h}^{(1)}\|_2} -\sqrt{2\log\big(\tfrac {\nu}h\big)}\Big).
\end{align*}
Since by Assumption \ref{as.prinsymb}, $a_p\in S^0$ is elliptic and $p_P\in \underline S^m$,  we find that $|\widetilde T_n^{(1)}-\widetilde T_n^{(1),\infty}(W)|=o_P(1)$ and
\begin{align}
  \widetilde T_n^{(1),\infty}(W) \leq  \sup_{(t,h)\in \mathcal{T}} w_h \Big(\frac{\big|\int \Re v_{t,h}^{(1)}(s) dW_s\big |}{\sqrt{\widehat g_n(t)} \ \|v_{t,h}^{(1)}\|_2} -\sqrt{2\log\big(\tfrac {\nu}h\big)}\Big)<\infty \quad \text{a.s.}
  \label{eq.Tn1inftyasbd}
\end{align}
by applying Theorem \ref{thm.msc}. Moreover, similar as in \eqref{eq.TnTn1diff} and using $w_h\sqrt{2\log\big(\tfrac \nu h\big)}\geq 1$, 
\begin{align*}
  \sup_{G\in \mathcal{G}} \big| T_n^{P,(1)}-\widetilde T_n^{(1)} \big|
  \leq \sup_{(t,h)\in B_n}  w_h\sqrt{2\log\big(\tfrac \nu h\big)}\frac{\big|V_{t,h}^P-\|v_{t,h}^{(1)}\|_2\big|}{V_{t,h}^P}
  \Big(1+\sup_{G\in \mathcal{G}} \widetilde T_n^{(1)}\Big)
\end{align*}
and
\begin{align*}
  &\big|\widetilde T_n^{(1),\infty}(W)-T_n^{P,(1),\infty}(W) \big| \\
  &\leq \sup_{(t,h)\in B_n}  w_h\sqrt{2\log\big(\tfrac \nu h\big)}\frac{\big|V_{t,h}^P-\|v_{t,h}^{(1)}\|_2\big|}{V_{t,h}^P}
  \Big(1+\widetilde T_n^{(1),\infty}(W)\Big).
\end{align*}
To finish the proof for $(II)$ it remains to verify
\begin{align}
 \sup_{(t,h)\in B_n}  w_h\sqrt{2\log\big(\tfrac \nu h\big)} \frac{\|v_{t,h}^P-v_{t,h}^{(1)}\|_2}{V_{t,h}^P}=o(1),
  \label{eq.1condproofSL}
\end{align}
which will be done below.

{\it (III):} By Lemma \ref{lem.empprocbd}, we obtain $|T_n^{P,(1),\infty}-T_n^{P,\infty}|=o_P(1)$ if for some $\alpha>1/2$,
\begin{align}
  \sup_{(t,h)\in B_n} w_h \frac{
  \TV\big((v_{t,h}^P-v_{t,h}^{(1)})\langle \cdot \rangle^\alpha \big)}{V_{t,h}^P}
  =o(1).
  \label{eq.cond_nom_conv}
\end{align}

Let $\chi$ be a cut function, i.e. $\chi\in \mathcal{S}$ (the Schwartz space), $\chi(x)=1$ for $x\in[-1,1]$ and $\chi(x)=0$ for $x\in (-\infty,-2]\cup[2,\infty)$ and define $p_{t,h}^{(1)}(x,\xi)=h^{-1}\chi(x)(a_P(xh+t)-a_P(t))$ and $p_{t,h}^{(2)}(x,\xi)=(xh)^{-1}(1-\chi(x))(a_P(xh+t)-a_P(t))$. Then, $p_{t,h}^{(1)}, p_{t,h}^{(2)}\in S^0$ and $\big(a_P(\cdot h+t)-a_P(t)\big)\phi = h\Op(p_{t,h}^{(1)})\phi+h\Op(p_{t,h}^{(2)})(x\phi)$. Define the function
\begin{align}
  d_{t,h}:= \int e^{is(\cdot-t)/h} \Big(\frac 1{\mathcal{F}(f_\epsilon)(-\tfrac sh)}-A\iota_s^{-\rho}\big|\tfrac{s}h\big|^r\Big)\iota_s^{-\mu}|s|^m \mathcal{F}(\phi)(s) \rqm s
  \label{eq.dthdef}
\end{align}
and note that
\begin{align*}
  \|d_{t,h}\|_2^2 \lesssim h^{1+2m}\int \big\langle \tfrac sh \big \rangle^{2r+2m-2\beta_0} \big|\mathcal{F}(\phi)(s)\big|^2 ds
  \lesssim h^{1+2\beta_0^\star-2r} \|\phi\|_{H^{r+m}}^2
\end{align*}
with $\beta_0^\star:=\beta_0 \wedge (m+r)$. Using \eqref{eq.Kthdef}, we have now the decomposition
\begin{align}
  v_{t,h}^{(1)}-v_{t,h}^P
  &=hK_{t,h}^{m,0} p_{t,h}^{(1)}+hK_{t,h}^{m,0} p_{t,h}^{(2)}
  +
  a_P(t)h^{-m} d_{t,h},
  \label{eq.vthvthPdiff}
\end{align}
where $\phi$ needs to be replaced by $x\phi$ in the second term of the right hand side. By assumption there exists $q>m+r+3/2$ such that $\phi \in H_5^q$. Since the assumptions on $p_{t,h}^{(1)}$ and $p_{t,h}^{(2)}$ of Lemma \ref{lem.Kthestimates} can be easily verified, we may apply Lemma \ref{lem.Kthestimates} to the first two terms on the right hand side of \eqref{eq.vthvthPdiff}. This yields together with Lemmas \ref{lem.rthlem}, \ref{lem.dthapprox}, and \ref{lem.TVvthPlem}, uniformly over $(t,h)\in \mathcal{T}$, 
\begin{align*}
  &\TV\big((v_{t,h}^P-v_{t,h}^{(1)})\langle \cdot \rangle^\alpha \big) \\
  &\leq \TV\big(\big(hK_{t,h}^{m,0} p_{t,h}^{(1)}+hK_{t,h}^{m,0} p_{t,h}^{(2)}
  +
  a_P(t)h^{-m}d_{t,h}\big)\langle \cdot \rangle^\alpha \mathbb{I}_{[t-1,t+1]}\big) \\
  &\quad +\TV\big(v_{t,h}^P\langle \cdot \rangle^\alpha \mathbb{I}_{\mathbb{R}\setminus [t-1,t+1]}\big)
  +\TV\big(v_{t,h}^{(1)}\langle \cdot \rangle^\alpha \mathbb{I}_{\mathbb{R}\setminus [t-1,t+1]}\big) \\
  &\lesssim h^{1-m-r}+h^{\beta_0^\star-m-r}+h^{1-r-m}.
\end{align*}
Since $m+r>1/2$ this implies \eqref{eq.cond_nom_conv}. From the decomposition \eqref{eq.vthvthPdiff} we obtain further $\|v_{t,h}^P-v_{t,h}^{(1)}\|_2 \lesssim h^{3/2-m-r}+h^{1/2+\beta_0^\star-m-r}$ and this shows \eqref{eq.1condproofSL}. Thus, the first part of the theorem is proved. 

Finally with Lemma \ref{lem.empprocbd} it is easy to check that \eqref{eq.Tn1inftyasbd} implies that \eqref{eq.limitbd3} is bounded since \eqref{eq.1condproofSL} and \eqref{eq.cond_nom_conv} also hold with $B_n$ and $o(1)$ replaced by $\mathcal{T}$ and $O(1)$, respectively.
\end{proof}

% AOS,AOAS: If there are supplements please fill:

\section{Technical results for the proofs of the main theorems}

We have the following uniform and continuous embedding of Sobolev spaces.

\begin{lem2}
\label{lem.unifemb}
Let $\mathcal{P} \subset S^m$ be a symbol class of pseudo-differential operators. Suppose further that for $\alpha \in \{0,1\}$, $k\in \mathbb{N}$ and finite constants $C_k$, depending on $k$ only,
\begin{align*}
  \sup_{p\in \mathcal{P}} |\partial_x^k \partial_\xi^\alpha p(x,\xi)|\leq C_k (1+|\xi|)^m,
  \quad \forall x,\xi\in \mathbb{R}.
\end{align*}
Then, for any $s\in \mathbb{R},$ there exists a finite constant $C$, depending only on $s,m$ and $\max_{k\leq 1+2|s|+2|m|} C_k$, such that
\begin{align*}
  \sup_{p\in \mathcal{P}} \|\Op(p)\phi\|_{H^{s-m}}\leq C \|\phi\|_{H^s}, \quad \text{for all} \ \ \phi \in H^s.
\end{align*}
\end{lem2}

\begin{proof}
This proof requires some subtle technicalities, appearing in the theory of pseudo-differential operators. By Theorem 2 in Hwang \cite{hwa}, there exists a universal constant $C_1,$ such that for any symbol $a\in S^0,$
\begin{align}
  \|\Op(a)u\|_2\leq C_1 \max_{\alpha,\beta\in\{0,1\}} \big\|\partial_x^\beta \partial_\xi^\alpha a(x,\xi)\big\|_{L^\infty(\mathbb{R}^2)} \|u\|_2, \quad \text{for all} \ u\in L^2.
  \label{eq.uniformL2bound}
\end{align}
For $r\in \mathbb{R}$ denote by $\Op(\langle \xi\rangle^r)$ the pseudo-differential operator with symbol $(x,\xi)\mapsto \langle \xi\rangle^r$. It is well-known that this symbol is in $S^r$. Throughout the remaining proof let
$$
	C=C(s,m, \max_{k\leq 1+2|s|+2|m|} C_k)
$$
denote a finite but unspecified constant which may even change from line to line. In order to prove the result it is sufficient to show that
\begin{align*}
  \sup_{p\in \mathcal{P}}\big\|\Op(\langle \xi\rangle^{s-m})\circ \Op(p)\circ \Op(\langle \xi\rangle^{-s})\psi \big\|_2\leq C\|\psi\|_2, \quad \text{for all} \ \psi\in L^2
\end{align*}
(set $\phi=\langle D\rangle^{-s} \psi$). The composition of two operators with symbols in $S^{m_1}$ and $S^{m_2}$, respectively, is again a pseudo-differential operator and its symbol is in $S^{m_1+m_2}$. Therefore, the operator $A:\mathcal{P}\rightarrow S^0,$ mapping $p\in \mathcal{P}$ to the symbol of $\Op(\langle \xi\rangle^{s-m})\circ \Op(p)\circ \Op(\langle \xi\rangle^{-s})$ (which is in $S^0$), is well-defined. With \eqref{eq.uniformL2bound} the lemma is proved, once we have established that
\begin{align*}
  \sup_{p \in \mathcal{P}} \max_{\alpha,\beta\in\{0,1\}} \big\|\partial_x^\beta \partial_\xi^\alpha Ap(x,\xi)\big\|_{L^\infty(\mathbb{R}^2)}\leq C<\infty.
\end{align*}
It is not difficult to see that $\Op(p)\circ \Op(\langle \xi\rangle^{-s})= \Op(p\langle \xi\rangle^{-s})$. By Theorem 4.1 in \cite{ali}, $Ap = \langle \xi\rangle^{s-m} \# (p\langle \xi\rangle^{-s})$, where $\#$ denotes the Leibniz product, i.e. for $p^{(1)}\in S^{m_1}$ and $p^{(2)}\in S^{m_2}$, $p^{(1)}\#p^{(2)}$ can be written as an oscillatory integral (cf. \cite{ali, wlo}), that is
\begin{align*}
  \big(p^{(1)}\#p^{(2)}\big)(x,\xi)
  &:=
  \Os-\int\int e^{-iy\eta}p^{(1)}(x,\xi+\eta)p^{(2)}(x+y,\xi) dy \rqm \eta \\
  &:= \lim_{\epsilon \rightarrow 0}
  \int\int \chi(\epsilon y,\epsilon \eta)e^{-iy\eta}p^{(1)}(x,\xi+\eta)p^{(2)}(x+y,\xi)dy \rqm \eta,
\end{align*}
for any $\chi$ in the Schwartz space of rapidly decreasing functions on $\mathbb{R}^2$ with $\chi(0,0)=1$. Further for $a\in S^m$ and arbitrary $l\in \mathbb{N}$, $2l>1+m$,
\begin{align*}
  &\Os-\int\int e^{-iy\eta}a(y,\eta)dy \rqm \eta \\
  &= \int\int e^{-iy\eta}\langle y\rangle^{-2}(1-\partial_\eta^2)\big[\langle \eta\rangle^{-2l}(1-\partial_y^2)^l a(y,\eta)\big]dy \rqm \eta
\end{align*}
and the integrand on the r.h.s. is in $L^1$ (cf. \cite{wlo}, p.235). This can be also used to show that differentiation and integration commute for oscillatory integrals,
%(cf. Abel Skript Satz 4.13, p. 31 ???)
\begin{align*}
  \partial_x^\alpha\partial_\xi^\beta 
  \Os-\int\int e^{-iy\eta}a(x,y,\xi,\eta)dy \rqm \eta
  =\Os-\int\int e^{-iy\eta} \partial_x^\alpha\partial_\xi^\beta a(x,y,\xi,\eta)dy \rqm \eta.
\end{align*}
Using Peetre's inequality, i.e. $\langle \xi+\eta\rangle^s\leq 2^{|s|}\langle \xi\rangle^{|s|}\langle \eta\rangle^s$, we see that for $\alpha,\beta\in \{0,1\},$ $p\in \mathcal{P},$ and $(x,\xi)$ fixed, the function $(y,\eta)\mapsto \partial_x^\beta\partial_\xi^\alpha \langle \xi+\eta\rangle^{s-m}p(x+y,\xi)\langle \xi\rangle^{-s}$ defines a symbol in $S^{s-m}$. Hence, for $\ell\in \mathbb{N}$, $1<2\ell-|s-m|\leq 2$, $\alpha,\beta\in \{0,1\}$, $p\in \mathcal{P},$ we can rewrite $\partial_x^\beta\partial_\xi^\alpha Ap(x,\xi)$ as
\begin{align*}
  \int\int e^{-iy\eta} \langle y\rangle^{-2}(1-\partial_\eta^2)\big[\langle \eta\rangle^{-2\ell}(1-\partial_y^2)^\ell \partial_x^\beta\partial_\xi^\alpha \langle \xi+\eta\rangle^{s-m}p(x+y,\xi)\langle \xi\rangle^{-s}\big] dy \rqm \eta.
\end{align*}
With the imposed uniform bound on $\partial_x^k\partial_\xi^\alpha p(x,\xi)$ we obtain, treating the cases $\alpha=0$ and $\alpha=1$ separately,
\begin{align*}
  &\sup_{p\in \mathcal{P}} \big|\partial_x^\beta\partial_\xi^\alpha Ap(x,\xi)\big| \\
  &\leq
  C\langle \xi\rangle^{m-s} \Big[ \int \big|(1-\partial_\eta^2) \langle \eta\rangle^{-2\ell}\langle \xi+\eta\rangle^{s-m}\big|d\eta  \\
  &\quad\quad\quad\quad\quad +\int \big|(1-\partial_\eta^2) \langle \eta\rangle^{-2\ell}\partial_\xi \langle \xi+\eta\rangle^{s-m}\big|d\eta\Big] \\
  &\leq  C+C\langle \xi\rangle^{m-s} \Big[ \int \big|\partial_\eta^2 \langle \eta\rangle^{-2\ell}\langle \xi+\eta\rangle^{s-m}\big|d\eta 
  +\int \big|\partial_\eta^2 \langle \eta\rangle^{-2\ell}\partial_\xi \langle \xi+\eta\rangle^{s-m}\big| d\eta\Big]
\end{align*}
using Peetre's inequality again and $2\ell>1+|s-m|$ for the second estimate. Since $\langle \xi\rangle^q \in S^q$ for $q\in \mathbb{R}$, it follows that $|\partial_\xi^\alpha \langle \xi\rangle^q|\lesssim \langle \xi\rangle^{q-\alpha}$, and since $\langle .\rangle\geq 1$,
\begin{align*}
  \partial_\eta^2 \langle \eta\rangle^{-2\ell}\langle \xi+\eta\rangle^{s-m}
  \lesssim 
  \sum_{k=0}^2 \langle \eta\rangle^{-2\ell-k} \langle \xi+\eta\rangle^{s-m-2+k}
  \lesssim \langle \eta\rangle^{-2\ell}\langle \xi+\eta\rangle^{s-m}.
\end{align*}
Similar for the second term. Application of Peetre's inequality as above completes the proof.
\end{proof}

Note that for bounded intervals $[a,b]$, partial integration holds $\int_a^b f'g= fg |_a^b - \int_a^b fg'$ whenever $f$ and $g$ are absolute continuous on $[a,b]$. As a direct consequence, we have $\int_{\mathbb{R}} f'g = -\int_{\mathbb{R}} fg'$ if $f'$ and $g'$ exist and $fg, f'g, fg' \in L^1$.

% Define the truncated and normalized functions $v_{t,h}^{(1)}=v_{t,h}\mathbb{I}_{[t-\sqrt{h}\log n,t+\sqrt h \log n]}$ and $v_{t,h}^{(2)}=( v_{t,h}-v_{t,h}^{(1)})/\|v_{t,h}\|_{L^2}$.

In order to formulate the key estimate for proving Theorems \ref{thm.msc} and \ref{thm.speciallimits}, let us introduce for fixed $\phi$ a generic symbol $\mfp \in S^{\om}$ and $\lambda=\lambda_\gamma^\mu$ as in \eqref{eq.lambdadef}:
\begin{align}
  (K_{t,h}^{\gamma,\overline m}\mfp)(u)=h^{-\om}\int \lambda\big(\tfrac sh\big)\mathcal{F}\big(\Op(\mfp)\phi\big)(s)
  e^{is(u-t)/h} \rqm s.
  \label{eq.Kthdef}
\end{align}
From the context it will be always clear which $\phi$ the operator $K_{t,h}^{\gamma,\overline m}\mfp$ refers to. To simplify the expressions we do not indicate the dependence on $\phi$ and $f_\epsilon$ explicitly. 

\begin{rem2}
\label{rem.vth=Kth}
Recall \eqref{eq.defathstar} and note that if $a\in S^{\om}$ then also $a_{t,h}^\star\in S^{\om}$. Due to
\begin{align*}
 \big(\Op(a_{t,h}^\star)\phi\big)\circ S_{t,h}=h^{-\om}\Op(a^\star)(\phi\circ S_{t,h})
\end{align*}
we obtain for $v_{t,h}$ in \eqref{eq.vthdef} the representation,
\begin{align}
  v_{t,h}(u)
  &= 
  h^{-{\om}} 
  \int \lambda_\gamma^\mu\big(\tfrac sh\big)\mathcal{F}\big(\Op(a_{t,h}^\star)\phi\big)(s)
  e^{is(u-t)/h} \rqm s
  =\big(K_{t,h}^{\gamma,\overline m} a_{t,h}^\star\big)(u).
  \label{eq.vthreprs}
\end{align} 
\end{rem2}

\begin{lem2}
\label{lem.Kthestimates}
For $\mfp\in S^{\overline m}$ and $\gamma+\overline m= m$ let $K_{t,h}^{\gamma,\overline m}\mfp$ be as defined in \eqref{eq.Kthdef}. Work under Assumption \ref{ass.noise} and suppose that 
\begin{itemize}
 \item[(i)] $\phi \in H_4^q$ with $q>m+r+3/2$,
 \item[(ii)] $\gamma \in \{0\} \cup [1,\infty)$, and
 \item[(iii)] for $k\in \mathbb{N}$, $\alpha\in \{0,1,\ldots,5\}$, there exist finite constants $C_k$ such that
\begin{align*}
  \sup_{(t,h)\in \mT} \big|\partial_x^k\partial_\xi^\alpha \mfp(x,\xi)\big|
  \leq C_k(1+|\xi|)^{\om}, \quad \text{for all} \ x,\xi\in \mathbb{R}.
\end{align*} 
\end{itemize}

Then, there exists a constant $C=C(q, r,\gamma, \om, C_l, C_u, \max_{k\leq 4q} C_k)$ ($C_l$ and $C_u$ as in Assumption \ref{ass.noise}) such that for $(t,h)\in \mathcal{T}$,
\begin{itemize}
 \item[(i)] $|(K_{t,h}^{\gamma,\overline m}\mfp)(u)|\leq C\|\phi\|_{H_4^q} h^{-m-r}\min\big(1,\tfrac {h^2}{(u-t)^2}\big)$,
 \item[(ii)] $|(K_{t,h}^{\gamma,\overline m}\mfp)(u)-(K_{t,h}^{\gamma,\overline m}\mfp)(u')|\leq C\|\phi\|_{H_4^q}h^{-m-r-1}|u-u'|$ and for $u,u'\neq t$, 
\begin{align*}  
  |(K_{t,h}^{\gamma,\overline m}\mfp)(u)-(K_{t,h}^{\gamma,\overline m}\mfp)(u')|
  &\leq C\|\phi\|_{H_4^q}\frac{h^{1-m-r} |u-u'|}{|u'-t| \ |u-t|} \\
  &=C\|\phi\|_{H_4^q} h^{1-m-r}\big|\int_{u'}^u \frac 1{(x-t)^2} dx\big|.
\end{align*}
\end{itemize}
\end{lem2}

\begin{proof}
During this proof, $C=C(q, r,\gamma, \om, C_l, C_u, \max_{k\leq 4q} C_k)$ denotes an unspecified constant which may change in every line. The proof relies essentially on the well-known commutator relation for pseudo-differential operators, $[x,\Op(p)]=i\Op(\partial_\xi p)$, with $\partial_\xi p: (x,\xi)\mapsto \partial_\xi p(x,\xi)$ (cf. Theorem 18.1.6 in \cite{hoer}). By induction for $k\in \mathbb{N}$,
\begin{align}
  x^k \Op(\mfp) = \sum_{r=0}^k \binom{k}{r} i^r \Op\big(\partial_\xi^r \mfp\big) x^{k-r}.
  \label{eq.binomcomrel}
\end{align}

As a preliminary result, let us show that for $k=0,1,2$ the $L^1$-norms of 
\begin{align}
  \langle s\rangle \ D_s^k \ \lambda\big(\tfrac sh\big) \mathcal{F}(\Op(\mfp)\phi)(s),
  \label{eq.sderivFquotL1bd}
\end{align}
are bounded by $C \|\phi\|_{H^q_2}h^{-r-\gamma}$. Using Assumption \ref{ass.noise} and Lemma \ref{lem.unifemb} this follows immediately for $k=0$ and $q>r+m+3/2$ by 
\begin{align}
  \int \Big|\langle s\rangle \ \lambda\big(\tfrac sh\big)\mathcal{F}(\Op(\mfp)\phi)(s)\Big| ds
  &\leq C_l^{-1}h^{-r-\gamma} \big\| \langle \cdot \rangle^{1+r+\gamma} \ \mathcal{F}(\Op(\mfp)\phi)\big\|_1
  \notag \\
  &\leq Ch^{-r-\gamma} \big\|\Op(\mfp)\phi\big\|_{H^{q-\om}} \notag \\  
  &\leq Ch^{-r-\gamma} \|\phi\|_{H^q}.
  \label{eq.L1boundk=0}
\end{align}
Now, $\mfp \in S^{\om}$ implies that for $k\in \mathbb{N}$, $\partial_\xi^k \mfp \in S^{\om-k}\subset S^{\om}$. Since by \eqref{eq.binomcomrel}, Assumptions (i) and (iii), and Lemma \ref{lem.unifemb},
\begin{align}
  \|\langle x\rangle^2\Op(\mfp)\phi\|_1
  \lesssim \|(1+x^4)\Op(\mfp)\phi\|_2\leq C \|\phi\|_{H_4^{\om}} <\infty,
  \label{eq.x2Opmfpbd}
\end{align}
we obtain for $j\in\{1,2\}$,
\begin{align*}
 D_s^j\mathcal{F}(\Op(\mfp)\phi)=(-i)^j \mathcal{F}(x^j\Op(\mfp)\phi)(s) 
\end{align*}
by interchanging differentiation and integration. Explicit calculations thus show 
\begin{align*}
  D_s \lambda\big(\tfrac sh\big)\mathcal{F}\big(\Op(\mfp)\phi\big)(s)
  &=
  \big(D_s\lambda\big(\tfrac sh\big)\big)\mathcal{F}\big(\Op(\mfp)\phi\big)(s) \\
  &\quad -i\lambda\big(\tfrac sh\big)
  \mathcal{F}\big(x\Op(\mfp)\phi\big)(s)
\end{align*}
and
\begin{align}
  D_s^2 \lambda\big(\tfrac sh\big)\mathcal{F}\big(\Op(\mfp)\phi\big)(s)
  &=
  \big(D_s^2\lambda\big(\tfrac sh\big)\big)\mathcal{F}\big(\Op(\mfp)\phi\big)(s) \notag \\
  &\quad -2i
  \big(D_s\lambda\big(\tfrac sh\big)\big)\mathcal{F}\big(x\Op(\mfp)\phi\big)(s) \notag \\
  &\quad -\lambda\big(\tfrac sh\big)
  \mathcal{F}\big(x^2\Op(\mfp)\phi\big)(s).
  \label{eq.Ds2expansion}
\end{align}
To finish the proof of \eqref{eq.sderivFquotL1bd} let us distinguish two cases, namely $(I)$ $\gamma \in \{0\} \cup [2,\infty)$ and $(II)$ $\gamma \in (1,2)$.

\medskip

{\it (I):} For $k=0,1,2$, $s\neq 0$, we see by elementary calculations, $\big|\langle s\rangle D_s^k\lambda\big(\tfrac sh\big)\big| \linebreak\leq Ch^{-r-\gamma}\langle s\rangle^{r+\gamma+1}$.  Using \eqref{eq.binomcomrel} and arguing similar as for \eqref{eq.L1boundk=0} we obtain (replacing $\phi$ by $x\phi$ or $x^2\phi$ if necessary) bounds of the $L^1$-norms which are of the correct order $\|\phi\|_{H_4^q}h^{-r-\gamma}$. 

\medskip

{\it (II):} In principal we use the same arguments as in $(I)$ but a singularity appears by expanding the first term on the r.h.s. of \eqref{eq.Ds2expansion}. In fact, it is sufficient to show that
\begin{align*}
  &\int_{-1}^1 \Big| \frac{D_s^2 \big|\tfrac sh\big|^\gamma \iota_s^{-\mu}}{\mathcal{F}(f_\epsilon)\big(-\tfrac sh\big)}
  \mathcal{F}\big(\Op(\mfp)\phi\big) (s) \Big| ds \\
  &\quad \leq C_l h^{-r-\gamma} \big\|\mathcal{F}\big(\Op(\mfp)\phi\big) \big\|_\infty
  \int_1^1 |s|^{\gamma-2} ds\\
  &\quad \lesssim C_l h^{-r-\gamma} \big\|\Op(\mfp)\phi \big\|_1\leq Ch^{-r-\gamma}\|\phi\|_{H_4^{\om}},
\end{align*}
where the last inequality follows from \eqref{eq.x2Opmfpbd}. Since this has the right order $h^{-r-\gamma}\|\phi\|_{H_4^q}$, \eqref{eq.sderivFquotL1bd} follows for $\gamma \in (1,2)$.

Together $(I)$ and $(II)$ prove \eqref{eq.sderivFquotL1bd}. Hence, we can apply partial integration twice and obtain for $t\neq u$,
\begin{align}
  (K_{t,h}^{\gamma,\overline m}\mfp)(u)=
  -\frac{h^{2-\om}}{(u-t)^2} \int e^{is(u-t)/h} \ D_s^2 \lambda\big(\tfrac sh\big) \mathcal{F}\big(\Op(\mfp)\phi\big)(s) \rqm s
  \label{eq.vthpartinttwice}
\end{align}
and similarly, first interchanging integration and differentiation,
\begin{align}
  D_u(K_{t,h}^{\gamma,\overline m}\mfp)(u)
  &=
  ih^{-\om-1}\int e^{is(u-t)/h} s\lambda\big(\tfrac sh\big) \mathcal{F}\big(\Op(\mfp)\phi\big)(s) \rqm s \notag \\
  &=
  -\frac{ih^{1-\om}}{(u-t)^2}
  \int e^{is(u-t)/h} D_s^2 s\lambda\big(\tfrac sh\big) \mathcal{F}\big(\Op(\mfp)\phi\big)(s) \rqm s
  \label{eq.vth_diff_partinttwice}
\end{align}

\medskip

{\it (i):} The estimates $|(K_{t,h}^{\gamma,\overline m}\mfp)(u)|\leq C\|\phi\|_{H_4^q}h^{-m-r}$ and $|(K_{t,h}^{\gamma,\overline m}\mfp)(u)| \linebreak\leq C \|\phi\|_{H_4^q} h^{2-m-r}/(u-t)^2$ follow directly from \eqref{eq.L1boundk=0} as well as \eqref{eq.vthpartinttwice} together with the $L^1$ bound of \eqref{eq.sderivFquotL1bd} for $k=2$.

{\it (ii):} To prove $|(K_{t,h}^{\gamma,\overline m}\mfp)(u)-(K_{t,h}^{\gamma,\overline m}\mfp)(u')|\leq C\|\phi\|_{H_4^q}h^{-m-r-1}|u-u'|$ it is enough to note that $|e^{ix}-e^{iy}|\leq |x-y|$. The result then follows from \eqref{eq.L1boundk=0} again. For the second bound, see \eqref{eq.vth_diff_partinttwice}. The estimate for the $L^1$-norm of \eqref{eq.sderivFquotL1bd} with $k=2$ completes the proof.
\end{proof}

\begin{lem2}
\label{lem.vthL2}
Work under the assumptions of Theorem \ref{thm.msc}. If $v_{t,h}$ is given as in \eqref{eq.vthdef}, then,
\begin{align*}
  \|v_{t,h}\|_2\gtrsim h^{1/2-m-r}.
\end{align*}
\end{lem2}

\begin{proof}
We only discuss the case $\gamma>0$. If $\gamma=0$ the proof can be done similarly. It follows from the definition that
\begin{align*}
  \|v_{t,h}\|_2^2
  &=\int \frac{1+|s|^{2\gamma}}{|\mathcal{F}(f_\epsilon)(-s)|^2}
  \big|\mathcal{F}\big(\Op(a^\star)(\phi\circ S_{t,h})\big)(s)\big|^2 ds \\
  &\quad -
  \Big\|\frac{\mathcal{F}\big(\Op(a^\star)(\phi\circ S_{t,h})\big)}{\mathcal{F}(f_\epsilon)(-\cdot)}\Big\|_2^2.
\end{align*}
Since the adjoint is given by $a^\star(x,\xi)=e^{\partial_x\partial_\xi}\overline{a}(x,\xi)$ in the sense of asymptotic summation, it follows immediately that $a^\star(x,\xi)=\overline a(x,\xi)+r(x,\xi)$ with $r\in S^{\om-1}$. From this we conclude that $\Op(a^\star)$ is an elliptic pseudo-differential operator. Because of $a^\star\in S^{\overline m}$ and ellipticity there exists a so called left parametrix $(a^\star)^{-1}\in S^{-\overline m}$ such that $\Op((a^\star)^{-1})\Op(a^\star)=1+\Op(a')$ and $a'\in S^{-\infty}$, where $S^{-\infty}=\bigcap_m S^m$ (cf. Theorem 18.1.9 in H\"ormander \cite{hoer}). In particular, $a'\in S^{-1}$. Moreover, $\Op((a^\star)^{-1}):H^{r+\gamma}\rightarrow H^{r+m}$ is a continuous and linear and therefore bounded operator (cf. Lemma \ref{lem.unifemb}). Introduce the function $Q=(\cdot \vee 0)^2.$ Furthermore, by convexity, $1+|s|^{2\gamma}\geq 2^{-\gamma}\langle s\rangle^{2\gamma}$ and there exists a finite constant $c>0$ such that
\begin{align*}
  \int &\frac{1+|s|^{2\gamma}}{|\mathcal{F}(f_\epsilon)(-s)|^2}
  \big|\mathcal{F}\big(\Op(a^\star)(\phi\circ S_{t,h})\big)(s)\big|^2 ds \\
  &\geq 2^{-\gamma} C_l^2
  \big \|\Op(a^\star)(\phi\circ S_{t,h})\|_{H^{r+\gamma}}^2 \\
  &\gtrsim
  \|\Op((a^\star)^{-1})\Op(a^\star)(\phi\circ S_{t,h})\|_{H^{r+m}}^2 \\
  &=\|(1+\Op(a'))(\phi\circ S_{t,h})\|_{H^{r+m}}^2\\
  &\geq Q\big(\|\phi \circ S_{t,h}\|_{H^{r+m}}
  - \|\Op(a')(\phi \circ S_{t,h})\|_{H^{r+m}}\big) \\
  &\geq
  Q\big(\|\phi \circ S_{t,h}\|_{H^{r+m}}
  -c \|\phi \circ S_{t,h}\|_{H^{r+m-1}}\big)
  \\
  &\geq h\int \big(1+\big| \tfrac sh \big |^2\big)^{m+r} \big|\mathcal{F}(\phi)(s)\big|^2 ds+O(h^{2(1-r-m)}) \\
  &\geq h^{1-2(r+m)} \int |s|^{2m+2r}\big|\mathcal{F}(\phi)(s)\big|^2ds
  +O(h^{2(1-r-m)}).
\end{align*}
On the other hand, we see immediately that
\begin{align*}
  \Big\|\frac{\mathcal{F}\big(\Op(a^\star)(\phi\circ S_{t,h})\big)}{\mathcal{F}(f_\epsilon)(-\cdot)}\Big\|_2^2
  &\lesssim 
  \big\|\Op(a^\star)(\phi\circ S_{t,h})\big\|_{H^r}^2 \\
  &\lesssim \|\phi\circ S_{t,h}\|_{H^{r+\om}}^2
  \lesssim h^{1-2(r+\om)}.
\end{align*}
Since $\phi \in L^2$ and $h$ tends to zero the claim follows. 
\end{proof}

\begin{lem2}[D\"umbgen, Spokoiny \cite{due2}, p.145]
\label{lem.phiL2}
Suppose that $\supp \psi \subset [0,1]$ and $\TV(\psi)<\infty$. If $(t,h), (t',h') \in \mT$, then
\begin{align*}
  \big\|\psi\big(\tfrac{\cdot-t}h\big)
  -\psi\big(\tfrac{\cdot-t'}{h'}\big)\big\|_2^2
  \leq 2\TV(\psi)^2\big(|h-h'|+|t-t'|\big).
\end{align*}
\end{lem2}

Let $\lceil x\rceil $ be the smallest integer which is not smaller than $x$.

\begin{lem2} 
\label{lem.vthvt'h'}
Let $0\leq \ell\leq 1/2$ and $q\geq0$. Assume that $\phi \in H^{\lceil q \rceil}\cap H^{q+\ell}$, $\supp \phi \subset [0,1]$ and $\TV(D^{\lceil q \rceil}\phi)<\infty$. Then, for $h\leq h'$,
\begin{align*}
  \|\phi\circ S_{t,h}-\phi\circ S_{t',h'}\|_{H^q}\lesssim h^{-q} \sqrt{|t-t'|^{2\ell}+|h'-h|}.
\end{align*}
In particular, for $\phi\in H^{\lceil r+m \rceil}\cap H^{r+m+1/2}$, $\supp \phi \subset [0,1]$ and $\TV(D^{\lceil r+m \rceil}\phi) \linebreak<\infty$, $h\leq h'$,
\begin{align*}
  \|v_{t,h}-v_{t',h'}\|_2 \lesssim h^{-r-m} \sqrt{|t-t'|+|h'-h|}.
\end{align*}
\end{lem2}

\begin{proof}
Since 
\begin{align*}
  &\big\|\phi\circ S_{t,h}-\phi\circ S_{t',h'}\big\|_{H^q}^2 \\
  &\quad \lesssim  
  \int \langle s\rangle^{2q} \big|1-e^{is(t-t')}\big|^2 \big|\mathcal{F}\big(\phi\big(\tfrac{\cdot}h\big)\big)(s)\big|^2 ds
  +\big\|\phi\big(\tfrac{\cdot}h\big)-\phi\big(\tfrac{\cdot}{h'}\big)\big\|_{H^{q}}^2
\end{align*}
and $|1-e^{is(t-t')}|\leq 2\min(|s||t-t'|,1)\leq 2\min(|s|^\ell|t-t'|^\ell,1)\leq 2|s|^\ell|t-t'|^\ell$, we obtain
\begin{align*}
  \big\|\phi\circ S_{t,h}-\phi\circ S_{t',h'}\big\|_{H^{q}}^2
  \lesssim |t-t'|^{2\ell} h^{1-2q-2\ell}+\big\|\phi\big(\tfrac{\cdot}h\big)-\phi\big(\tfrac{\cdot}{h'}\big)\big\|_{H^{q}}^2
\end{align*}
(note that $\phi \in H^{q+\ell}$). Set $k=\lceil q\rceil$. Then
\begin{align*}
 \big\|\phi\big(\tfrac{\cdot}h\big)-\phi\big(\tfrac{\cdot}{h'}\big)\big\|_{H^{q}}^2
  &\lesssim
  h^{1-2q} 
  \big\|\phi-\phi\big(\tfrac{h}{h'}\cdot\big)\big\|_{H^{q}}^2 \\
  &\lesssim
  h^{1-2q} 
  \big\|\phi-\phi\big(\tfrac{h}{h'}\cdot\big)\big\|_2^2
  +h^{1-2q} \big\|D^k\big(\phi-\phi\big(\tfrac{h}{h'}\cdot\big)\big)\big\|_2^2.
\end{align*}
For $j\in \{0,k\}$,
\begin{align*}
  \big\|D^j\big(\phi-\phi\big(\tfrac{h}{h'}\cdot\big)\big)\big\|_2^2
  &\leq 2\big\|\phi^{(j)}-\phi^{(j)}\big(\tfrac{h}{h'}\cdot\big)\big\|_2^2
  +2\big(1-\big(\tfrac h{h'}\big)^j\big)^2 \big\|\phi^{(j)}\big(\tfrac h{h'}\cdot\big)\big\|_2^2 \\
  &\lesssim
  h^{-1} \big\|\phi^{(j)}\big(\tfrac{\cdot}{h}\big)-\phi^{(j)}\big(\tfrac{\cdot}{h'}\big)\big\|_2^2
  +|h'-h| \ h^{-1}\|\phi^{(j)}\|_2^2.
\end{align*}
Now, application of Lemma \ref{lem.phiL2} completes the proof for the first part. The second claim follows from
\begin{align*}
  \|v_{t,h}-v_{t',h'}\|_2^2
  &=\int |\lambda(s)|^2 \big|\mathcal{F}\big(\Op(a^\star)(\phi\circ S_{t,h}-\phi\circ S_{t',h'}))\big)(s)\big|^2 ds \\
  &\lesssim 
  \big\|\phi\circ S_{t,h}-\phi\circ S_{t',h'}\big\|_{H^{r+m}}^2.
\end{align*}
\end{proof}
% Since $\phi$ is defined on a bounded domain, $\bigcup_{t,h \in B_n} \supp \phi\circ S_{t,h}$ is bounded and hence we have the compact and therefore bounded embedding of the H\"older continuous functions with index larger $r+m$. More precisely, for any $\delta>0$, $\mathcal{C}^{r+m+\delta}\hookrightarrow H^{r+m}$ (cf. Triebel \cite{triebel}, Theorem 1.97 and p.16). If $q$ is not an integer, the norm on $\mathcal{C}^{q}$ is
% \begin{align*}
%   \|f\|_{\mathcal{C}^{q}}= \sum_{\alpha=0}^{[q]} \|D^\alpha f\|_{\infty}+ \sup_{x\neq y} \frac{|D^{[q]}f(x)-D^{[q]}f(y)|}{|x-y|^{q-[q]}},
% \end{align*}
% where $[q]$ denotes the integer part of $q$ (cf. Triebel \cite{triebel}, p. 15). Thus,

\begin{lem2}
 \label{lem.Atthh}
Let $A_{t,t',h,h'}$ be defined as in \eqref{eq.defAtthh} and work under Assumption \ref{as.testfcts}. Then, for a global constant $K>0$, 
\begin{align*}
  A_{t,t',h,h'}\leq K \sqrt{|t-t'|+|h-h'|}.
\end{align*}
\end{lem2}

\begin{proof}
Without loss of generality, assume that for fixed $(t,h)$, $V_{t,h}\geq V_{t',h'}$. We can write
\begin{align*}
  A_{t,t',h,h'}
  &\leq \frac{\|\psi_{t,h}\sqrt h -\psi_{t',h'}\sqrt{h'}\|_2}
  {V_{t,h}}
  + \sqrt {h'} \|\psi_{t',h'}\|_2
  \Big|\frac 1{V_{t,h}}-\frac 1{V_{t',h'}}\Big| \\
  &\leq \frac{\|\psi_{t,h}\sqrt h -\psi_{t',h'}\sqrt{h'}\|_2}
  {V_{t,h}}
  +
  \sqrt{h'} \frac{|V_{t,h}-V_{t',h'}|}
  {V_{t,h}}.
\end{align*}
By triangle inequality, $\|\psi_{t,h}\sqrt{h}-\psi_{t',h'}\sqrt{h'}\|_2\leq \sqrt{h'}\|\psi_{t,h}-\psi_{t',h'}\|_2+|\sqrt{h}-\sqrt{h'}| \ \|\psi_{t,h}\|_2$. Thus,
\begin{align*}
  A_{t,t',h,h'}
  \leq 
  \frac{\sqrt {h'}}{V_{t,h}}
  \Big(\|\psi_{t,h}-\psi_{t',h'}\|_2+|V_{t,h}-V_{t',h'}|\Big)+\sqrt{|h-h'|}.
\end{align*}
If $h'\leq h$, then the result follows by Assumption \ref{as.testfcts} (iv) and some elementary computations. Otherwise we can estimate $\sqrt {h'} \leq \sqrt{|h-h'|}+\sqrt{h}$ and so
\begin{align*}
  A_{t,t',h,h'}
  \leq 
  \frac{\sqrt {h}}{V_{t,h}}
  \Big(\|\psi_{t,h}-\psi_{t',h'}\|_2+|V_{t,h}-V_{t',h'}|\Big)+5\sqrt{|h-h'|}.
\end{align*}
\end{proof}

\begin{rem}
\label{rem.ualpha_ids}
For the proofs of the subsequent lemmas, we make often use of elementary facts related to the function $\langle \cdot \rangle^\alpha\in S^\alpha$ with $0<\alpha <1$. Note that for $t\in[0,1]$, $D_u\langle u\rangle^\alpha \leq \alpha \langle u\rangle^{\alpha-1} \in S^{\alpha-1}$, $D_u\langle u\rangle^\alpha\leq \alpha$,
\begin{align}
  \langle u\rangle^\alpha \leq \frac 12 (1+|u|^\alpha)\leq 1+|u-t|^\alpha, 
  \quad \text{and} \ \ \ 
  \langle u\rangle^{\alpha-1}\leq 2|u-t|^{\alpha-1},
\end{align}
where the last inequality follows from $|u-t|^{1-\alpha}\langle u\rangle^{\alpha-1}\leq |u|^{1-\alpha}\langle u\rangle^{\alpha-1}+1\leq2$. 
\end{rem}

\begin{lem2}
\label{lem.rthlem}
For $(t,h)\in \mT$ let $r_{t,h}$ be a function satisfying the conclusions of Lemma \ref{lem.Kthestimates} for $r,m$ and $\phi$. Assume $1/2<\alpha<1$. Then, there exists a constant $K$ independent of $(t,h)\in \mT$ and $\phi$ such that
\begin{align*}
  \big |r_{t,h}(u)\langle u\rangle^\alpha-r_{t,h}(u')\langle u'\rangle^\alpha \big|
  \leq K\|\phi\|_{H_4^q}h^{1-m-r} \Big| \int_{u'}^u \frac1{(x-t)^{2-\alpha}}+\frac 1{(x-t)^2} dx \Big|,
\end{align*}
for all \ $u,u'\neq t$ and
\begin{align*}
  \TV\big(r_{t,h}\langle \cdot \rangle^\alpha\mathbb{I}_{[t-1,t+1]}\big)&\leq K\|\phi\|_{H_4^q}h^{-m-r}, \quad \\
  \TV\big(r_{t,h}\langle \cdot \rangle^\alpha\mathbb{I}_{\mathbb{R}\setminus [t-1,t+1]}\big)&\leq K\|\phi\|_{H_4^q}h^{1-m-r}.
\end{align*}

\end{lem2}

\begin{proof}
Let $C$ be as in Lemma \ref{lem.Kthestimates}. In this proof $K=K(\alpha,C)$ denotes a generic constant which may change from line to line. Without loss of generality, we may assume that $|u-t|\geq |u'-t|$. Furthermore, the bound is trivial if $u'\leq t\leq u$ or $u\leq t\leq u'$. Therefore, let us assume further that $u\geq u'>t$ (the case $u\leq u'<t$ can be treated similarly). Together with the conclusions from Lemma \ref{lem.Kthestimates} and Remark \ref{rem.ualpha_ids} this shows that
\begin{align*}
  &\big|r_{t,h}(u)
  \langle u \rangle^\alpha
  -r_{t,h}(u')\langle u' \rangle^\alpha\big| \\
  &\leq \big|r_{t,h}(u)\big| \ \big| \langle u \rangle^\alpha- \langle u' \rangle^\alpha\big|
  +  \langle u' \rangle^\alpha \big|r_{t,h}(u)
  -r_{t,h}(u')\big| \\
  &\leq K\|\phi\|_{H_4^q} \Big[h^{2-m-r}\frac 1{(u-t)^2}
  +h^{1-m-r}\frac{|u'-t|^\alpha +1}{|u'-t| \ |u-t|}\Big]|u-u'|.
\end{align*}
Clearly, the second term in the bracket dominates uniformly over $h\in (0,1]$. By Taylor expansion
\begin{align*}
  \frac{|u-u'|}{|u'-t|^{1-\alpha} \ |u-t|}
  &=\frac{u-u'}{(u-t)^\alpha (u'-t)^{1-\alpha} (u-t)^{1-\alpha}} \\
  &\leq \frac{(u-t)^{1-\alpha}-(u'-t)^{1-\alpha}}{(1-\alpha)(u'-t)^{1-\alpha}(u-t)^{1-\alpha}}
  =\int_{u'}^u \frac1{(x-t)^{2-\alpha}} dx.
\end{align*}
Hence,
\begin{align*}
  \frac{1}{|u'-t| \ |u-t|}|u-u'|=\big|\int_{u'}^u \frac 1{(x-t)^2} dx\big|
\end{align*}
completes the proof for the first part. For the second part decompose $r_{t,h}\mathbb{I}_{[t-1,t+1]}$ in $r_{t,h}^{(1)}=r_{t,h}\mathbb{I}_{[t-h,t+h]}$ and $r_{t,h}^{(2)}=r_{t,h}\mathbb{I}_{[t-1,t+1]}-r_{t,h}^{(1)}$. Observe that the conclusions of Lemma \ref{lem.Kthestimates} imply
\begin{align*}
  \TV\big(r_{t,h}^{(1)}\langle \cdot \rangle^\alpha\big)
  &\leq \|\langle \cdot \rangle^\alpha\mathbb{I}_{[t-h,t+h]}\|_\infty
  \TV(r_{t,h}^{(1)})
  +\TV\big(\langle \cdot \rangle^\alpha\mathbb{I}_{[t-h,t+h]}\big)\|r_{t,h}^{(1)}\|_\infty \\
  &\leq K\|\phi\|_{H_4^q} h^{-m-r}.
\end{align*}
By using the first part of the lemma, we conclude that uniformly in $(t,h)\in \mT$,
\begin{align*}
  \TV\big(r_{t,h}\langle\cdot \rangle^\alpha\mathbb{I}_{[t-1,t+1]}\big)
  &\leq \TV\big(r_{t,h}^{(1)}\langle\cdot \rangle^\alpha\big)
  +\TV\big(r_{t,h}^{(2)}\langle\cdot \rangle^\alpha\big) \\
  &\lesssim K\|\phi\|_{H_4^q}( h^{-m-r}+h^{-m-r})
\end{align*}
and also $\TV\big(r_{t,h}\langle \cdot \rangle^\alpha\mathbb{I}_{\mathbb{R}\setminus [t-1,t+1]}\big)\leq K\|\phi\|_{H_4^q} h^{1-m-r}$.
\end{proof}

\begin{lem2}
\label{lem.dthapprox}
Work under Assumptions \ref{ass.noise} and \ref{as.Ffepsilonapprox} and suppose that $m+r>1/2$, $\langle x \rangle \phi \in L^1$, and $\phi \in H_1^{m+r+1}$. Let $d_{t,h}$ be as defined in \eqref{eq.dthdef}. Then, there exists a constant $K$ independent of $(t,h)\in \mT$, such that for $1/2<\alpha <1$,
$$
	\TV(d_{t,h}\langle \cdot \rangle^\alpha \mathbb{I}_{[t-1,t+1]})\leq Kh^{\beta_0\wedge (m+r)-r}\log \big(\tfrac 1h\big).
$$
\end{lem2}

\begin{proof}
For convenience let $\beta_0^\star:=\beta_0\wedge (m+r)$ and substitute $s\mapsto -s$ in \eqref{eq.dthdef}, i.e.
\begin{align*}
    d_{t,h}(u):= \int e^{-is(u-t)/h} \Big(\frac 1{\mathcal{F}(f_\epsilon)(\tfrac sh)}-A\iota_s^{\rho}\big|\tfrac{s}h\big|^r\Big)\iota_s^{\mu}|s|^m \mathcal{F}(\phi)(-s) \rqm s.
\end{align*}
Define
\begin{align*}
  F_h(s):= \frac 1{\mathcal{F}(f_\epsilon)(\tfrac sh)}-A\iota_s^{\rho}\big|\tfrac{s}h\big|^r.
\end{align*}
By Assumptions \ref{ass.noise} and \ref{as.Ffepsilonapprox}, we can bound the $L^1$-norm of 
\begin{align}
 s\mapsto \langle s\rangle F_h(s)\iota_s^{\mu}|s|^m \mathcal{F}(\phi)(-s)
  \label{eq.part1dthapprox}
\end{align}
uniformly in $(t,h)$ by $\int \langle s\rangle \langle \tfrac sh \rangle^{r-\beta_0} |s|^m \big | \mathcal{F}(\phi)(-s) \big| ds$. Bounding $\langle \tfrac sh \rangle^{r-\beta_0}$ by $\langle \tfrac sh \rangle^{r-\beta_0^\star}$ and considering the cases $r\leq \beta_0^\star$ and $r> \beta_0^\star$ separately, we find $h^{\beta_0^\star-r}\int \langle s\rangle ^{1+r+m-\beta_0^\star} |\mathcal{F}(\phi)(-s)| ds
 \lesssim h^{\beta_0^\star-r}\|\phi\|_{H^{r+m+1}}$ as an upper bound for \eqref{eq.part1dthapprox}, uniformly in $(t,h)\in \mathcal{T}$. Furthermore, 
\begin{align*}
  D_s F_h(s) = - \frac{ D_s \mathcal{F}(f_\epsilon)(\tfrac sh)}{\big(\mathcal{F}(f_\epsilon)(\tfrac sh)\big)^2} - Ar i\iota_s^{\rho-1}h^{-1} \big|\tfrac{s}h\big|^{r-1} 
\end{align*}
and by Assumptions \ref{ass.noise} and \ref{as.Ffepsilonapprox}, 
\begin{align*}
  \Big|sD_s F_h(s) \Big|
  &\leq \big|sD_s \mathcal{F}(f_\epsilon)(\tfrac sh) \big|
  \Big|A^2\iota_s^ {2\rho}\big|\tfrac{s}h\big|^{2r}-\frac{1}{\big(\mathcal{F}(f_\epsilon)(\tfrac sh)\big)^2}\Big| \\
  &\quad +
  |A| r \big|\tfrac{s}h\big|^{r} 
  \Big|-A  (ri)^{-1}\iota_s^{\rho+1}h\big|\tfrac{s}h\big|^{r+1} D_s\mathcal{F}(f_\epsilon)\big(\tfrac sh\big)
  - 1\Big| \\
  &\lesssim \Big( \big|\tfrac sh\big|\big\langle \tfrac sh\big\rangle^{r-1}+\big|\tfrac sh\big|^{r}\Big) \big\langle \tfrac sh\big\rangle^{-\beta_0}
  \leq 2\big\langle \tfrac sh\big\rangle^{r-\beta_0^\star}.
\end{align*}
Similarly as above, we can conclude that the $L^1$-norm of 
\begin{align*}
 s\mapsto D_s sF_h(s)\iota_s^{\mu}|s|^m\mathcal{F}(\phi)(-s) 
\end{align*}
is bounded by $\mathrm{const.} \times h^{\beta_0^\star-r}\|\phi\|_{H_1^{r+m+1}}$, uniformly over all $(t,h)\in \mT$. Therefore, we have by interchanging differentiation and integration first and partial integration,
\begin{align*}
  D_u d_{t,h}(u)
  &=\frac {-i}h \int se^{-is(u-t)/h}F_h(s)\iota_s^{\mu}|s|^m \mathcal{F}(\phi)(-s) \rqm s \\
  &=\frac {-1}{u-t}
  \int e^{-is(u-t)/h}D_s sF_h(s)\iota_s^{\mu}|s|^m \mathcal{F}(\phi)(-s) \rqm s
\end{align*}
and the second equality holds for $u\neq t$. Together with \eqref{eq.part1dthapprox} this shows that $|d_{t,h}(u)|\lesssim h^{\beta_0^\star-r}$ and $|D_ud_{t,h}(u)|\lesssim h^{\beta_0^\star-r-1}\min(1,h/|u-t|)$. Using Remark \ref{rem.ualpha_ids} we find for the sets $A_{t,h}^{(1)}:=[t-h,t+h]$ and $A_{t,h}^{(2)}:=[t-1,t+1]\setminus A_{t,h}^{(1)}$,
\begin{align*}
  \TV(d_{t,h}\mathbb{I}_{[t-1,t+1]})
  &\leq 2\|d_{t,h}\|_\infty
  + \int_{A_{t,h}^{(1)}}|D_ud_{t,h}(u) |du
  +\int_{A_{t,h}^{(2)}}|D_ud_{t,h}(u) |du \\
  &\lesssim h^{\beta_0^\star-r}\log\big(\tfrac 1h\big).
\end{align*}
Thus, $\TV(d_{t,h}\langle \cdot \rangle^\alpha \mathbb{I}_{[t-1,t+1]}) \lesssim \|d_{t,h}\|_\infty +\TV(d_{t,h}\mathbb{I}_{[t-1,t+1]})\lesssim h^{\beta_0^\star-r}\log\big(\tfrac 1h\big)$.
\end{proof}

\begin{lem2}
\label{lem.TVvthPlem}
Work under the assumptions of Theorem \ref{thm.speciallimits} and let $v_{t,h}^P$ be defined as in \eqref{eq.vthPdef}. Then, for $1/2<\alpha<1$,
\begin{align*}
  \TV(v_{t,h}^P\langle \cdot \rangle^\alpha \mathbb{I}_{\mathbb{R}\setminus [t-1,t+1]})\leq Kh^{1-r-m},
\end{align*}
where the constant $K$ does not depend on $(t,h)$.
\end{lem2}

\begin{proof}
The proof uses essentially the same arguments as the proof of Lemma \ref{lem.Kthestimates}. Let $q:=\lfloor r+m+5/2\rfloor$ and recall that by assumption $\langle x\rangle^2\phi \in L^1$. Decomposing the $L^1$-norm on $\mathbb{R}$ into $L^1([-1,1])$ and $L^1(\mathbb{R}\setminus[-1,1]),$ using Cauchy-Schwarz inequality, and $\|\mathcal{F}(\phi)\|_\infty \leq \|\phi\|_1$, we see that for $j\in \{0,1\}$, the $L^1$-norm of $s\mapsto D_s^{j} |s|^{r+m}\iota_s^{-\rho-\mu} \mathcal{F}(\phi)(s)$ is bounded by $\mathrm{const.}\times(\|\phi\|_{H_1^q}+\|\phi\|_{1})$. Similarly, for $k\in \{0,1,2\}$ the $L^1$-norms of $s\mapsto D_s^k |s|^{r+m+1}\iota_s^{-\rho-\mu+1}\mathcal{F}(\phi)(s)$ are bounded by a multiple of $\|\phi\|_{H_2^q}+\|\phi\|_{1}$. Hence we have
\begin{align*}
  v_{t,h}^P(u)=\frac{Ah^{1-r-m}ia_P(t)}{u-t}
  \int e^{is(u-t)/h}D_s|s|^{r+m}\iota_s^{-\rho-\mu} \mathcal{F}(\phi)(s) \rqm s
\end{align*}
and 
\begin{align*}
  D_u v_{t,h}^P(u)=\frac{-Ah^{1-r-m}a_P(t)}{(u-t)^2}
  \int e^{is(u-t)/h}D_s^2|s|^{r+m+1}\iota_s^{-\rho-\mu+1} \mathcal{F}(\phi)(s) \rqm s.
\end{align*}
Together with Remark \ref{rem.ualpha_ids} this shows that
\begin{align*}
  \TV\big(v_{t,h}^P\langle \cdot\rangle^\alpha\mathbb{I}_{[t+1,\infty)}\big)
  &\leq \|v_{t,h}^P\langle \cdot\rangle^\alpha\mathbb{I}_{[t+1,\infty)}\|_\infty
  +\int_{t+1}^\infty |D_uv_{t,h}^P(u)\langle \cdot\rangle^\alpha| du \\
  &\lesssim h^{1-r-m}+
  \int_{t+1}^\infty\frac{h^{1-r-m}}{|u-t|^{2-\alpha}}+\frac{h^{1-r-m}}{|u-t|^{2}} du
  \lesssim h^{1-r-m}.
\end{align*}
Similarly we can bound the total variation on $(-\infty, t-1]$.
\end{proof}

The next lemma extends a well-known bound for functions with compact support to general c\`{a}dl\`{a}g functions. We found this result useful for estimating the supremum over a Gaussian process if entropy bounds are difficult.

\begin{lem2}
\label{lem.empprocbd}
Let $(W_t)_{t\in \mathbb{R}}$ denote a two-sided Brownian motion. For a class of real-valued c\`{a}dl\`{a}g functions $\mathcal{F}$ and any $\alpha>1/2$ there exists a constant $C_\alpha$ such that
\begin{align*}
  \sup_{f\in \mathcal{F}} \big|\int f(s) dW_s \big|
  \leq C_\alpha \sup_{s\in [0,1]} |\overline W_s| \sup_{f\in \mathcal{F}} \TV(\langle \cdot \rangle^\alpha f),
\end{align*}
where $\overline W$ is a standard Brownian motion on the same probability space.
\end{lem2}

\begin{proof}
The proof consists of two steps. First suppose that $\bigcup_{f\in \mathcal{F}} \supp f \subset [0,1]$ and assume that the $f$ are of bounded variation. Then, for any $f\in \mathcal{F}$, there exists a function $q_f$ with $\|q_f\|_\infty \le \TV(f)$ and a probability measure $P_f$ with $P_f[0,1[ = 1$, such that $f(u)=\int_{[0,u]} q_f(u) P_f(du)$ for all $u\in \mathbb{R}$, because $f$ is c\`{a}dl\`{a}g and thus $f(1)=0$. With probability one,
\begin{align*}
  \sup_{f\in \mathcal{F}} \big|\int f(s)dW_s\big| = \sup_{f\in \mathcal{F}}
  \Big| \int W_s q_f(s) P_f(ds)\Big|
  \leq \sup_{s\in [0,1]} |W_s| \ \sup_{f\in \mathcal{F}} \TV(f).
\end{align*}

Now let us consider the general case. If $C_\alpha:=\|\langle \cdot\rangle^{-\alpha}\|_2$ then $h(s)= C_\alpha^{-2}\langle s\rangle^{-2\alpha}$ is a density of a random variable. Let $H$ be the corresponding distribution function. Note that
\begin{align*}
  \big(\overline W_t\big)_{t\in [0,1]}
  =
  \Big(\int_0^t \sqrt{h(H^{-1}(s))}dW_{H^{-1}(s)}\Big)_{t\in [0,1]}
\end{align*}
is a standard Brownian motion satisfying $d\overline W_{H(s)}=\sqrt{h(s)}dW_s$ and thus with $Af=\langle \cdot \rangle^\alpha f,$
\begin{align*}
  \sup_{f\in \mathcal{F}} \big|\int f(s)dW_s\big|
  &=C_\alpha \sup_{f\in \mathcal{F}}  \big| \int Af(s) d\overline W_{H(s)} \big| \\
  &=C_\alpha \sup_{f\in \mathcal{F}}  \big| \int_0^1 Af(H^{-1}(s)) d\overline W_{s} \big|.
\end{align*}
Since $\TV(Af\circ H^{-1})= \TV(Af)$ the result follows from the first part.
\end{proof}

In the next lemma, we study monotonicity properties of the calibration weights $w_h$.

\begin{lem2}
\label{lem.delta1_tech}
For $h \in (0,1]$ and $\nu > e$ let $w_h := \sqrt{2^{-1}\log(\nu/h)}/\log \log(\nu/h)$. Then
\begin{itemize}
\item[(i)] $h \mapsto w_h$ is strictly decreasing on $\bigl( 0, \nu \exp(e^{-2}) \bigr]$, \ and
\item[(ii)] $h \mapsto w_h h^{1/2}$ is strictly increasing on $(0,1]$.
\end{itemize}
\end{lem2}

\begin{proof}
With $x = x(h) := \log\log(\nu/h) > 0$, we have $\log w_h = -\log(2)/2 + x/2 - \log x$. Since the derivative of this w.r.t.\ $x$ equals $1/2 - 1/x$ and is strictly positive for $x > 2$, we conclude that $\log w_h$ is strictly increasing for $x(h) \ge 2$, i.e.\ $h \le \nu \exp(e^{-2})$. Moreover, $\log(w_h h^{1/2}) = \log(\nu/2)/2 + x/2 - \log x - e^x/2$, and the derivative of this w.r.t.\ $x > 0$ equals $1/2 - 1/x - e^x/2 < 0$. Thus, $w_h h^{1/2}$ is strictly increasing in $h \in (0,1]$.
\end{proof}

\begin{lem2}
\label{lem.condTVreplace}
Condition (iii) in Assumption \ref{as.testfcts} is fulfilled with $\kappa_n=w_{u_n}u_n^{1/2}$ whenever Condition $(ii)$ of Assumption \ref{as.testfcts} holds, and for all $(t,h)\in B_n$, $\supp \psi_{t,h} \subset [t-h,t+h]$.
\end{lem2}

\begin{proof}
Let $1/2<\alpha<1$. Then $\langle \cdot \rangle^\alpha : \mathbb{R} \rightarrow \mathbb{R}$ is Lipschitz. Recall that $\TV(fg)\leq \|f\|_\infty \TV(g)+\|g\|_\infty \TV(f)$. Since $\bigcup_{(t,h) \in B_n} \supp \psi_{t,h} \subset [-1,2]$ is bounded and contains the support of all functions $s\mapsto \psi_{t,h}(s) \big[\sqrt{g(s)}-\sqrt{g(t)}\big]\langle s \rangle^\alpha$ (indexed in $(t,h)\in B_n$), we obtain uniformly over $(t,h)\in B_n$ and $G\in \mathcal{G}$,
\begin{align*}
  &\TV\Big(\psi_{t,h}(\cdot) \big[\sqrt{g(\cdot)}-\sqrt{g(t)}\big]\langle \cdot \rangle^\alpha\Big) \\
  &\quad \lesssim
  \big\|\psi_{t,h}(\cdot) \big[\sqrt{g(\cdot)}-\sqrt{g(t)}\big]\big\|_\infty
  + \TV\Big(\psi_{t,h}(\cdot) \big[\sqrt{g(\cdot)}-\sqrt{g(t)}\big]\Big)
\end{align*}
Furthermore, 
\begin{align*}
  \TV\Big(\psi_{t,h}(\cdot) \big[\sqrt{g(\cdot)}-\sqrt{g(t)}\big]\Big)
  &\leq \|\psi_{t,h}\|_\infty \TV\big(\big[\sqrt{g(\cdot)}-\sqrt{g(t)}\big]\mathbb{I}_{[t-h,t+h]}(\cdot)\big) \\
  &\quad +\TV\big(\psi_{t,h}\big)\big\|\big[\sqrt{g(\cdot)}-\sqrt{g(t)}\big]\mathbb{I}_{[t-h,t+h]}(\cdot)\big\|_\infty \\
  &\lesssim V_{t,h} h^{1/2},
\end{align*}
where the last inequality follows from Assumption \ref{as.testfcts} (ii) as well as the properties of $\mathcal{G}$. With Lemma \ref{lem.delta1_tech} (ii) the result follows.
\end{proof}

\section{Further results on multiscale statistics}

The following result shows that multiscale statistics computed over sufficiently rich index sets $B_n$ are also bounded from below.

\begin{lem2}
\label{lem.limitlimit}
Assume that $K_n\rightarrow \infty$, $\psi_{t,h}=\psi\big(\tfrac{\cdot-t}h\big)$ and $V_{t,h}=\|\psi_{t,h}\|_2=\sqrt{h} \|\psi\|_2$. Suppose that $\lim_{j\rightarrow \infty} \log(j) |\int \psi(s-j)\psi(s) ds| \rightarrow 0$.
% 
% 
% further that for all $h>0$, 
% \begin{align*}
%   t\mapsto X_h(t):=\frac{\int \psi_{t,h}(s) dW_s}{\|\psi_{t,h}\|_2}
% \end{align*}
% is stationary and there exists a sequence of positive numbers $(c_k)_k$ such that $c_k/\log(k)\rightarrow \infty$ and 
% \begin{align}
%  \sup_{h>0}\ \max_{k\leq \lfloor h^{-1}\rfloor}c_k\big|\Cov\big(X_h(0),X_h(kh)\big)\big|<\infty.
%   \label{eq.cov_decay}
% \end{align}
Then, with $w_h$ and $B_{K_n}^\circ$ as defined in \eqref{eq.whdef_first} and \eqref{eq.defBKncirc}, respectively,
\begin{align*}
  \sup_{(t,h)\in B_{K_n}^\circ}
  w_h \left(
  \frac{\big|\int \psi_{t,h}(s)dW_s \big|}{\|\psi_{t,h}\|_2}-\sqrt{2\log \tfrac \nu h }\right) \rightarrow -\frac 14, \quad \text{in probability.}
\end{align*}
\end{lem2}

\begin{proof}
Write $K:=K_n$ and let $\xi_{j}:=\|\psi_{t,h}\|_2^{-1} \int \psi_{j/K,1/K}(s) dW_s$ for $j=0,\ldots,K-1$. Now, $(\xi_{j})_{j}$ is a stationary sequence of centered and standardized normal random variables. In particular the distribution of $(\xi_{j})_{j}$ does not depend on $K$ and the covariance decays by assumption at a faster rate than logarithmically. By Theorem 4.3.3 (ii) in \cite{lead} the maximum behaves as the maximum of $K$ independent standard normal r.v., i.e.
\begin{align*}
  \P\big(\max(\xi_1,\ldots,\xi_K)\leq a_K+b_K t\big) \rightarrow \exp\big(-e^{-t}\big), \quad \text{for} \ t\in \mathbb{R} \ \text{and} \ K\rightarrow \infty,
\end{align*}
where 
\begin{align*}
  b_K:=\frac 1{\sqrt{2\log K}}, \ \ \ \text{and} \ \ \ 
  a_K=\sqrt{2\log K}-\frac{\log\log K+\log(4\pi)}{\sqrt{8\log K}}.
\end{align*}
Using the tail-equivalence criterion (cf.\ \cite{emb}, Proposition 3.3.28), we obtain further
\begin{align*}
  \lim_{K\rightarrow \infty}\P\big(\max(|\xi_1|,\ldots,|\xi_K|)\leq a_K+b_K (t+\log 2)\big) = \exp\big(-e^{-t}\big), \quad \text{for} \ t\in \mathbb{R}.
\end{align*}
% Let
% \begin{align*}
%   T_n^\circ :=\sup_{(t,h)\in B_n^\circ}w_h\frac{\big|\int \phi^{(m)}\big(\tfrac{s-t}h\big)dW_s\big|}{\sqrt h}-\widetilde w_h.
% \end{align*}
Note that $T_n^\circ:=\sup_{(t,h)\in B_n^\circ}w_h (\|\psi_{t,h}\|_2^{-1}|\int \psi_{t,h}(s)dW_s|-\sqrt{2\log(\nu/h)})$ has the same distribution as $w_{K^{-1}}\max(|\xi_1|,\ldots,|\xi_{K}|)-w_{K^{-1}}\sqrt{2\log(\nu K)}$. It is easy to show that
\begin{align*}
  \sqrt{\log \nu K}=\sqrt{\log K}+\frac{\log \nu}{2\sqrt{\log K}}
  +O\Big(\frac 1{\log^{3/2} K}\Big)
\end{align*}
and
\begin{align*}
  \Big|\frac1{w_{K^{-1}}}-\frac {\log\log K}{\sqrt{\tfrac 12\log K}}\Big|
  =O\left(\frac{\log\log K}{ \log^{3/2} K}\right).
\end{align*}
Assume that $\eta_n \rightarrow 0$ and $\eta_n \log\log K\rightarrow \infty$. Then for sufficiently large $n$,
\begin{align*}
  & \P\big(T_n^\circ>-\tfrac 14+\eta_n\big) \\
  & =
  \P\Big(\max(|\xi_1|,\ldots,|\xi_{K}|)>\big(-\tfrac 14+\eta_n\big)/w_{K^{-1}}
  	+\sqrt{2\log \nu K}\Big) \\
  &  =
  \P\Big(\max(|\xi_1|,\ldots,|\xi_{K}|)> \\
  & \qquad \big(-1+4\eta_n\big)
  \frac{\log\log K}{\sqrt{8\log K}}
  +\sqrt{2\log K}+\frac{\log \nu}{\sqrt{2\log K}}
  +O\big(\frac{\log\log K}{\log^{3/2} K}\big)\Big) \\
  &\leq
  \P\Big(\max(|\xi_1|,\ldots,|\xi_{K}|)>a_K+b_K2\eta_n\log\log K\Big)\rightarrow 0 .
\end{align*}
Similarly,
\begin{align*}
  &\P\big(T_n^\circ\leq -\tfrac 14-\eta_n\big)
  \leq \P\Big(\max(|\xi_1|,\ldots,|\xi_{K}|)\leq a_K-b_K \eta_n\log\log K\Big)\rightarrow 0.
\end{align*}
\end{proof}

%%%%%%%Gehoert noch zum naechsten Lemma

In order to illustrate the general multiscale statistic discussed in Section \ref{sec.genmsc}, let us show in the subsequent example that it is also possible to choose $B_n$ in order to construct (level-dependent) values for simultaneous wavelet thresholding.

\begin{exam}
\label{exam.wav_thresh}
Observe that $\widehat d_{j,k}=T_{k2^{-j},2^{-j}}$ and  $d_{j,k}=\E T_{k2^{-j},2^{-j}}=\int \psi_{k2^{-j},2^{-j}}(s)g(s)ds=\int \psi(2^js-k)g(s)ds$ are the (estimated) wavelet coefficients and if $j_{0n}$ and $j_{1n}$ are integers satisfying $2^{-j_{1n}}n\log^{-3}n\rightarrow \infty$ and $j_{0n}\rightarrow \infty$, then for $\alpha\in (0,1)$ and 
\begin{align*}
 B_n=\big\{(k2^{-j},2^{-j}) \big|  \ k=0,1,\ldots,2^j-1, \ j_{0n}\leq j\leq j_{1n}, \ j \in \mathbb{N} \ \big\},
\end{align*}
Theorem \ref{thm.gmsc} yields in a natural way level-dependent thresholds $q_{j,k}(\alpha)$, such that
\begin{align*}
  \lim_{n\rightarrow \infty}\P\Big(\big|\widehat d_{j,k} -d_{j,k} \big|\leq q_{j,k}(\alpha),
  \ \text{for all} \ j,k, \ \text{with} \ (k2^{-j}, 2^{-j})\in B_n \Big)= 1-\alpha.
\end{align*}
\end{exam}

\end{appendices}

\bibliographystyle{plain}       % (uses file "plain.bst")
\bibliography{refsPart1}           % expects file "refsPart1.bib"

% \bibliographystylesupp{plain}      % (uses file "plain.bst")
% \bibliographysupp{refsPart1}           % expects file "refsPart1.bib"

\end{document}